\documentclass[a4paper]{article}
\usepackage [a4paper,left=3 cm,bottom=2.5cm,right=3cm,top=2.5cm]{geometry}
\usepackage[english]{babel}
\usepackage{amssymb}
\usepackage{amsmath,amsthm}
\usepackage{subcaption}
\usepackage{tabularx,array,ulem}
\usepackage{graphicx, url}
\usepackage{enumitem} 
\usepackage{xcolor}
\usepackage{bbm}
\DeclareMathOperator{\sgn}{sgn}

\newcommand{\E}{\mathbb{E}}
\newcommand{\K}{\mathbb{K}}

\newcommand{\R}{\mathbb{R}}

\newcommand{\Abs}[1]{\left|#1\right|}

\newcommand{\Indi}[1]{\mathbbm{1}_{#1}}

\makeatletter
\newcommand*{\barfix}[2][.175ex]{%
  \mathpalette{\@barfix{#1}}{#2}%
}
\newcommand*{\@barfix}[3]{%
  \vbox{%
    \kern#1\relax
    \hbox{$#2#3\m@th$}%
  }%
}
\makeatother

\RequirePackage{amsthm,amsmath,amsfonts,amssymb}
\RequirePackage[numbers]{natbib}
\RequirePackage[colorlinks,citecolor=blue,urlcolor=blue]{hyperref}
\usepackage{bbm,bm}

\usepackage{cancel}
\usepackage{soul}

\newcommand{\modar}{\color{black}}
\newcommand{\modarn}{\color{black}}
\newcommand{\abs}[1]{\left|#1\right|}
\newcommand{\norm}[1]{\left\lVert#1\right\rVert}
\newcommand{\bracket}[2]{<#1,#2>}
\newcommand{\bracketIto}[2]{{\left[#1,#2\right]}}

\newtheorem{remark}{Remark}
\newtheorem{theorem}{Theorem}
\newtheorem{lemma}{Lemma}

\newtheorem{proposition}{Proposition}
\newtheorem{definition}{Definition}

\newcommand{\keywords}[1]{\par\addvspace\baselineskip
\noindent\enspace\ignorespaces#1}

\newcommand{\revar}{\color{black}}
\newcommand{\revarn}{\color{black}}
\newcommand{\revnotat}{\color{black}}
\usepackage{etoolbox}%

\newcommand{\modc}{\color{black}}
\newcommand{\modchi}{\color{black}}

\newcommand{\gaussian}{\color{black}\mathfrak{g}}

\newcommand{\modrev}{\color{black}}

\newcommand{\modr}{\color{black}}

\begin{document}


\title{Malliavin calculus for the optimal estimation of the invariant density of discretely observed diffusions in intermediate regime.}
\author{Chiara Amorino\thanks{ Universit\'e du Luxembourg, L-4364 Esch-Sur-Alzette, Luxembourg.The author gratefully acknowledges financial support of ERC Consolidator Grant 815703 “STAMFORD: Statistical Methods for High Dimensional Diffusions”.} \qquad Arnaud Gloter \thanks{Laboratoire de Math\'ematiques et Mod\'elisation d'Evry, CNRS, Univ Evry, Universit\'e Paris-Saclay, 91037, Evry, France.}} 

\maketitle

\begin{abstract}
Let $(X_t)_{t \ge 0}$ be solution of a one-dimensional stochastic
	differential equation. {\modarn Assuming that a discrete observation of the process $(X_t)_{t \in [0, T]}$ is available, when $T$ tends to $\infty$, {\modrev our aim is to study the convergence rate for the estimation of the invariant density in cases where the effect of sampling is non-negligible. This scenario is referred to as the 'intermediate regime'.}}	
We find the convergence rates associated to the kernel density estimator we proposed and a condition on the discretization step $\Delta_n$ which plays the role of threshold between the intermediate regime and the continuous case. In intermediate regime the convergence rate is $n^{- \frac{2 \beta}{2 \beta + 1}}$, where $\beta$ is the smoothness of the invariant density. After that, we complement the upper bounds previously found with a lower bound over the set of all the possible estimator, which provides the same convergence rate: it means it is not possible to propose a different estimator which achieves better convergence rates. This is obtained by the two {\modr hypotheses} method; the most challenging part consists in bounding the Hellinger distance between the laws of the two models. The key point is a Malliavin representation for a score function, which allows us to bound the Hellinger distance through a quantity depending on the Malliavin weight. 
\end{abstract}



\section{Introduction}

In this work, we consider the process $(X_t)_{t \ge 0}$, solution to the following stochastic differential equation:
\begin{equation}{\label{eq: model intro}}
X_t= X_0 + \int_0^t b( X_s)ds + \int_0^t a(X_s){\modarn dB_s,}
\end{equation}
where {\modarn $B$} is a one dimensional Brownian motion. Starting from the discrete observation of the process at the times $0 = t_0^n \le \dots \le t_n^n = T$, with $T \rightarrow \infty$, we aim at {\modchi discussing the effect of the discretization of the observations on the optimal rate of convergence for the nonparametric estimation of the invariant density.}

The field of nonparametric statistics for diffusion processes has become
more and more relevant, in statistics. Due to their practical relevance as standard models in many areas of applied science such as genetics, meteorology or financial mathematics, the statistical analysis of stochastic differential equations receives nowadays special attention. \\
Inference for stochastic differential equations (SDEs) based on
the observation of sample paths on a time interval [0, T] has already been widely investigated in several different context.
Moreover, these works have opened the field of inference for more complex stochastic
differential equations such as diffusions with jumps, diffusions with mixed effects or McKean-Vlasov diffuson models; to name a few.

Regarding the issue of nonparametric
invariant density {\revarn estimation for stochastic processes, contributions in this context include \cite{Cas_Lea}, \cite{Bosq98}; more recently \cite{AmoNua} and \cite{DexStraTro} starting from the observation of diffusion with jumps; \cite{Pan} and \cite{ComMar} for diffusions driven by a fractional Brownian motion, \cite{HofMae} for interacting particles system, \cite{MarRos} for iid copies of diffusions as in \eqref{eq: model intro}.}
The nonparametric estimation of the invariant density starting from the observation of a stochastic differential equation {\revarn in an anisotropic context} has recently been studied in \cite{Strauch} and \cite{Minimax} assuming that the continuous record of the process is available. 
Closer to the purpose of this work,
\cite{disc} deals with the estimation of the invariant density from the discrete observation of a stochastic differential equation as in \eqref{eq: model intro}, for $d \ge 2$. In particular, the results in \cite{disc} provide a condition on the discretization step needed in order to recover the continuous convergence rates, (which are the convergence rate achieved in the case where the continuous trajectory of the process is available). Moreover, Theorems 2 and 3 in \cite{disc} state that, in the case where such a condition is not satisfied, the kernel density estimator achieves the convergence rate $(\frac{1}{n})^\frac{2 \bar{\beta}}{2 \bar{\beta} + d}$, where $\bar{\beta}$ is the harmonic mean smoothness of the invariant density we want to estimate and $d$ the dimension of the process. {\revarn Up to our knowledge, few results deals with the optimality in the context of a discrete sampling. A notable exception is \cite{But}, where
the authors 
obtains lower bounds for the
estimation of the invariant density associated to the same process as in \eqref{eq: model intro} starting from the observation of $X_\Delta,\dots , X_{n\Delta}$ for $n \rightarrow \infty$ where  $\Delta=1$ is fixed. However, their approach heavily relies on the fact the discretization step does not go to zero. For example their observation can be replaced by some independent random variable while in our context $\Delta=\Delta_n$ goes to $0$ and so the data are definitively far from being independent. Hence, the technique proposed in \cite{But} can not work in our framework and some other methods need to be introduced. \\}

{\modrev 
	{\revarn In this paper we identify, in dimension $1$, the conditions on the sampling step that delineate between two regime for the estimation rate. Under a regime where the sampling step  goes to $0$ fast enough, it is possible to estimate the stationary distribution with the same rate as in the continuous observation case.}
		We introduce the term 'intermediate regime' to refer to the high-frequency case where the discretization step $\Delta_n$ approaches 0, but not rapidly enough to achieve the same convergence rate as continuous observation. {\revarn Precisely,} 
		we define the 'intermediate regime' as the scenario where the discretization step $\Delta_n := t_{i + 1} - t_i$ is greater than $(\frac{1}{T_n})^\frac{1}{2 \beta}$ (according with Theorem \ref{th: discrete d =1} below). In this context, we observe that the kernel density estimators achieve a convergence rate of $(\frac{1}{n})^\frac{2 {\beta}}{2 {\beta} + 1}$, which aligns with the findings presented in \cite{disc}.} We complement this upper bound with the corresponding lower bound, demonstrating that the infimum over all possible estimators of the pointwise $L^2$ error is always larger than $(\frac{1}{n})^\frac{2 {\beta}}{2 {\beta} + 1}$. This finding implies that the kernel density estimator we propose achieves the best possible convergence rate in the intermediate regime.

{\modrev The lower bound presented in Theorem \ref{th: borne inf discrete} represents the main result of this work and is based on the two {\modr hypotheses} method. In particular, we introduce two models, denoted as $X$ and $\Tilde{X}$, sharing the same drift but with different diffusion coefficients. It is worth noting that constructing the two hypotheses by disturbing the diffusion coefficient, instead of the drift, might seem unusual to readers familiar with lower bound proofs. However, we have chosen this approach because, even in the case of parameter estimation, the drift coefficient is estimated at a rate of $\sqrt{T}$, which hinders the ability to perceive the dependence on the discretization step. As our results confirm, perturbing the diffusion coefficient proves to be the appropriate strategy in this context, enabling the observation of the dependence on the discretization step.} \\
In order to evaluate the total variation distance between these two models we use an interpolation argument which leads us to the introduction of the process $X^\epsilon$ for $\epsilon \in [0, 1]$ (see its definition in \eqref{eq: model epsilon}). Then, the main proof of the lower bound is reduced to the research of a bound for
{\revarn the Fisher information $\E[(\frac{\dot{p}^\epsilon_{\Delta_n}(x_0,X^\epsilon_{\Delta_n})}{p^\epsilon_{\Delta_n}(x_0,X^\epsilon_{\Delta_n})})^2]$,} where $p^\epsilon_{\Delta_n}(x_0,y)$ is the transition density of the process $X^\epsilon$ starting in $x_0$ and arriving in $y$ after a time $\Delta_n$ and $\dot{p}^\epsilon_{\Delta_n}(x_0,y)$ its derivative with respect to $\epsilon$. 

The central aspect of our proof relies on a Malliavin representation for the quantity $\frac{\dot{p}^\epsilon_{\Delta_n}(x_0,y)}{p^\epsilon_{\Delta_n}(x_0,y)}$. This approach mirrors the one used to establish the LAMN (Local Asymptotic Mixed Normality) property of the process. \\
Regarding the literature concerning the LAMN property, it was initially proven for a statistical model of one-dimensional diffusion processes with synchronous, equispaced observations by Dohnal \cite{9 Ogi}. Later, the results were extended to multidimensional diffusions by Gobet \cite{Gobet LAMN}, who utilized a Malliavin calculus approach. Subsequently, in \cite{11 Ogi}, Gobet demonstrated the LAN (Local Asymptotic Normality) property for ergodic diffusion processes as T goes to infinity. This was further extended to the case of nonsynchronously observed diffusion processes in \cite{Ogi}.

In particular, our methodology consists in estimating the local Hellinger distance at time $\Delta_n$ and then conclude by {\modarn tensorization.} Using Malliavin calculus, we can then bound the Hellinger distance by a quantity depending on the Malliavin weight. The approach we propose in this part is close to the one presented in \cite{Clement}. However, in \cite{Clement} the author can bound the Hellinger distance by the $L^2$ norm of the Malliavin weight, while in our case it appears to be not enough. {\modchi A} challenge, in our paper, consists indeed in proposing a sharp bound for the conditional expectation of the Malliavin weight (see Proposition \ref{prop: bound Malliavin} and Lemma \ref{l: main term malliavin} below).

This is achieved by obtaining some occupation formulas and some upper bounds for the conditional first moment for integrals of the local time. \\
The estimation of occupation time functionals is a well-studied topic in the literature: it appears
in the study of numerical approximation schemes for stochastic differential equations (\cite{GobLab08}, \cite{22 Alt}, \cite{25 Alt}) and in
the analysis of statistical methods for stochastic processes (\cite{5 Alt}, \cite{9 Alt} , \cite{17 Alt}). Furthermore, their smoothness properties play an important role for solving ordinary differential equations, for example in combination with the phenomenon of regularization by noise (see for example \cite{Gub}). Some estimations for occupation time functionals of stationary Markov processes can be found in \cite{AltCho}, while \cite{Alt} applies also to non-Markovian processes. 

 Denoting as $\Tilde{B}$ a Brownian motion and as $(L_t^z(\Tilde{B}))_t$ the local time at level $z$ of the Brownian motion $\Tilde{B}$, we show  in Lemma \ref{L: majo func Local time cond Delta} some upper bounds of {\modrev $\E_y[\int_0^{\Delta_n} \frac{d L_s^z(\Tilde{B})}{s^\gamma} | \Tilde{B}_{\Delta_n} = x]$ for $\gamma \ge 0$, where $\E_y[\cdot]$ is the conditional expectation given $\Tilde{B}_0 = y$. In particular, we deduce some controls which extend the results in \cite{Gradinaru_et_al_1999}, where this quantity has been studied in detail for $\gamma \le \frac{1}{2}$. We remark that, for $\gamma= 0$, the conditional expectation above turns out being the expectation of the local time for the Brownian bridge, which has been intensively studied in \cite{Pitman}.}\\
{\modr Let us introduce the notation $\E_{(a,b)}$ for the expectation under the law of $(X_t)_{t\in[0,T_n]}$,  stationary solution of \eqref{eq: model intro}. Then, we} are able to prove that,
when {\modrev $\Delta_n > (\frac{1}{T_n})^\frac{1}{2 \beta}$}, the following lower bound holds true:
\begin{equation}
\inf_{\tilde{\pi}_{T_n}} \sup_{(a,b) \in \Sigma} {\modr \E_{(a,b)}}[(\tilde{\pi}_{T_n}(x^*) - \pi(x^*))^2] \ge c (\frac{1}{n})^{\frac{2 \beta}{2 \beta + 1}},
\label{eq: lower intro}
\end{equation}
where the infimum is taken {\modrev over all estimators} of the invariant density based on  $X_{0}, X_{\Delta_n},\dots, X_{n \Delta_n}$. {\modr By estimator we mean any real-valued random variable given as a measurable function of $X_{0}$, $X_{\Delta_n}$, \dots , $X_{n \Delta_n}$.}
Here above $\Sigma$ is a class of coefficients for which the stationary density has some prescribed regularity and $\beta$ is the smoothness of the invariant density $\pi$. \\
The lower bound in \eqref{eq: lower intro} complements the upper bounds we show in our Theorem \ref{th: discrete d =1}:
{\modrev \begin{equation}
		\label{eq: borne sup intro}
\sup_{(a,b)\in \Sigma} {\modr \E_{(a,b)}} [|\hat{\pi}_{h,n}(x^*) - \pi (x^*)|^2] \underset{\sim}{<}
\begin{cases}
 \frac{1}{T_n} \qquad \mbox{if } \Delta_n \le (\frac{1}{T_n})^\frac{1}{2 \beta}, \\
n^{- \frac{2{\beta}}{2{\beta}+ 1}}  \qquad \mbox{if } \Delta_n > (\frac{1}{T_n})^\frac{1}{2 \beta},
\end{cases}
\end{equation} }
where $\hat{\pi}_{h,n}(x)$ is the kernel density {\revarn estimator.} 
\\
{\modrev We observe that the convergence rate recovered above for $\Delta_n \le (\frac{1}{T_n})^\frac{1}{2 \beta}$ is the superoptimal rate $\frac{1}{T_n}$, which is the optimal rate when the continuous trajectory of the process is available and has already been deeply studied in the literature (see for example \cite{Kut_2004}). Here we also study in detail what happens in the intermediate regime, which is completely new.} \\
We remark that the convergence rate we found in the intermediate regime is the same as for the estimation of a probability
density belonging to an H\"older class, associated to $n$ iid random variables $X_1, \dots , X_n$. 
{ To summarize our finding, the results \eqref{eq: lower intro}--\eqref{eq: borne sup intro} show that the optimal rate of estimation for $\pi(x^*)$ is the slowest rate between the super-efficient rate $1/T_n$ and the classical non parametric one $n^{-\frac{2\beta}{1+2\beta}}$.  The condition $\Delta_n = \left(1 / T_n\right)^{1 /(2 \beta)}$ is the critical value for which these two rates are equal, and defines the frontier between the two regimes. \\
{\modr While our findings are confined to dimension $1$, it prompts curiosity about the potential existence of a similar dichotomy in higher dimensions. A comparison with Theorems 2 and 3 in \cite{disc} reveals a similar dichotomy in the upper bounds, even though the critical values exhibit different structures depending on the dimension. On the contrary, establishing a lower bound becomes notably more intricate in higher dimensions. In Remark \ref{R: extension}, we discuss the additional challenges that must be addressed to extend the lower bound established in Theorem \ref{th: borne inf discrete} to higher dimensions.}}

The outline of the paper is the following. In Section \ref{s: model} we introduce the model and we list the assumptions we will need in the sequel, while Section \ref{S: results} is devoted to the construction of the estimator and the statement of our main results.
In {\revarn Sections \ref{Ss : proof var upper bound}--\ref{Ss : proof Th upper bound}, we give the proof of the upper bound \eqref{eq: borne sup intro}, while in Section \ref{Ss : proof lower bound}, we construct the two {\modr hypotheses} setting and deduce the lower bound \eqref{eq: lower intro}. The Section \ref{S: study of Malliavin weight} is devoted to the proof of the main control on the Malliavin weight
used in the representation of the Fisher information.}
The proof of the technical results is delegated to the {\revar Appendix \ref{s: technical}, while Appendix \ref{S: Appendix Malliavin} is devoted to an introduction of Malliavin calculus, presenting some helpful tools used along the manuscript.}

\section{Model Assumptions}{\label{s: model}}

We aim at proposing a non-parametric estimator for the invariant density associated to a mono-dimensional diffusion process $X$.
The diffusion is a strong solution of the following stochastic differential equation:
\begin{equation}
X_t= X_0 + \int_0^t b( X_s)ds + \int_0^t a(X_s){\modarn dB_s,} \quad t \in [0,T], 
\label{eq: model}
\end{equation}
where $b : \mathbb{R} \rightarrow \mathbb{R}$, $a : \mathbb{R} \rightarrow \mathbb{R}$ and {\modarn $B = (B_t, t \ge 0)$} is a standard Brownian motion. The initial condition $X_0$ and {\modarn $B$} are independent. 

{\revarn 
	\textbf{A1}: \textit{The functions $b(x)$ and $a(x)$ are globally Lipschitz functions of class $\mathcal{C}^{3}$. For all $x\in\mathbb{R}$ the following hold true for the functions $a$, $b$ and their first three derivatives :
		\begin{equation*} |b(0)|\le b_0,~ |a(x)|\le a_0,~| b^{(l)}(x)|\le b_l, ~ |a^{(l)}(x)|\le a_l, \text{ with $l=1,2,3$,}
		\end{equation*}		
	  for some positive constants $(a_l)_{0\le l \le 3}$, $(b_l)_{0\le l \le 3}$.
		%
		Furthermore, for some $a_{\text{min}} > 0$,
		$$a_{\text{min}}^2 \le a^2(x). $$
}}
\textbf{A2 (Drift condition) }: \textit{ \\
	There exist {\revarn $\tilde{C} > 0$} and {\revarn $\tilde{\rho} > 0$} such that $x b(x), \le {\revarn -\tilde{C} |x|}$, $\forall x : |x| \ge {\revarn \tilde{\rho}}$.
} 
\\
Under the assumptions A1 - A2 the process $X$ admits a unique invariant distribution $\mu$ and the ergodic theorem holds {\revarn (see Theorem 1.4 in \cite{Kut})}. We suppose that the invariant probability measure $\mu$ of $X$ is absolutely continuous with respect to the Lebesgue measure and from now on we will denote its density as $\pi$: $ d\mu = \pi dx$. 
\\
We want to estimated the invariant density $\pi$ belonging to the H{\"o}lder class $\mathcal{H} (\beta, \mathcal{L})$ defined below.
\begin{definition}
Let $\beta > 0$, $\mathcal{L} > 0$. A function $g : \mathbb{R} \rightarrow \mathbb{R}$ is said to belong to the {\revar H{\"o}lder} class $\mathcal{H} (\beta, \mathcal{L})$ of functions if,
$$\left \|  g^{(k)} \right \|_\infty \le \mathcal{L} \qquad \forall k = 0,1, \dots , \lfloor \beta \rfloor, $$
$$\left \|  g^{(\lfloor \beta \rfloor)}(. + t) - g^{(\lfloor \beta \rfloor)}(.) \right \|_\infty \le \mathcal{L} |t|^{\beta - \lfloor \beta \rfloor} \qquad \forall t \in \mathbb{R},$$
for $g^{(k)}$ denoting the $k$-th order {\revar derivative} of $g$ and $\lfloor \beta \rfloor$ denoting the largest integer strictly smaller than $\beta$.
\end{definition}

	This leads us to consider a class of coefficients $(a,b)$ for which the stationary density $\pi=\pi_{(a,b)}$ has some prescribed H\"older regularity.
%
{\revarn 
 	\begin{definition}
	Let $\beta > 0$, $\mathcal{L} > 0$, $(a_l)_{0\le l \le 3}\in(0,\infty)^4$, $(b_l)_{0\le l \le 3}\in(0,\infty)^4$, $0<a_\text{min}< a_0$, 
	 $\tilde{C}>0, ~\tilde{\rho}>0$.
	
	We define $\Sigma (\beta, \mathcal{L},a_{\text{min}},(a_l)_{0\le l \le 3}, (b_l)_{0\le l \le 3},\tilde{C},\tilde{\rho})$ the set of couple of functions $(a,b)$ where  $a: \mathbb{R} \rightarrow \mathbb{R}$ and $b: \mathbb{R} \rightarrow \mathbb{R}$ are such that 
	\begin{itemize}
		\item the coefficients $a$ and $b$ satisfy A1 with the constants $a_\text{min}, (a_l)_{0\le l \le 3}, (b_l)_{0\le l \le 3}$,
		\item 	$b$ satisfies A2 with the constants $(\tilde{C},\tilde{\rho})$,
		\item  the density $\pi_{(a,b)}$ of the invariant measure associated to the stochastic differential equation \eqref{eq: model} belongs to $\mathcal{H} (\beta,  \mathcal{L})$.
	\end{itemize}
	\label{def: insieme sigma v2}
\end{definition}
}
{\revarn We know from Theorem 1.4 in \cite{Kut} that the stationary density $\pi=\pi_{(a,b)}$ is explicit as a function of the coefficients of the one dimensional {\modr SDE}, and in consequence the conditions in Definition \ref{def: insieme sigma v2} are not independent. In particular, if the coefficients $a$ and $b$ of the diffusion are $\mathcal{C}^k$, then the stationary density is also of class $\mathcal{C}^k$. Thus, for a diffusion satisfying A1--A2, the stationary density must be at least $\mathcal{C}^3$, and we can check that its derivatives up to order three are bounded. Hence, $\pi$ is at least of class $\mathcal{H}(3,\mathcal{L})$ for some $\mathcal{L}>0$. 
	
	Let us emphasize that the status of the conditions in Definition \ref{def: insieme sigma v2} are different. Condition (A1) is a technical smoothness condition on the model, and we require up to order three regularity for the coefficients in order to apply results in \cite{MenPes} {\modr on} the transition density of the process (see Theorem \ref{thm : transition density} below). Condition (A2) is a classical mean reverting condition on the drift, sufficient to get the existence of a stationary probability. The last condition on the H\"older smoothness of $\pi$ is a standard assumption {\modr on} the regularity of the non-parametric function which is estimated.}

{\modrev
In the subsequent discussion, our analysis heavily relies on well-established results pertaining to the transition density $p_t(x,y)$ of diffusion processes. To ensure clarity, we deem it essential to explicitly outline these 
{\revarn results}. The bound on the transition density in the presence of an unbounded drift can be derived from \cite{MenPes}. Notably, when contrasting this bound with the known results for the transition density under a bounded drift, the primary distinction arises from substituting the initial point with the flow of the initial point.
For a fixed $(s, x) \in \mathbb{R}^+ \times \mathbb{R}^d$, we denote by $\theta_{t,s}(x)$ the deterministic flow that satisfies the differential equation:
$$\dot{\theta}_{t,s}(x) = b(\theta_{t,s}(x)), \quad t \ge 0, \qquad \theta_{s,s}(x) = x.$$
With this in mind, {\revarn we recall the first point of Theorem 1.2 of \cite{MenPes} and the results of Section 4 of the same reference, in the theorem below}}. 
{\revarn 
\begin{theorem}[\cite{MenPes}]\label{thm : transition density}
Under A1-A2, for any {\revarn $\tau > 0$,} $(s, t) \in [0, \infty)^2$, {\revarn $0 < t-s < \tau$}, the unique weak solution of \eqref{eq: model} admits a {\revarn transition density} $p_{t - s}(x,y)$ which is continuous in {\revarn $x, y \in \mathbb{R}$.} Moreover, there exist $\lambda_0 \in (0, 1]$ and ${C_0}, c_0 \ge 1$ such that, for any $(s,t)$, $0< t-s < {\revarn \tau}$ and $x,y \in {\revarn \mathbb{R}}$ it is  
   \begin{equation*}
   	p_{t-s}(x,y)\le {C_0} {\revarn (t-s)^{- \frac{1}{2}}} e^{ - {\lambda_0} \frac{|\theta_{t,s}(x) - y|^2}{t-s} } 
   \end{equation*}
and the constants depend only on {\revarn $\tau ,~a_\text{min},~b_0,~a_0,~ a_1$ and $b_1$.} 

Moreover, for $k=1,2$, we have the control on the derivatives,
\begin{equation*}
	\abs{\frac{\partial^k p_{t-s}(x,y)}{\partial y^k}}\le {C_0} {\revarn (t-s)^{- \frac{1+k}{2}}} e^{ - {\lambda_0} \frac{|\theta_{t,s}(x) - y|^2}{t-s} } 
\end{equation*} 
where the constants $C_0$ and $\lambda_0$ depends on $\tau$, $a_{\text{min}}$, $a_l$, $b_l$ for $l=0,\dots,3$. 
\end{theorem}
}

In the next section we will propose an estimator for the estimation
of the invariant density $\pi$ starting from a discrete observation of the process $X$. In particular, we want to find the convergence rates in intermediary regime, i.e. when the discretization step goes to zero but the associated error is not negligible.

\section{Construction estimator and main results}\label{S: results}

We suppose that we observe a finite sample $X_{t_0}, \dots, X_{t_n}$, with $0= t_0 < t_1 < \dots < t_n=: T_n$. The process $X$ is solution of the stochastic differential equation \eqref{eq: model}. Every observation time point depends also on $n$ but, in order to simplify the notation, we suppress this index. We assume the discretization scheme to be uniform which means that, for any $i \in \{ 0, \dots, n-1 \}$, it is $t_{i + 1} - t_i = : \Delta_n$. We will be working in a high-frequency setting i.e. the discretization step $\Delta_n \rightarrow 0$ for $n \rightarrow \infty$. We assume moreover that $T_n = n \Delta_n \rightarrow \infty$ for $n \rightarrow \infty$.  

It is natural to estimate the invariant density $\pi \in \mathcal{H} (\beta, \mathcal{L})$ by means of a kernel estimator.
We therefore introduce some kernel function $K: \mathbb{R} \rightarrow \mathbb{R}$ satisfying 
\begin{equation}
\int_\mathbb{R} K(x) dx = 1, \quad \left \| K \right \|_\infty < \infty, \quad \mbox{supp}(K) \subset [-1, 1], \quad \int_\mathbb{R} K(x) x^l dx = 0,
\label{eq: properties K}
\end{equation}
for all $l \in \left \{ 0, \dots , M \right \}$ with $M \ge  \beta$. \\
When the continuous trajectory of the process is available the convergence rates for the estimation of the invariant density are known. 
In particular, in the mono-dimensional case, it is known that {\modr in our framework} the proposed kernel estimator achieves the parametric rate $\frac{1}{T}$ and such a rate is optimal (see for example \cite{Kut_2004} or Theorem 1 in \cite{Kut_1998}).

We propose to estimate the invariant density $\pi \in \mathcal{H} (\beta, \mathcal{L})$ associated to the process $X$, solution to \eqref{eq: model}, disposing only of the discrete observation of the process. To do that, we propose the following kernel estimator: for $x^*\in \R$ we define 
\begin{align}\label{eq: def discrete estimator} 
\hat{\pi}_{h, n} (x^*) & := \frac{1}{n \Delta_n} \frac{1}{ h} \sum_{i = 0}^{n-1} K(\frac{x^* - X_{t_i}}{h}) (t_{i + 1} - t_i) \\
&= \frac{1}{n} \sum_{i = 0}^{n-1} \mathbb{K}_h(x^* - X_{t_i}), \nonumber
\end{align}
with $K$ a kernel function as in \eqref{eq: properties K}.

The asymptotic behaviour of the estimator proposed in \eqref{eq: def discrete estimator} is based on the bias-variance decomposition. To find the convergence rates it achieves we need a bound on the variance, {\modrev which heavily relies on Castellana and Leadbetter condition, as introduced in \cite{Cas_Lea}. \\
The work of Castellana and Leadbetter in \cite{Cas_Lea} establishes that, subject to the condition \textbf{CL} described below, the density can be estimated using non-parametric estimators (including kernel estimators) with a parametric rate of $\frac{1}{T}$.\\
To introduce the condition \textbf{CL}, we require that the process $X$ belongs to a class of real processes with a common marginal density $\pi$ with respect to the Lebesgue measure on $\mathbb{R}$. Furthermore, we assume that the joint density of $(X_s, X_t)$ exists for all $s \neq t$, is measurable, and satisfies $\pi_{(X_s, X_t)} = \pi_{(X_t, X_s)} = \pi_{(X_0, X_{t-s})}$. This joint density is denoted by $\pi_{|t-s|}$ for all $s, t \in \mathbb{R}$. Additionally, we define the function $g_u$ as $g_u(x,y) = \pi_u(x, y) - \pi(x) \pi(y)$. The condition \textbf{CL} can be stated as follows:
\begin{itemize}
\item[\textbf{CL:}] $u \mapsto \left \| g_u \right \|_\infty$ is integrable on $(0, \infty)$ and $g_u(\cdot, \cdot)$ is continuous for each $u > 0$. 
\end{itemize}
It is important to remark that in our context it is $\pi_u(x, y)= \pi (x) p_u(x, y)$, where we recall that $p_u(x,y)$ is the transition density. Condition \textbf{CL} can be fulfilled by ergodic continuous diffusion processes (see \cite{Ver99} for sufficient conditions). In \cite{Comte_Mer}, the authors developed a projection estimator and established that its $L^2$-integrated risk achieves the parametric rate of $\frac{1}{T}$, but under a weaker condition known as \textbf{WCL}.
\begin{itemize}
\item[\textbf{WCL:}] There exists a positive integrable function $k$ (defined on $\mathbb{R}$) such that
$$\sup_{y \in \mathbb{R}} \int_0^\infty \vert g_u(x, y) \vert du \le k(x), \qquad \text{ for all } x \in \mathbb{R}.$$
\end{itemize}
{\modr The sufficient conditions can be separated into two components: a local irregularity condition referred to as \textbf{WCL1}, and an asymptotic independence condition referred to as \textbf{WCL2}.} These conditions require the existence of two positive integrable functions $k_1$ and $k_2$ defined on $\mathbb{R}$, as well as a positive constant $u_0$, satisfying the following conditions:
\begin{align*}
&\mbox{\bf WCL1: } \sup_{y \in \mathbb{R}} \int_0^{u_0} |g_u(x,y)|\, du < k_1(x), \qquad \text{for all } x \in \mathbb{R},\\
&\mbox{\bf WCL2: } \sup_{y \in \mathbb{R}} \int_{u_0}^\infty  |g_u(x,y)| \, du < k_2(x), \qquad \text{for all } x \in \mathbb{R}.
\end{align*}
In this paper, which primarily focuses on estimation based on discrete observations, we introduce analogous conditions in the discrete framework. We refer to these conditions as \textbf{WDCL1} and \textbf{WDCL2}, where the additional 'D' denotes 'discrete'. There exist two {\revarn positive 
	functions $k_1$} and $k_2$ on $\R$, as well as a positive constant $u_0$, such that
{\revarn for all $n \ge 1$,}
\begin{align*}
	&\mbox{\bf WDCL1: } \sup_{y \in \mathbb{R}} \, {\revarn \Delta_n} \sum_{i = 1}^{\tilde{i}} |g_{t_i}(x,y)|< k_1(x), \qquad \text{for all } x \in \mathbb{R}, \\
	&\mbox{\bf WDCL2: } \sup_{y \in \mathbb{R}} \, {\revarn \Delta_n} \sum_{i = {\tilde{i}} + 1}^{\revarn n-1} |g_{t_i}(x,y)|< k_2(x), \qquad \text{for all } x \in \mathbb{R},
\end{align*}
where we have introduced {\revarn $\tilde{i} : = \left\lfloor \frac{u_0}{\Delta_n} \right\rfloor \wedge (n-1)$, which is 
such that $\tilde{i} := \sup \left \{ i \in [1, n-1] \, \mbox{ such that } i \Delta_n \le u_0 \right \}$ for $n$ large enough.}
\\
In order to show these conditions hold, the main tools consists in the bounds on the transition density $p_t(x,y)$ and its derivatives {\revarn  obtained in \cite{MenPes} and recalled in Theorem  \ref{thm : transition density}. It leads us to the following bound on the variance.}

%
{\revarn
\begin{proposition}
		Let $\beta > 0$, $\mathcal{L} > 0$, $(a_l)_{0\le l \le 3}\in(0,\infty)^4$, $(b_l)_{0\le l \le 3}\in(0,\infty)^4$, $0<a_\text{min}< a_0$, 
	$\tilde{C}>0, ~\tilde{\rho}>0$ and denote by $\Sigma=\Sigma (\beta, \mathcal{L},a_{\text{min}},(a_l)_{0\le l \le 3}, (b_l)_{0\le l \le 3},\tilde{C},\tilde{\rho})$ the set of coefficients $(a,b)$ introduced in Definition \ref{def: insieme sigma v2}. 
	We assume that $X$ is a stationary solution of \eqref{eq: model} and let $\hat{\pi}_{h,n}$ be the estimator proposed in \eqref{eq: def discrete estimator}. Then,	 there exist $c > 0$ and $T_0 > 0$ such that, for $T_n \ge T_0$,
	\begin{equation*}
		Var(\hat{\pi}_{h,n}(x^*)) \le \frac{c}{T_n}+ \frac{c}{T_n} \frac{\Delta_n}{h}.
	\end{equation*}
	Moreover, the constants $c, T_0$ are uniform over the set of coefficients $(a,b) \in \Sigma$ and $x^* \in \mathbb{R}$.
	\label{prop: bound var discrete d=1}
	\end{proposition}
}
{\revarn
	We deduce the following result on the risk of the estimator $\hat{\pi}_{h,n}$.
\begin{theorem}
		Let $\beta > 0$, $\mathcal{L} > 0$, $(a_l)_{0\le l \le 3}\in(0,\infty)^4$, $(b_l)_{0\le l \le 3}\in(0,\infty)^4$, $0<a_\text{min}< a_0$, 
	$\tilde{C}>0, ~\tilde{\rho}>0$ and denote  $\Sigma=\Sigma (\beta, \mathcal{L},a_{\text{min}},(a_l)_{0\le l \le 3}, (b_l)_{0\le l \le 3},\tilde{C},\tilde{\rho})$. 
Then, there exist  $c > 0$ and $T_0 > 0$ such that
for $T_n \ge T_0$, the following hold true.
\begin{itemize}
    \item[$\bullet$] If {\modrev $\Delta_n \underset{\sim}{<} (\frac{1}{T_n})^{\frac{1}{2 \beta}}$}, then there exists a sequence $(h_n)_n$ such that 
    $$\sup_{(a,b) \in \Sigma} \mathbb{E}_ {(a,b)}[|\hat{\pi}_{h_n,n}(x^*) - \pi (x^*)|^2] \le \frac{c}{T_n}.$$
    \item[$\bullet$] If otherwise {\modrev $\Delta_n \underset{\sim}{>} (\frac{1}{T_n})^{\frac{1}{2 \beta}} $}, then then there exists a sequence $(h_n)_n$ such that
    $$\sup_{(a,b) \in \Sigma} \mathbb{E}_ {(a,b)}[|\hat{\pi}_{h_n,n}(x^*) - \pi (x^*)|^2] \le c (\frac{1}{n})^{\frac{2 \beta}{2 \beta + 1}}.$$
\end{itemize}
We use the notation $\E_{(a,b)}$ to emphasize that we compute the expectation under the law of $(X_t)_{t\in[0,T_n]}$  stationary solution of \eqref{eq: model}.
\label{th: discrete d =1}
\end{theorem}
}
{\modrev 
	\begin{remark}
	The conditions $\Delta_n \underset{\sim}{<} (\frac{1}{T_n})^{\frac{1}{2 \beta}}$ and $\Delta_n \underset{\sim}{>} (\frac{1}{T_n})^{\frac{1}{2 \beta}}$ {\modr separate} the two different regimes under consideration. The first corresponds to the case where it is possible to recover the continuous convergence rate, while the second is what we refer to as the intermediate regime. When $\Delta_n$ equals the threshold $(\frac{1}{T_n})^{\frac{1}{2 \beta}}$, and knowing that $T_n = n \Delta_n$, we find that $\Delta_n = n^{- \frac{1}{2 \beta + 1}}$. Substituting this value, we discover that $\frac{1}{T_n} = (\frac{1}{n})^{\frac{2 \beta}{2 \beta + 1}}$. Consequently, there is no effective difference between the two cases in Theorem \ref{th: discrete d =1} when the discretization step is equal to the critical value $(\frac{1}{T_n})^{\frac{1}{2 \beta}}$.\\
From Theorem \ref{th: discrete d =1} it follows that, in the intermediate regime, the convergence rate achieved by the proposed estimator is $(\frac{1}{n})^{\frac{2 \beta}{2 \beta + d}}$, where $d$ is the dimension and it is here equal to $1$. It has been shown in Theorems 2 and 3 of \cite{disc} that the convergence rates in the intermediate regime are the same also in higher dimension, up to replacing $\beta$ with $\bar{\beta}$, the harmonic mean of the smoothness over the $d$ different directions.
It is interesting to remark that it is also the convergence rate for the estimation of a probability density belonging to an {\modr H\"older} class, associated to $n$ iid random variables $X_1, \dots, X_n$.
\end{remark}
	\begin{remark}
		It can be seen in the proof of Theorem \ref{th: discrete d =1}, that the optimal bandwidths $(h_n)_n$ depend on the unknown smoothness degree $\beta$. Moreover, the condition which {\modr separates} the two regimes depends on $\beta$ as well.		 
		 Consequently, it can be worthwhile to propose an adaptive procedure, akin to the one initially introduced by Goldenslugher and Lepski in \cite{GL}, which allows to choose the bandwidth using only the data without the prior knowledge of $\beta$. {\modr This aspect has been analyzed in a context close to ours in \cite{Minimax}, which studies the same model as in \eqref{eq: model} but assumes that the continuous observation of the process is available. The analogous procedure, in the case where only a discrete sampling of the process is available, has not yet been considered and is left for future investigation.}
\end{remark}
\begin{remark} {\modr Observe that we know the rate $1/T_n$ is optimal in a minimax sense starting from a discrete sampling. This is a consequence of Theorem 4.3 in \cite{Kut}, which establishes that the convergence rate $1/T$ is optimal in a minimax sense for the estimation of $\pi$ starting from the observation of the continuous trajectory of the process.} Thus, in the case 
	$\Delta_n \underset{\sim}{<} (\frac{1}{T_n})^{\frac{1}{2 \beta}}$, our estimator based on a discrete sampling is optimal {\modr in a minimax sense}.
	In the intermediate regime, we can also prove that the convergence rate is optimal, as demonstrated in Theorem \ref{th: borne inf discrete} below. 
\end{remark}
}}
{\modrev The following theorem demonstrates the optimality of the convergence rate identified in the intermediate regime. Its proof relies on Malliavin calculus, which is extensively detailed in Appendix \ref{S: Appendix Malliavin}. Particularly, the crucial element for our result is the Malliavin representation of a score function through a Malliavin weight {\revarn (refer to Section
	\ref{Ss: proof of prop sans Malliavin} 
 and Appendix \ref{Ss: proof of score by Malliavin}
	below for more information).}
}
{\revarn
\begin{theorem}
	Let  $\beta \ge  3$ , $\mathcal{L} > 0$, $(a_l)_{0\le l \le 3}\in(0,\infty)^4$, $(b_l)_{0\le l \le 3}\in(0,\infty)^4$, $0<a_\text{min}< a_0$
	{\modr $\tilde{C}>0$.} Assume also that {\revar 
			$(\frac{1}{T_n})^{\frac{1}{2 \beta}} \le \Delta_n $, $\forall n \ge 1$, and $\Delta_n = O(n^{-\varepsilon})$, for some $\varepsilon>0$}. 
	{\revar 
		Then, there exist {\modr $\tilde{\rho}$, $c > 0$} and $T_0 > 0$ such that, for $T_n \ge T_0$,}
	$$\inf_{\tilde{\pi}_{T_n}} \sup_{(a,b) \in \Sigma} \E_{(a,b)}[(\tilde{\pi}_{T_n}(x^*) - \pi(x^*))^2] \ge c (\frac{1}{n})^{\frac{2 \beta}{2 \beta + 1}},$$
	where $\Sigma=\Sigma (\beta, \mathcal{L},a_{\text{min}},(a_l)_{0\le l \le 3}, (b_l)_{0\le l \le 3},\tilde{C},\tilde{\rho})$ and the infimum is taken {\revar over all estimators} of the invariant density at point $x^*$. 
	\label{th: borne inf discrete}
\end{theorem}
}
{\revar \begin{remark}
	Contrary to the upper bound, the lower bound is not stated over any set $\Sigma$ potentially given by Definition \ref{def: insieme sigma v2}, but only with some sufficiently {\modr large value for $\tilde{\rho}$.} Actually, it is impossible to get the lower bound for all values {\modr $\tilde{\rho}$} as the mean reverting condition A2 could conflict with the upper bounds on the drift defined by the constants $b_0$, $b_1$ appearing in A1. In such case, the set $\Sigma$ is empty, and the lower bound cannot hold true.
\end{remark}}
Theorem \ref{th: borne inf discrete} here above implies that, on a {\revar sufficiently large} class of diffusion $X$ discretely observed (with a uniform discretization step), whose invariant density belongs to ${\modarn \mathcal{H}}(\beta,{\revarn  \mathcal{L}})$, it is not possible to find an estimator with rate of estimation better than $(\frac{1}{n})^{\frac{2 \beta}{2 \beta + 1}}$. \\
Comparing the results here above with the upper bound in Theorem \ref{th: discrete d =1} we observe that the convergence rate we found in the lower bound and in the upper bound in the intermediate regime are the same. It follows that the estimator $\hat{\pi}_{h,n}$ we proposed in \eqref{eq: def discrete estimator} achieves the best possible convergence rate. {\modr On the other side, however, it is worth underlining that the class $\Sigma$ considered in Theorem \ref{th: borne inf discrete} is more restricted than the classes allowed in the upper bounds.}\\
\\
\begin{remark}\label{R: extension}
{\modrev One may wonder if it possible to extend the lower bound gathered in Theorem \ref{th: borne inf discrete} in higher dimension, proving in this way that the convergence rate in intermediate regime found in \cite{disc} is optimal for any $d$. On one side, the proof of Theorem \ref{th: borne inf discrete} relies on Malliavin calculus, which can be easily extended for $d > 1$. On the other side, we need some controls on the local time to bound the main term coming from the Malliavin weight and this does not allow us to move to higher dimension. In particular, the result in Lemma \ref{l: main term malliavin} below holds true only for $d = 1$, and less sharp estimations on the conditional expectation in the left hand side of Lemma \ref{l: main term malliavin} are not enough to recover the wanted convergence rate. An idea to overcome the problem and to obtain similar controls in higher dimension could be to 
		use {\revar the solution of the Poisson equation associated to the generator of the diffusion,}
		in a similar way as in Lemma 1 of \cite{Richard}. This will be object of further investigation. 
}
\end{remark}

\section{Proofs}{\label{S: proofs}}

This section is devoted to the proof of our main results. {\revarn Remark, that in the proofs, the constant $c$ may change from line to line, but remains uniform on the class of models with diffusion and drift coefficients in $\Sigma$.}

\subsection{Proof of Proposition \ref{prop: bound var discrete d=1}}\label{Ss : proof var upper bound}
\begin{proof}
{\modrev
	{\revarn We start by {\modr expanding} the variance term in the following way
	\begin{align*}
		Var(\hat{\pi}_{h,n}(x^*)) & = Var(\frac{1}{n \Delta_n} \sum_{j = 0}^{n - 1} \K_h(x^* - X_{t_j}) \Delta_n) \nonumber \\
		& = \frac{\Delta_n^2}{T_n^2} \left\{nk(t_0)+ 2 \sum_{j=1}^{n-1} (n-j) k(t_j)\right\},
	\end{align*}
with
\begin{equation*}
	k(t_j) =  \text{Cov}(\mathbb{K}_h (x^* - X_{t_j}), \mathbb{K}_h (x^* - X_0)).
 \end{equation*}
We have $k(t_0) \le \E\left[ \mathbb{K}_h (x^* - X_0)^2\right]=\int_\R \mathbb{K}_h (x^* - y)^2 \pi(y) dy \le \frac{\norm{\pi}_\infty \int_\R \K^2(y)dy}{h} \le \frac{c}{h}$, where $c$ is some constant independent of $(a,b) \in \Sigma$, as we know $\norm{\pi}_\infty \le \mathcal{L}$ by Definition \ref{def: insieme sigma v2}.  This leads us to write,
	\begin{equation}{\label{eq: def I tilde_rev}}
	Var(\hat{\pi}_{h,n}(x^*)) \le \frac{c\Delta_n}{T_n h} + 2 \abs{I}
\end{equation} where 
$I= \frac{\Delta_n^2}{T_n^2}\sum_{j=1}^{n-1} (n-j) k(t_j)$.
}	
{\modrev On {\revarn $I$} we want to use the weak discrete Castellana and Leadbetter condition as formulated in \textbf{WDCL1} and \textbf{WDCL2}. We observe it is 
$$k(t_j) =  \text{Cov}(\mathbb{K}_h (x^* - X_{t_j}), \mathbb{K}_h (x^* - X_0)) = \int_{\mathbb{R}} \int_{\mathbb{R}} \mathbb{K}_h (x^* - y) \mathbb{K}_h (x^* - z) g_{t_j} (y,z) dy\, dz\,$$
where {\revarn we recall that $g_{t_j} (y,z) = \pi(y) p_{t_j}(y,z) - \pi(y) \pi(z)$. 
	Hence, we can write }
\begin{align*}
|{\revarn I}| & \le | \frac{\Delta_n^2}{T_n^2} \sum_{{\revarn j = 1}}^{{\revarn n-1}} (n-j) \int_{\mathbb{R}} \int_{\mathbb{R}} \mathbb{K}_h (x^* - y) \mathbb{K}_h (x^* - z) g_{t_j} (y,z) dy\, dz\ | \\
& \le c \frac{\Delta_n}{T_n} \int_{\mathbb{R}} \int_{\mathbb{R}} |\mathbb{K}_h (x^* - y) |\, | \mathbb{K}_h (x^* - z)| \, | \sum_{{\revarn j = 1} }^{{\revarn n-1}} g_{t_j} (y,z) | dy dz \\
& \le c \frac{\Delta_n}{T_n} \int_{\mathbb{R}} |\mathbb{K}_h (x^* - y) |\, \sup_{z \in \R} | \sum_{{\revarn j = 1} }^{{\revarn n-1}} g_{t_j} (y,z) | dy,
\end{align*}
where we have used that 
\begin{equation*}
\int_{\R} | \mathbb{K}_h (x^* - z)| dz \le c,   
\end{equation*}
for some $c > 0$. Now the result follows once we show that \textbf{WDCL1} and \textbf{WDCL2} hold true, as stated in next proposition and proved at the end of this section. }

\begin{proposition}{\label{prop: WDCL}}
{\revarn Suppose that $(a,b) \in \Sigma$. 
Then, conditions \textbf{WDCL1} and \textbf{WDCL2} are satisfied with two bounded functions $k_1$ and $k_2$. Moreover, $\norm{k_1}_\infty$ and $\norm{k_2}_\infty$ are bounded independently of $(a,b) \in \Sigma$.}
\end{proposition}
{\modrev 
It directly follows from the boundedness of $\pi$ and the properties of $K$
\begin{align*}
|{\revarn  I }| & \le c \frac{\Delta_n}{T_n} \int_{\mathbb{R}} |\mathbb{K}_h (x^* - y)|
{\revarn \left[ \frac{k_1(y)+k_2(y)}{\Delta_n} \right]  dy}
\\
& {\revarn \le \frac{c(\norm{k_1}_\infty+\norm{k_2}_\infty) }{T_n}  \int_{\mathbb{R}} |\mathbb{K}_h (x^* - y)| dy \le \frac{c}{T_n}},
\end{align*}
where we have used the fact that the kernel function has compact support and its $L^1$ norm is bounded by a constant. 
\noindent We put all the pieces together {\revarn in \eqref{eq: def I tilde_rev},} which implies the following bound holds true: 
 {\revarn 
 $$Var(\hat{\pi}_{h,n}(x^{ *})) \le \frac{c\Delta_n}{T_n h} + \frac{c}{T_n},$$ 
 which gives the proposition.}}}
\end{proof}
\begin{proof}[Proof of Proposition \ref{prop: WDCL}]
	{\modrev
		We start proving {\bf WDCL1}. Because of {\revarn Theorem \ref{thm : transition density}} we have {\revarn for $0\le t \le 2$,}
		\begin{equation*}
			p_t(x,y) \le c t^{- \frac{1}{2}} 
			e^{ - {\lambda_0} \frac{|{\revarn \theta_{t,0}(x)} - y|^2}{t} } 
			\le {\revarn {c}} t^{- \frac{1}{2}} 
				.
		\end{equation*}
		As $\sup_{y \in \mathbb{R}} \pi(y) < \infty$, it gives
		\begin{align}{\label{eq: start WDCL1}}
			 \sup_{y \in \R} \sum_{{\revarn i = 1}}^{{\revarn \tilde{i}}} |g_{t_i}(x,y)| & \le \sup_{y \in \R} 
			\sum_{{\revarn i = 1}}^{{\revarn \tilde{i}}} {\revarn  c (t_i^{- \frac{1}{2}}  + 1)}  \nonumber
			 \\
			&{\revarn \le \big( \frac{c}{\sqrt{\Delta_n}} \sum_{{\revarn i = 1}}^{{\revarn \tilde{i}}}  \frac{1}{\sqrt{i}} \big) + c \, \tilde{i} 
				\le \frac{c \sqrt{\tilde{i}}}{\sqrt{\Delta_n}} + c \, \tilde{i} =  \frac{c \sqrt{\tilde{i}\Delta_n} + \tilde{i}\Delta_n}{\Delta_n},}
		\end{align}
		where we chose $\tilde{i} = \sup \{ i \in {\revarn [1, n-1]} \mbox{ such that } i \Delta_n \le 2 \}$. Then, Equation \eqref{eq: start WDCL1} here above provides {\bf WDCL1} with
		{\revarn $u_0=2$ and $k_1 (x) = c (\sqrt{2}+2)$.} 
		\\
		We move to the proof of {\bf WDCL2}. Let us set $\varphi(\xi) := \mathbb{E}[\exp(i \xi X_t)]$ and $\varphi_x(\xi, t) := \mathbb{E}[\exp(i \xi X_t)| X_0 = x]$ and claim that there exists a constant $\hat{c} > 0$ such that for all $\xi \in \mathbb{R}$,
		\begin{equation} 
			|\varphi(\xi)| \le  \hat{c}(1 + |\xi|)^{-2}.
			\label{eq: cond varphi}
		\end{equation}
		Moreover, there exists $\tilde{c} > 0$,  such that for all $t \geq 2$, $x \in \mathbb{R}$, and $\xi \in \mathbb{R}$,
		\begin{equation}
			|\varphi_x(\xi, t)| \le \tilde{c}(1 + |\xi|)^{-2}.
			\label{eq: cond varphi x}
		\end{equation}
		We will now prove that, if the conditions \eqref{eq: cond varphi} and \eqref{eq: cond varphi x} hold true, then the result follows. After that, we will conclude our proof by proving that the above mentioned conditions are satisfied in our context. \\
		{\revarn 
Using the inverse Fourier transform, we can write
\begin{equation}\label{eq : Fourier inverse}
	2 \bm{\pi} (p_t(x,y) - \pi(y)) = \int_{\mathbb{R}} \exp(- i \xi y)(\varphi_x(\xi, t)-  \varphi(\xi)) d\xi.
\end{equation}
We set $\psi_\xi(y)=e^{i \xi y} - \int_\R e^{i \xi z} \pi(z) dz$ which is a centered function under the stationary probability, and remark that, with this notation,
$	\varphi_x(\xi, t)-  \varphi(\xi)=P_t(\psi_\xi)(x) $ where $(P_t)_{t \ge 0}$ is the semi-group of the diffusion,
$P_t(\psi_\xi)(x)=\int_\R  p_t(x,y) \psi_\xi(y)dy$.
We know from Lemma 8 in \cite{Minimax} that we have the following semi-group contraction property, with a constant $c$ uniform on $\Sigma$,
\begin{equation}{\label{eq: semi}}
	\norm{P_t(\psi_\xi)}_{L^2(\pi)} \le c e^{-t/c} \norm{\psi_\xi}_\infty.
\end{equation}
{\revarn Let us state the following lemma, whose proof is postponed to the  Appendix \ref{s: technical}.
	\begin{lemma}\label{l: hyper contract}
	There exists a constant $c>0$ depending on the class $\Sigma$, such that for  all $\psi \in L^1(\pi)$, $s\in(0,1]$, $x\in \R$, we have, $|P_s(\psi)(x)| \le \frac{c}{\sqrt{s} \pi(x)} \norm{\psi}_{L^1(\pi)}$.
	\end{lemma}
}	
Applying this lemma with $s=1$, we can write for $t \ge 2$,
\begin{align*}
	\abs{P_t(\psi_\xi)(x)}&=\abs{P_1( P_{t-1}(\psi_\xi))(x)} \le \frac{1}{\pi(x)} \norm{P_{t-1}(\psi_\xi)}_{L^1(\pi)}
	\\
	& \le  \frac{1}{\pi(x)} \norm{P_{t-1}(\psi_\xi)}_{L^2(\pi)} \le \frac{c}{\pi(x)} e^{-(t-1)/c} \norm{\psi_\xi}_\infty,
\end{align*}
where in the second line we used $\norm{\cdot}_{L^1(\pi)} \le \norm{\cdot}_{L^2(\pi)}$ and eventually the contraction property of the semi-group 
as in \eqref{eq: semi}. Since the functions $|\psi_\xi|$ are bounded by the constant $2$, we get
$	\abs{P_t(\psi_\xi)(x)} \le c e^{-t/c}$ for some $c>0$.

Meanwhile, as we claimed that both \eqref{eq: cond varphi} and \eqref{eq: cond varphi x} are satisfied, we have
$ \abs{P_t(\psi_\xi)(x)}=\abs{\varphi_x(\xi,t)-\varphi(\xi)} \le \frac{\hat{c}+\tilde{c}}{1+|\xi|^2}$.
Using \eqref{eq : Fourier inverse}, we deduce
\begin{align*}
	\abs{p_t(x,y) - \pi(y)} &= \frac{1}{2 \bm{\pi} }
	\abs{\int_\R \exp(-i\xi y) P_t(\Psi_\xi)(x)d\xi} \le \frac{1}{2 \bm{\pi} } \int_\R  \abs{P_t(\Psi_\xi)(x)}^{3/4}\abs{P_t(\Psi_\xi)(x)}^{1/4}d\xi
\\	&\le\frac{1}{2 \bm{\pi} } \int_\R  \abs{\frac{\hat{c}+\tilde{c}}{1+|\xi|^2 }}^{3/4}\abs{\frac{c e^{-t/c}}{\pi(x)}}^{1/4}d\xi \le \frac{c}{\pi(x)^{1/4}} e^{-t/c}
\end{align*}
for some $c>0$.
Hence, we get that there exists a finite constant $c$ such that, for all $t \ge 2$ and $x, y \in \R$,
		\begin{equation*}
\vert g_t(x,y) \vert=\pi(x)|p_t(x,y) - \pi(y)| \le c \pi(x)^{3/4} e^{-t/c}.
		\end{equation*}}
		It is worth noting that the right hand side is independent of $y$.

		Replacing it in the definition of \textbf{WDCL2} it yields
		\begin{align*}
			&\sup_{y \in \R} \Delta_n\sum_{i = \tilde{i}+1}^{{\revarn n-1}} |g_{t_i}(x,y)| \le \sup_{y \in \R}  \Delta_n \sum_{i = \tilde{i}+1}^{{\revarn n-1}}  
			{\revarn c\pi(x)^{3/4}} {\revarn e^{- t_i/c}} \\
			& \le c {\revarn  \pi(x)^{3/4} \Delta_n  \sum_{i = \tilde{i}+1}^{n-1} e^{- i \Delta_n/c}
			\le  c \pi(x)^{3/4} \int_2^{\infty} e^{- s/c} ds \le c^2 \pi^{3/4}(x).}			
		\end{align*}}
	{\revarn This implies \textbf{WDCL2} with $u_0 = 2$ and $k_2(x) = c^2 \pi(x)^{3/4}$, and we remark that $\norm{k_2}_{\infty}$ is bounded by $c^2\mathcal{L}^{3/4}$, independently of $(a,b)\in\Sigma$.}
		\\
		{\modrev To conclude, we need to demonstrate that the constraints presented in \eqref{eq: cond varphi} and \eqref{eq: cond varphi x} are satisfied.
		The proof closely follows the argumentation provided on page 7 of \cite{Ver99}. It is based on the integrability of the derivatives of the transition density, which is well-established in the case of a bounded drift (refer to Theorem 7, Chapter 9, Section 6 in Friedman \cite{Friedman64}). However, when dealing with an unbounded drift, the proof becomes more challenging and relies on some bounds {\revarn obtained in \cite{MenPes}.}
		We are ready to prove \eqref{eq: cond varphi} and \eqref{eq: cond varphi x}. Let us start with the proof of \eqref{eq: cond varphi x}. Integrating by parts and using Fubini Theorem yields
		\begin{align*}
			|\varphi_x(\xi, t)| & = |\int_{\R} \exp(i \xi y ) p_t(x,y) dy| = |\xi|^{-2} |\int_{\R} \exp(i \xi y ) \partial^2_y p_t(x,y) dy| \\
			& = |\xi|^{-2} |\int_{\R} \exp(i \xi y ) [\partial^2_y \int_{\R} p_{t-1}(x,z) p_{1}(z,y) dz ]dy| \le |\xi|^{-2} \int_{\R^2} p_{t-1}(x,z) |\partial^2_y  p_{1}(z,y)| dz \, dy \\
			& = |\xi|^{-2} \int_{\R}  p_{t-1}(x,z) \int_{\R}|\partial^2_y  p_{1}(z,y)| dy \, dz.
		\end{align*}
		We want to prove that 
		\begin{equation}{\label{eq: bound second deriv}}
			\int_{\R}|\partial^2_y  p_{1}(z,y)| dy < \infty.
		\end{equation}
		{\revarn From  the results of Section 4 of \cite{MenPes} recalled in Theorem \ref{thm : transition density}, 
			 we know that, under our hypothesis, it is} 
		$$|\partial^2_y  p_{1}(z,y)| \le {\revarn C_0} e^{- {\revarn \lambda_0} {|\theta_{1,0}(z) - y|^2}}$$
		{\revarn where $\theta_{t,s}(z)$ is the deterministic flow as introduced above Theorem \ref{thm : transition density} and $C_0$, $\lambda_0$ depends only on the class $\Sigma$.} 
		Then, the change of variable $\tilde{y} := \theta_{1,0}(z) - y$ yields 
		$$\int_{\R}|\partial^2_y  p_{1}(z,y)| dy \le \int_{\R} {\revarn C_0}e^{-  {\revarn \lambda_0} |\tilde{y}|^2} d\tilde{y},$$
		which is bounded as we wanted. It follows 
		$$|\varphi_x(\xi, t)| \le c|\xi|^{-2} \int_{\R} p_{t-1}(x,z) dz \le c|\xi|^{-2}{\revarn ,}$$
		{\revarn for some $c>0$. This gives \eqref{eq: cond varphi x}.}
		
		The same inequalities hold true for $\varphi$ and provide \eqref{eq: cond varphi}. Indeed, 
		\begin{align*}
			|\varphi(\xi)| & = |\int_{\R} \exp(i \xi y ) \pi(y) dy| = |\xi|^{-2} |\int_{\R} \exp(i \xi y ) [\partial^2_y \int_{\R} \pi(z) p_{1}(z,y) dz ]dy| \\
			& \le |\xi|^{-2} \int_{\R} \pi(z) \int_{\R}|\partial^2_y  p_{1}(z,y)| dy \, dz \le c |\xi|^{-2},
		\end{align*}
		where we have used \eqref{eq: bound second deriv} once again. 
		The proof of Proposition \ref{prop: WDCL} is therefore concluded.}
\end{proof}
	
\subsection{Proof of Theorem \ref{th: discrete d =1}}\label{Ss : proof Th upper bound}
 \begin{proof}
		If $\Delta_n \le (\frac{1}{T_n})^{\frac{1}{2 \beta}}$, then it is enough to choose $h(T_n) := (\frac{1}{T_n})^{\frac{1}{2 \beta}}$ to get
	\begin{align*}
			\mathbb{E}[|\hat{\pi}_{h,n}(x^*) - \pi (x^*)|^2] & \le ch^{2 \beta} + {\revarn \frac{c}{T_n} } + \frac{c}{T_n} \frac{\Delta_n}{h}  \le {\revar  \frac{c}{T_n},} 
		\end{align*}
		which is the first result we aimed {\revarn to show.}\\
		\\
		On the other side, when $\Delta_n > (\frac{1}{T_n})^{\frac{1}{2 \beta}}$, as $T_n = n \Delta_n$ it is also
		$$\Delta_n > (\frac{1}{n})^{\frac{1}{2 \beta + 1}}.$$
		Using the bias-variance decomposition and Proposition \ref{prop: bound var discrete d=1} it follows 
		\begin{align*}
			\mathbb{E}[|\hat{\pi}_{h,n}(x^*) - \pi (x^*)|^2] & \le ch^{2 \beta} + {\revarn \frac{c}{T_n} }  + \frac{c}{T_n} \frac{\Delta_n}{h} \\
			& \le ch^{2 \beta} + {\revarn c (\frac{1}{n})^{1 - \frac{1}{2 \beta + 1}}}  + \frac{c}{n} \frac{1}{h }. 
		\end{align*}
		We take $h(n) := (\frac{1}{n})^{\frac{1}{2 \beta + 1}}$, it yields
		$$\mathbb{E}[|\hat{\pi}_{h,n}(x^*) - \pi (x^*)|^2] \le c (\frac{1}{n})^{\frac{2 \beta}{2 \beta + 1}} + {\revar c (\frac{1}{n})^{\frac{2 \beta}{2 \beta + 1}}} + c (\frac{1}{n})^{\frac{2 \beta}{2 \beta + 1}} $$
		and so the balance is achieved with the convergence rate we wanted.
		\end{proof} 
\subsection{Proof of Theorem \ref{th: borne inf discrete}}\label{Ss : proof lower bound}

The proof of Theorem \ref{th: borne inf discrete} relies on the two {\modr hypotheses} method, as explained for example in Section 2.3 of \cite{Ts}. In the sequel we will introduce the Hellinger distance and we will need to bound it. In order to do that we will use Malliavin calculus {\revarn as it appears in Section \ref{Ss: proof of prop sans Malliavin}.}. The reader may refer to \cite{Nualart} for a detailed exposition of the subject {\revarn and we recall in Appendix \ref{S: Appendix Malliavin} useful notations and results related to the Malliavin calculus.}
\begin{proof}
	{\modr We remind that the set  $\Sigma(\beta, \mathcal{L},a_{\text{min}},(a_l)_{0\le l \le 3}, (b_l)_{0\le l \le 3},\tilde{C},\tilde{\rho})$ 
	has been introduced in Definition \ref{def: insieme sigma v2}.  Using a scaling argument, which consists in replacing the process $X$ by $\lambda X$ with $\lambda>0$, it is possible to assume that $a_{\text{min}}<1<a_0$. This choice will simplify some notations.  }
{\modr	In the proof we will lower bound the risk using in particular the following two models: 
	\begin{align*}
		dX_t = b(X_t) dt + a(X_t) dB_t,
		\quad  X_0 \sim \pi(x) dx,~ \text{has stationary distribution,}
		\\
		d\tilde{X}_t = b(\tilde{X}_t) dt + \tilde{a}(\tilde{X}_t) dB_t, 
		\quad  \tilde{X}_0 \sim \tilde{\pi}(x) dx,~ \text{has stationary distribution,}
	\end{align*}
	where we take $a(x) = 1$ and $\tilde{a}(x)= 1 + \frac{1}{M_n} \psi_{h_n}(x)$ with $\psi_{h_n}(x) = \psi(\frac{x-x^*}{h_n})$, $\psi: \R \rightarrow \R$ is a $C^\infty$ function with support on $[-1, 1]$ such that  {\revar $\psi(0)= 1$,  $\int_{-1}^1 \psi(z) dz=0$ and $\norm{\psi}_\infty \le 1$.} The quantities $M_n$ and $h_n$ will be calibrated later and satisfy $M_n \rightarrow \infty$, $h_n \rightarrow 0$ and $M_nh_n \to \infty$ for $n \rightarrow \infty$; moreover {\revarn we impose $M_n >2$, $h_n<1$, $M_n h_n^3 \ge \frac{1}{\alpha_0} $, where the constant $\alpha_0\in(0,1)$ will be fixed later.}}
	The function $b$ is such that
		\begin{align}\label{eq: def drift}	
			b(x)= -\eta \widehat{\text{sgn}} ({\modr \frac{x-x^*}{A}}), \quad
			\widehat{\text{sgn}} (x) = \begin{cases}
				0 \qquad |x| \le 1 \\
				\sgn x \qquad |x| > 2 \\
				\in (0,  \sgn x) \qquad 1 < |x| \le 2,
			\end{cases}
		\end{align}
		where $x\mapsto \widehat{\text{sgn}} (x)$ is $C^\infty$ {\modr 
			and $\eta>0$, $A>1$.}
	We build $b(x)$ such that it is a $C^\infty$ function satisfying A1-A2 with {\modr $(b_l)_{0\le l \le 3}$, $\tilde{C}$ and some $\tilde{\rho}$. The  constants $\eta$  and $A$ will be calibrated with that regard. }
	{\modr First,} we observe that both the couples of coefficients $(a, b)$ and $(\tilde{a}, b)$ satisfy {\modr the Assumptions A1.} 
	 Indeed, we have $0<1-\frac{1}{M_n}<\tilde{a}(x)<1+\frac{1}{M_n}$, {\modr $|b(0)|=\eta  |\widehat{\text{sgn}}(-x^*/A)|$, }
		 $\norm{\tilde{a}^{(k)}}_\infty \le \frac{\norm{\psi^{(k)}}_\infty}{h_n^kM_n} \le \frac{\norm{\psi^{(k)}}_\infty}{h_n^3M_n} \le \alpha_0 \norm{\psi^{(k)}}_\infty$, $\norm{b^{(k)}}_\infty \le {\modr \frac{\eta}{A^k}} \norm{\widehat{\text{sgn}}^{(k)}}_\infty$, for $k=1,\dots,3$. 
	 Thus, by letting {\modr $A\ge |x^*|$ we get $|b(0)|=0 \le b_0$ and if $\alpha_0$ and $\eta/A$ are sufficiently small,} the coefficients  $(a, b)$ and $(\tilde{a}, b)$ are satisfying {\modr A}ssumption A1 for any fixed $a_\text{min}<1<a_0$, $(a_l)_{1\le l \le 3}$, $(b_l)_{0\le l \le 3}$. 
	 {\modr We now check that the condition A2 holds with $\tilde{C}=\eta/2$ and some $\tilde{\rho}>0$.
For $|x| \ge |x^*|+2A$, we have $\frac{|x-x^*|}{A} \ge 2$, and we can write
\begin{align*}	
	xb(x)&=-\eta x\text{sgn} (\frac{x-x^\star}{A})=-\eta (x-x^*)\text{sgn} (\frac{x-x^\star}{A})
	-\eta x^* \text{sgn}(\frac{x-x^*}{A})	\\
	&\le -\eta |x-x^*| + \eta |x^*| \le -\eta |x| + 2 \eta |x^*|.
\end{align*} 
 We deduce that if $|x| \ge 4 |x^*|$, we have $xb(x) \le -\eta|x|/2 $, and the condition A2 follows with $\tilde{C}=-\eta/2$ and $\rho=\max(|x^*|+2A,4|x^*|)$. 
 It entails to set $\eta=2\tilde{C}$ while $\alpha_0$, $1/A$ can be chosen arbitrarily, up to guarantee they are small enough to ensure the constraints A1 on the coefficients $a$ and $b$.}
	
	As a consequence {\modr of A2}, we know {\revarn from Theorem 1.4 in \cite{Kut}} that both $X$ and $\tilde{X}$ admit a unique invariant distribution that we call $\mu$ (and $\tilde{\mu}$, respectively). We denote their densities as $\pi$ and $\tilde{\pi}$, respectively. \\
	Having as a purpose to show the lower bound using the two {\modr hypotheses} method, following {\revar Sections 2.2--2.4 in} \cite{Ts}, the strategy consists in finding two densities $\pi$ and $\tilde{\pi}$ such that 
	{\revar for $n$ large enough,}
	\begin{enumerate}
		\item $\pi$, $\tilde{\pi} \in \mathcal{H}(\beta,2 \mathcal{L})$, so that $(a,b)$, $(\tilde{a}, b) \in 
		{\modr \Sigma}(\beta, \mathcal{L}, a_{\text{min}},(a_l)_{0\le l \le 3},(b_l)_{0\le l \le 3},\tilde{C}, \tilde{\rho} )=: {\modr \Sigma}$.
		\item $|\tilde{\pi}({\modr x^*}) - \pi({\modr x^*})| \ge \frac{{\revar c}}{M_n}$ {\revar for some $c>0$.}
		\item {\revar $\limsup_n H^2(Law((X_{i \Delta_n})_{i = 0, \dots, n}), Law((\tilde{X}_{i \Delta_n})_{i = 0, \dots, n}) < \epsilon_0 < 2$, where $H^2$ is the squared Hellinger distance on probabilities.}
	\end{enumerate}
	Then, it follows {\revarn from Theorem 2.2 and Section 2.2. in \cite{Ts},}
	\begin{equation}
		\inf_{\tilde{\pi}_{T_n}} \sup_{(a,b) \in \Sigma} \E[(\tilde{\pi}_{T_n}({\modr x^*}) - \pi({\modr x^*}))^2] \ge {\revar \tilde{c} (\frac{1}{M_n})^{2},}
		\label{eq: result lower with calibration}
	\end{equation}
	{\revar for some $\tilde{c}>0$.}
	We now want to check that the three points here above hold true {\revar with a choice of calibration $h_n=\left( \frac{1}{n} \right)^{\frac{1}{1+2\beta}}$, $M_n=1/( \alpha_0h_n^{\beta})=\frac{n^{\frac{\beta}{1+2\beta}}}{\alpha_0}$, where $\alpha_0$ {\revarn is some sufficiently small constant in $(0,1]$.} Let us stress that this choice is consistent with $M_nh_n^3 \ge 1/\alpha_0$ as $\beta\ge 3$ and $h_n \le 1$.}
	
	{\revar \underline{Proof of point 1.}} Regarding {\revar this point, we need to prove that $\pi, \tilde{\pi} \in \mathcal{H}(\beta,{\revarn \mathcal{L}})$
		{\modr as soons as $\alpha_0>0$ and $1/A>0$ are fixed small enough}. 
		We know from Theorem 1.4 in } \cite{Kut} that
	\begin{equation}\label{eq : pi explicit}
		\pi(x) = c_\pi e^{2 \int^x_{\modr x^*} b(y) dy} {\revar  ~= c_\pi e^{-2 \eta \int^x_{\modr x^*} \widehat{\text{sgn}}({\modr \frac{y-x^*}{A}}) dy} }\qquad \mbox{and }
	\end{equation}
	\begin{equation}\label{eq : pi tilde explicit}
		{\revar 	\tilde{\pi}(x) =  \frac{\tilde{c}_\pi}{(1 + \frac{1}{M_n} \psi_{h_n}(x))^2} e^{2 \int^x_{\modr x^*}\frac{b(y)}{(1 + \frac{1}{M_n }\psi_{h_n}(y))^2} dy}=
			\frac{\tilde{c}_\pi}{(1 + \frac{1}{M_n} \psi({\modr \frac{x-x^*}{h_n}}))^2}e^{-2\eta  \int^x_{\modr x^*} \widehat{\text{sgn}}({\modr \frac{y-x^*}{A}}) dy},}
	\end{equation}
	{\revar where we used that $b=0$ on the interval {\modr $[x^*-A,x^*+A]$} which contains the support of $\psi_{h_n}$ as $h_n \le 1$ {\modr and $A\ge1$}.}
	We recall that {\modr $A \ge 1$ in the} definition of $b$ can be chosen as large as we want. 
	{\revar Using the definition of $\widehat{\text{sgn}}$ in \eqref{eq: def drift}, we have
		\begin{align*}
			1 \ge  \int_{{\modr x^*-A}}^{{\modr x^*+A}}\pi(x)dx = c_\pi \int_{{\modr x^*-A}}^{{\modr x^*+A}}
			\exp(-2\eta{\modr \int_{{\modr x^*}}^x  \widehat{\text{sgn}}((y-x^*)/A)dy}) dx =   c_\pi \int_{{\modr x^*-A}}^{{\modr x^*+A}}
			{\modr \exp(0)}dx  \ge {\modr 2A c_\pi.}			
		\end{align*}
		We deduce {\modr $c_\pi \le 1/(2A)$, and} analogously we have {\modr $\tilde{c}_\pi \le 2/A$.}
		From \eqref{eq : pi explicit} and the sign of $b$, we infer that $\norm{\pi}_\infty \le c_\pi \le {\modr 1/(2A)}$. And thus $\norm{\pi}_\infty \le \mathcal{L}$ if {\modr $1/A$ is} chosen small enough. In the same manner, and after differentiations of \eqref{eq : pi explicit}}, we see that  $\norm{\frac{\partial^k}{\partial x^k} \pi}_\infty \le {\modr c(k)/A}$ for some constants $c(k)$ and all $k \ge 1$. 
	 Choosing {\modr $1/A$} small enough, we obtain 
	$\pi \in \mathcal{H}(\beta, {\revarn \mathcal{L}})$. 
	{\revar It is more delicate to see that $\tilde{\pi} \in \mathcal{H}(\beta, {\revarn \mathcal{L}})$ under the condition 
		\begin{equation}
			\frac{1}{M_n} \frac{1}{h_n^\beta}  = \alpha_0  \le 1.
			\label{eq: cond parameter Holder}
		\end{equation}
		By differentiation of \eqref{eq : pi tilde explicit} with $\tilde{c}_\pi \le {\modr 2/A}$, we can prove that 
		$\norm{\frac{\partial^k}{\partial x^k} \tilde{\pi}}_\infty \le \frac{\tilde{c}(k)}{{\modr A}M_n h_n^k}$ for some constants $\tilde{c}(k)$ and all $k \ge 1$. Choosing {\modr $A$ large enough,} it is sufficient with \eqref{eq: cond parameter Holder} to imply that $\norm{\frac{\partial^k}{\partial x^k}\tilde{\pi}}_\infty \le \mathcal{L}$ for all
		$k \in \left\{ 0, 1, \dots , \lfloor \beta \rfloor \right\}$. It remains to prove that $\frac{\partial^{ \lfloor \beta \rfloor}}{\partial x^{ \lfloor \beta \rfloor}}\tilde{\pi}$ is $(\beta- \lfloor \beta \rfloor)$--H\"older. We proceed, as in Lemma 3 of \cite{Lower},
		\begin{align*}
			\abs{\frac{\partial^{ \lfloor \beta \rfloor}}{\partial x^{ \lfloor \beta \rfloor}}\tilde{\pi}(x+l)-\frac{\partial^{ \lfloor \beta \rfloor}}{\partial x^{ \lfloor \beta \rfloor}}\tilde{\pi}(x)}
			&= \abs{\frac{\partial^{ \lfloor \beta \rfloor}}{\partial x^{ \lfloor \beta \rfloor}}\tilde{\pi}(x+l)-\frac{\partial^{ \lfloor \beta \rfloor}}{\partial x^{ \lfloor \beta \rfloor}}\tilde{\pi}(x)}^{\beta-\lfloor \beta \rfloor}
			\abs{\frac{\partial^{ \lfloor \beta \rfloor}}{\partial x^{ \lfloor \beta \rfloor}}\tilde{\pi}(x+l)-\frac{\partial^{ \lfloor \beta \rfloor}}{\partial x^{ \lfloor \beta \rfloor}}\tilde{\pi}(x)}^{1-\beta-\lfloor \beta \rfloor} 
			\\ & \le  \norm{\frac{\partial^{ \lfloor \beta \rfloor+1}}{\partial x^{ \lfloor \beta \rfloor}}\tilde{\pi}}_\infty^{{\beta-\lfloor \beta \rfloor}} |l|^{\beta-\lfloor \beta \rfloor}
			\left(2 \norm{\frac{\partial^{ \lfloor \beta \rfloor}}{\partial x^{ \lfloor \beta \rfloor}}\tilde{\pi}}_\infty\right)^{1-\beta-\lfloor \beta \rfloor} 
			\\ & \le \left(\frac{ \tilde{c}(\lfloor \beta \rfloor+1) }{{\modr A}M_n h_n^{\lfloor \beta \rfloor+1}}\right)^{\beta-\lfloor \beta \rfloor}
			\left(2 \frac{\tilde{c}(\lfloor \beta \rfloor)}{{\modr A}M_n h_n^{\lfloor\beta\rfloor}} \right)^{1-\beta-\lfloor \beta \rfloor}  |l|^{\beta-\lfloor \beta \rfloor}
			\\& =  \frac{1}{{\modr A}M_n h_n^\beta}  \tilde{c}(\lfloor \beta \rfloor+1)^{\beta-\lfloor \beta \rfloor} (2 \tilde{c}(\lfloor \beta \rfloor))^{1-\beta-\lfloor \beta \rfloor} 
			|l|^{\beta-\lfloor \beta \rfloor} 
			\\
			& \le \mathcal{L} |l|^{\beta-\lfloor \beta \rfloor} 
		\end{align*}
		where in the last line we used \eqref{eq: cond parameter Holder} and the fact that {\modr $A$ can be chosen large enough.} Remark that the choices of ${\modr A}$ {\revarn and $\alpha_0$} depend only on the values {\modr $\beta$, $\mathcal{L}$, $|x^*|$} and is fixed from now on.
	}
	\\
	{\revar \underline{Proof of point 2.}}
	We observe it is $\pi({\modr x^*}) = c_\pi$ and {\revar $\tilde{\pi}({\modr x^*}) = \frac{\tilde{c}_\pi}{(1 + \frac{1}{M_n})^2}$.}
	We know that
	$$1 = \int_{\R} \pi(x) dx = c_\pi \int_{\R} e^{2 \int_{{\modr x^*}}^x b(y) dy} dx $$
	and so 
	$$\frac{1}{c_\pi} = \int_{\R} e^{2 \int_{{\modr x^*}}^x b(y) dy} dx.$$
	In the same way
	$$1= \int_{\R} \tilde{\pi}(x) dx =
	\tilde{c}_\pi \int_{\R} \frac{1}{\revar (1 + \frac{1}{M_n} \psi_{h_n}(x))^2} e^{2 \int_{{\modr x^*}}^x 
	{\revar  b(y)} dy} dx.
$$
{\revar 
	We can therefore write, recalling that the support of $\psi_{h_n}$ is $[{\modr x^*}-h_n,{\modr x^*}+h_n]$,}
\begin{align*}
	\frac{1}{\tilde{c}}_\pi &= \int_{\R} \frac{1}{{\revar (1 + \frac{1}{M_n} \psi_{h_n}(x))^2}} e^{2 \int_{{\modr x^*}}^x b(y) dy} dx \\
	& = \int_{|x{\modr - x^*}| > h_n} e^{2 \int_{{\modr x^*}}^x b(y) dy} dx + \int_{|x{\modr - x^*}| \le h_n} \frac{1}{{\revar (1 + \frac{1}{M_n} \psi_{h_n}(x))^2}} e^{2 \int_{{\modr x^*}}^x b(y) dy} dx \\
	& = \frac{1}{c_\pi} + \int_{|x{\modr - x^*}| \le h_n} (\frac{1}{{\revar (1 + \frac{1}{M_n} \psi_{h_n}(x))^2}}-1) e^{2 \int_{{\modr x^*}}^x b(y) dy} dx.
\end{align*}
Remarking that 
$$\sup_{x \in \R} |\frac{1}{{\revar (1 + \frac{1}{M_n} \psi_{h_n}(x))^2}}-1) |= {\revar \frac{-\frac{2}{M_n}\psi_{h_n}(x)+\frac{1}{M_n^2}\psi_{h_n}(x)^2}{ (1 + \frac{1}{M_n} \psi_{h_n}(x))^2}
	=O(\frac{1}{M_n}),}$$
it is easy to see that 
$$\frac{1}{\tilde{c}_\pi} - \frac{1}{c_\pi} = O(\frac{h_n}{M_n})$$
and so it is also 
{\revar 
	\begin{equation} \label{eq : comparaison c c tilde}
		\tilde{c}_\pi - c_\pi= O(\frac{h_n}{M_n}).
\end{equation}}
It follows 
\begin{align*}
	|\tilde{\pi} (0) - \pi(0)| &  = |\frac{\tilde{c}_\pi}{{\revar (1 + \frac{1}{M_n})^2}} - c_\pi| \\
	&={\revar | \frac{\tilde{c}_\pi-c_\pi}{(1 + \frac{1}{M_n})^2}  - c_\pi \frac{\frac{2}{M_n}+\frac{1}{M_n^2}}{(1 + \frac{1}{M_n})^2} | }
	\\&= {\revar\frac{2 c_\pi}{M_n} + O(\frac{h_n}{M_n}) + O(\frac{1}{M_n^2})}
	\\ & {\revar \ge \frac{c}{M_n}}
\end{align*}
which implies the point 2, as we wanted. \\
\\
{\revar \underline{Proof of point 3.}} 
We first recall some properties of total variation and Hellinger distance (refer to Section 2.4 in \cite{Ts}). 
Let $\mathbb{P}$ and $\mathbb{Q}$ be two probability measures on the probability space $(\Omega, \mathcal{F})$, dominated by $\mu$. 
the Hellinger distance $H$ is defined by 
$$H^2(\mathbb{P}, \mathbb{Q}) := \int_{\Omega} \big(\sqrt{\frac{d \mathbb{P}}{d\mu}} - \sqrt{\frac{d \mathbb{Q}}{d\mu}}\big)^2 d\mu.$$
For a product measure we have the following tensorization property, 
$$H^2 (\otimes_{i = 1}^n \mathbb{P}_i, \otimes_{i = 1}^n \mathbb{Q}_i) \le \sum_{i = 1}^n H^2(\mathbb{P}_i, \mathbb{Q}_i).$$
Such a property has been extended to the distribution of Markov chains in Proposition 2.1 of \cite{Clement}. In particular, adapting the results in \cite{Clement} to our framework, we have $(X_{i\Delta_n})_{i \ge 0}$ and $(\tilde{X}_{i \Delta_n})_{i \ge 0}$ two homogeneous Markov chains on $\R$ with transition densities $p$ and $q$ with respect to the Lebesgue measure. The conditional Hellinger distance given $X_0 = \tilde{X}_0 = x$ is defined by 
$$H^2_x(p,q) = \int_{\R}(\sqrt{p(x,y)} - \sqrt{q(x,y)})^2 dy.$$
We denote by $\mathbb{P}^n$ and $\mathbb{Q}^n$ the laws of $X_0, X_{\Delta_n}, \dots , X_{n \Delta_n}$ and $\tilde{X}_0, \tilde{X}_{\Delta_n}, \dots , \tilde{X}_{n \Delta_n}$, respectively. We know that
$${\revar \frac{\mathbb{P}^n(dx)}{dx}} = \prod_{i = 0}^{n- 1} p_{\Delta_n}(x_i, x_{i + 1}) \quad \mbox{and } {\revar \frac{\mathbb{Q}^n(dx)}{dx}} = \prod_{i = 0}^{n- 1} q_{\Delta_n}(x_i, x_{i + 1}),$$
where
$$p_{\Delta_n}(x, y) dy = \mathbb{P}(X_{\Delta_n}\in  dy | X_0 = x), $$
$$q_{\Delta_n}(x, y) dy = \mathbb{P}(\tilde{X}_{\Delta_n}\in  dy | X_0 = x). $$
{\revarn From a slight extension of Proposition 2.1 in \cite{Clement}, which allows for different initial laws and can be obtained by the same induction argument
	as in the proof of Proposition 2.1 in \cite{Clement} },
we have that 
\begin{equation}
H^2(\mathbb{P}^n, \mathbb{Q}^n) \le \frac{1}{2} \sum_{i = 0}^{n -1} \left[\E[H^2_{X_{i \Delta_n}}(p_{\Delta_n}, q_{\Delta_n})] + \E[H^2_{\tilde{X}_{i \Delta_n}}(p_{\Delta_n}, q_{\Delta_n})]\right] 
+{H^2(\pi,\tilde{\pi}).}
\label{eq: hellinger in clement}
\end{equation}
{\modarn Using \eqref{eq : pi explicit}--\eqref{eq :  pi tilde explicit} {\revar with \eqref{eq : comparaison c c tilde},} one can show that 
\begin{equation} \label{eq : Hellinger initiale}
	H^2(\pi,\tilde{\pi}) \le \frac{C h_n}{M_n^2}.
\end{equation}}
Hence, we need to control the conditional {\revar squared} Hellinger distance $H^2_x(p_{\Delta_n}, q_{\Delta_n})$. With this purpose in mind, we interpolate between the two laws $X$ and $\tilde{X}$. For $\epsilon \in [0, 1]$ we define
\begin{equation}
dX_t^\epsilon = b(X_t^\epsilon) dt + (1 + \epsilon \frac{1}{M_n} \psi_{h_n}(X_t^\epsilon))d{\modarn B_t}, \qquad \widehat{X}_0 = x_0,
\label{eq: model epsilon}
\end{equation}
{\revar and let $p_{\Delta_n}^\epsilon(x_0, y)$ denote} the density of $X_{\Delta_n}^\epsilon$. We observe that $p_{\Delta_n}^0 = p_{\Delta_n}$, while $p_{\Delta_n}^1 = q_{\Delta_n}$. \\
As the coefficients are $C^\infty$ and bounded, the function $(x, y, \epsilon) \mapsto p_{\Delta_n}^\epsilon(x, y)$ is smooth. Then, it is possible to bound the conditional Hellinger distance as below: 
\begin{align*}
H^2_{x_0}(p_{\Delta_n}, q_{\Delta_n}) & = \int_{\R}(\sqrt{p^0_{\Delta_n}(x_0,y)} - \sqrt{p^1_{\Delta_n}(x_0,y)})^2 dy \\
&= \int_{\R}(\int_0^1 \frac{\dot{p}^\epsilon_{\Delta_n}(x_0,y)}{\sqrt{p^\epsilon_{\Delta_n}(x_0,y)}} d\epsilon)^2 dy.
\end{align*}
We recall that, as the volatility coefficient is uniformly lower bounded, it is $p^\epsilon_{\Delta_n}> 0$. From Jensen inequality and {\revar Fubini--Tonelli's} theorem we get that the term here above is upper bounded {\revar in the following way }
\begin{align}\nonumber
H^2_{x_0}(p_{\Delta_n},q_{\Delta_n}) &\le  \int_{\R}\int_0^1 (\frac{\dot{p}^\epsilon_{\Delta_n}(x_0,y)}{\sqrt{p^\epsilon_{\Delta_n}(x_0,y)}})^2 d\epsilon dy \\
\nonumber
&= \int_0^1 \int_{\R} (\frac{\dot{p}^\epsilon_{\Delta_n}(x_0,y)}{\sqrt{p^\epsilon_{\Delta_n}(x_0,y)}})^2  dy d\epsilon \\
\nonumber
& {= } \int_0^1 \int_{\R} ( \frac{\dot{p}^\epsilon_{\Delta_n}(x_0,y)}{p^\epsilon_{\Delta_n}(x_0,y)})^2
{ p^\epsilon_{\Delta_n}(x_0,y)}  dy d\epsilon  
\\ \label{eq: majo H2 par Fisher}
& = \int_0^1 { \E_{x_0}} [(\frac{\dot{p}^\epsilon_{\Delta_n}(x_0,{ X^\epsilon_{\Delta_n}})}{p^\epsilon_{\Delta_n}(x_0,{ X^\epsilon_{\Delta_n}})})^2] d\epsilon ,
\end{align}
{\modarn  having denoted as $\E_{x_0}[\cdot]$ the conditional expectation given $X_0^\epsilon = x_0$.}
{\revar  We recognize in the above integral the Fisher information of the statistical model $\epsilon \to X^\epsilon_{\Delta_n}$. 
We now state the following control whose proof is delayed to Section \ref{Ss: proof of prop sans Malliavin}. 
\begin{proposition} Assume that $h_n \le \Delta_n$, $1/M_n = O(h_n^\beta) $ with $\beta\ge 3$ 
	{\revarn and $\Delta_n=O(n^{-\varepsilon})$ for some $\varepsilon>0$.  Then,}
	\label{prop: control Fisher sans Malliavin} 
	there exists a constant $\hat{c}>0$ such that for all $\epsilon \in [0,1]$,
	\begin{equation*}
		\int_{\mathbb{R}} \E_{x_0} [(\frac{\dot{p}^\epsilon_{\Delta_n}(x_0,{ X^\epsilon_{\Delta_n}})}{p^\epsilon_{\Delta_n}(x_0,{X^\epsilon_{\Delta_n}})})^2] \pi^\epsilon(x_0)dx_0
		\le \hat{c} \frac{h_n}{M_n^2},
	\end{equation*} 
	where $\pi^{\epsilon}$ denotes the density of the stationary distribution of $X^{\epsilon}$.
\end{proposition}
Remark that the choices for $h_n$ and $M_n$ with the condition $\Delta_n \ge (\frac{1}{T_n})^{1/(2\beta)}$ implies that $h_n \le \Delta_n$ and $1/M_n = O(h_n^\beta) $. Hence, we can use the previous proposition.
Before this, we note that in the same manner as \eqref{eq : pi tilde explicit} and \eqref{eq : comparaison c c tilde} are obtained,  the stationary distribution of $X^\epsilon$ is explicitly given by 
\begin{equation}\label{eq : pi epsilon explicit}
	\revar 	{\pi^\epsilon}(x) =  \frac{{c}_{\pi^\epsilon}}{(1 + \frac{\epsilon}{M_n} \psi_{h_n}(x))^2} e^{-2\eta \int_{{\modr x^*}}^x
		 \widehat{\text{sgn}}({\modr \frac{y-x^*}{A}}) dy},
\end{equation}
with ${c}_{\pi^\epsilon}-c_\pi=O(\frac{h_n}{M_n})$. Remarking that $c_\pi$ is a constant independent of $n$ and $\epsilon$, and  by comparison of the 
expressions \eqref{eq :  pi explicit}, \eqref{eq :  pi tilde explicit} and  \eqref{eq : pi epsilon explicit} we deduce that 
$\frac{\pi}{\pi^\epsilon}$ and $\frac{\tilde{\pi}}{\pi^\epsilon}$ are bounded by constants independent of $\epsilon$ and $n$.
}
{\revar Now, we integrate \eqref{eq: majo H2 par Fisher} with respect to $\pi(x_0)dx_0$, 
and infer from the Proposition \ref{prop: control Fisher sans Malliavin} that
\begin{equation*}
	\E[H^2_{X_{i \Delta_n}}(p_{\Delta_n}, q_{\Delta_n})]= \int_{\R} H^2_{x_0}(p_{\Delta_n}, q_{\Delta_n}) \pi(x_0)dx_0 \le \int_0^1 
	\int_{\R} \E_{x_0} [(\frac{\dot{p}^\epsilon_{\Delta_n}(x_0,{ X^\epsilon_{\Delta_n}})}{p^\epsilon_{\Delta_n}(x_0,{ X^\epsilon_{\Delta_n}})})^2] \pi(x_0) dx_0d\epsilon		
	\le c \frac{h_n}{M_n^2}.
\end{equation*}	
Similarly, integrating \eqref{eq: majo H2 par Fisher} with respect to $\tilde{\pi}(x_0)dx_0$ yields to $\E[H^2_{\tilde{X}_{i \Delta_n}}(p_{\Delta_n}, q_{\Delta_n})] \le c \frac{h_n}{M_n^2}$. 
{\modarn }
Then, 
\eqref{eq: hellinger in clement} and 
\eqref{eq : Hellinger initiale}
provide
\begin{equation*}
	H^2(\mathbb{P}^n, \mathbb{Q}^n ) \le  {c} n\frac{h_n}{M_n^2}.
\end{equation*}
}
It follows that the third point of the scheme for the lower bound through the two {\modr hypotheses} method is satisfied, up to requiring the following condition:
\begin{equation}
{\revarn \limsup_n  \frac{nh_n}{M_n^2} <}{\modchi \frac{\epsilon_0}{{c}} ,}
\label{eq: cond laws}
\end{equation}
{\modchi with $\epsilon_0$ small enough}. {\revar Recalling $M_n=1/(\alpha_0 h_n^\beta)$ and $h_n=n^{-1/(1+2\beta)}$, we deduce $n\frac{h_n}{M_n^2}=\alpha_0^2 n h_n^{1+2\beta}=\alpha_0^2$. Thus, \eqref{eq: cond laws}  
holds true as soon as $\alpha_0^2 < \epsilon_0/{c}$, which is a feasible choice as 
$\alpha_0 \in (0,1]$ is arbitrary. This proves \eqref{eq: result lower with calibration}.}
\end{proof}

\subsection{Proof of Proposition \ref{prop: control Fisher sans Malliavin}}
	\label{Ss: proof of prop sans Malliavin}

The key point is that we have a Malliavin representation for the score function: 
\begin{equation}\label{eq : Score by Malliavin}
	\frac{\dot{p}^\epsilon_{\Delta_n}(x_0,y)}{p^\epsilon_{\Delta_n}(x_0,y)} = 
{\modarn \E_{x_0}}[W_{x_0, \Delta_n, \epsilon} | X_{\Delta_n}^\epsilon = y]
\end{equation}
for some {\revar random variable $W_{x_0, \Delta_n, \epsilon}$ which is usually referred as} Malliavin weight.
{\revar  The reader may find in Nualart \cite{Nualart} a detailed exposition on Malliavin calculus. In Appendix \ref{s: Malliavin sub section recall}, we also provide a recap of the notations and main properties of the Malliavin operators utilized in this paper. Specifically, the Malliavin operators are defined within the underlying Hilbert space $H = L^2([0, \revar \Delta_n])$. We denote by $\delta$ the Skorohod integral, which is defined as the adjoint operator of the Malliavin operator $D$. Additionally, $\bracket{\cdot}{ \cdot }$ represents the scalar product in $L^2([0, \Delta_n])$, and $\norm{\cdot}_H$ signifies the corresponding norm.}
	  From Theorem 5 in \cite{Glo_Gob} we know it is 
\begin{equation}
W_{x_0, \Delta_n, \epsilon} = \delta \big(\frac{D_\cdot X_{\Delta_n}^\epsilon \dot{X}_{\Delta_n}^\epsilon}{\bracket{D_\cdot X_{\Delta_n}^\epsilon}{D_\cdot X_{\Delta_n}^\epsilon }	}\big),
\label{eq: Malliavin weight}
\end{equation}
 {\revar and $\dot{X}_{\Delta_n}^\epsilon =\frac{\partial  X_{\Delta_n}^\epsilon}{\partial \epsilon}$
	is solution to the {\modr SDE} obtained by formal differentiation of \eqref{eq: model epsilon}:
	\begin{equation}
		\label{eq: X dot avec drift}
		\dot{X}^\epsilon_{t}=\int_0^t b'(X^\epsilon_s) \dot{X}^\epsilon_{s} ds + \int_0^t \left[\frac{\epsilon}{M_n} \psi'_{h_n}(X_s^\epsilon)\dot{X}^\epsilon_{s} + \frac{1}{M_n} \psi_{h_n}(X_s^\epsilon) \right] dB_s.
	\end{equation}
	For the sake of readability, we give in Appendix \ref{Ss: proof of score by Malliavin} a short proof of the formulae \eqref{eq : Score by Malliavin}--\eqref{eq: Malliavin weight}.} 
{\modc We remark that the terms in \eqref{eq: Malliavin weight} depend on $x_0$ as $X_{\Delta_n}^\epsilon$ is actually $X_{\Delta_n}^{\epsilon, x_0}$. 
The Malliavin weight can be {\revarn bounded} as in the following {\revar proposition, which will be {\revarn shown} in the Section \ref{S: study of Malliavin weight}.} 
\begin{proposition} {\revarn  
		Assume that $h_n \le \Delta_n$, $1/M_n = O(h_n^\beta) $ with $\beta\ge 3$ 
		{\revarn and $\Delta_n=O(n^{-\varepsilon})$ for some $\varepsilon>0$.}		}
Let $(X_{i \Delta_n}^\epsilon)_{i = 0, \dots , n-1}$ be the discrete observations of the process solution to \eqref{eq: model epsilon} and $W_{x_0, \Delta_n, \epsilon}$ the Malliavin weight as in \eqref{eq: Malliavin weight}. Then, there exists a constant $\hat{c} > 0$ such that, for any $\epsilon \in [0, 1]$,
{\revar
	\begin{equation}	\label{eq : but control Malliavin}
		|\int_{\mathbb{R}}\mathbb{E}_{x_0} [\big(\E_{x_0}[W_{x_0, \Delta_n, \epsilon} | X_{\Delta_n}^\epsilon]\big)^2]
		\pi^{\epsilon}(x_0)
		dx_0 | \le \hat{c} \frac{h_n}{M_n^2},
	\end{equation}}
	where $\pi^{\epsilon}$ denotes the density of the stationary {\revar distribution} of $X^{\epsilon}$.
\label{prop: bound Malliavin}
\end{proposition}
{\revar Proposition \ref{prop: control Fisher sans Malliavin} is then an immediate consequence of Proposition \ref{prop: bound Malliavin} with \eqref{eq : Score by Malliavin}.
\qed}
\section{Proof of Proposition \ref{prop: bound Malliavin}}
\label{S: study of Malliavin weight}
{\modr After translation of the process $X^\varepsilon$ by a fixed value, it is possible to assume  that $x^*=0$. This will lighten the notations in the proofs appearing in this section and in the Appendix.}
The proof of the bound of the Malliavin weight is divided in several steps. {\modr In the first step below, we show that the main contribution in the integral of  the LHS of 
\eqref{eq : but control Malliavin} comes from a neighborhood of $x^*=0$. In step 2, we compute a more tractable expression of the Malliavin weight \eqref{eq: Malliavin weight} as sum of different terms, while in step 3--5 we study the contribution of each term.}
  \\
\\
\textbf{Step 1: Modifying the process.} \\
A first step is to prove that we can remove the contribution of the drift and that {\revar only the case where $x_0$ is such that $|x_0| \le \Delta_n^{\frac{1}{2} - \gamma}$, for $\gamma > 0$ arbitrarily small matters.}  To do that, we need the following lemma, whose proof can be found in the {\revar Appendix \ref{s: technical}.} 
\begin{lemma}
Assume that {\revar $(A_n)_n$, $(B_n)_n$, $(A'_n)_n$, $(B'_n)_n$ }are some sequences such that, on some set $\Omega_n$, it is 
{\revar $$(A_n, B_n) = (A'_n,B'_n).$$}
Moreover the complementary set of $\Omega_n$ is such that
\begin{equation}
\mathbb{P}(\Omega_n^c) \le {\revarn \kappa_r} n^{-r}
\label{eq: proba omega}
\end{equation}
for any $r > 1$ {\revarn and constants $\kappa_r\ge0$. Assume also that} 
{\modar
	\begin{equation}
		\left \| {\revar A_n} \right \|_{L^p} + \left \| {\revar A'_n } \right \|_{L^p} \le c_p n^{r_0}
		\label{eq: maj grossiere moment}
\end{equation}  }
for some {\revarn $r_0 \ge 0$, $p > 1$, $c_p \ge 0$.} Then, for any ${ \revar 1 \le p' }  < p$ and for any $r> 1$, we have
 $${\revar \left \| |\E[A_n | B_n] - \E[A'_n | B'_n]|\right \|_{L^{p'}} } \le c n^{-r},$$
where the constant c depends on {\revarn $p'$, $r_0$, $p$, $c_p$ and the $(\kappa_r)_r$. } 
\label{l: AB}
\end{lemma}
We observe that, if in the lemma here above we can choose $p' = 2$ and so if $p > 2$, then it follows 
\begin{equation}
{\revar \E[\E[A_n | B_n]^2] \le c \E[\E[A_n' | B_n']^2] }+ cn^{- r}.
\label{eq: AB}
\end{equation}
We want to apply this inequality to our context, we therefore need to define {\revar $A_n := W_{x_0, \Delta_n, \epsilon}$,} where $W_{x_0, \Delta_n, \epsilon}$ has been defined in \eqref{eq: Malliavin weight}, and {\revar $B_n = X_{\Delta_n}^\epsilon$.} Having as a purpose to show that the contribution provided by the case where $|x_0| \ge \Delta_n^{\frac{1}{2} - \gamma}$ is negligible, we introduce the set 
$$\Omega_n := \left \{ \omega : \, X_s^\epsilon(\omega ) \notin [-h_n, h_n] \quad \forall s \in [0, \Delta_n] \right \}.$$
Let $|x_0| \ge \Delta_n^{\frac{1}{2} - \gamma}$.
{\revarn 	We  have assumed $h_n \le \Delta_n$, 
hence,} using Markov inequality, we get
{\revar 
\begin{align*}
    \mathbb{P}_{x_0}(\Omega_n^c) & \le \mathbb{P}_{x_0} 
  (\exists s \in [0, \Delta_n]: \, |X_s^{\epsilon} - x_0|\ge \Delta_n^{\frac{1}{2} - \gamma} - {\revar  \Delta_n}) \\
 & \le \mathbb{P}_{x_0} ( \sup_{s \in [0,\Delta_n] } |X_s^\epsilon-x_0| \ge \frac{\Delta_n^{1/2-\gamma}}{2}), \quad \text{for $n$ large enough}
 \\
 & \le 2^{r'} \frac{\E_{x_0}\left[ \sup_{s \in [0,\Delta_n] } |X_s^\epsilon-x_0|^{r'}\right] }{\Delta_n^{(1/2-\gamma)r'}}, \quad \text{with any $r'  > 1$.}
\end{align*}
To control the expectation appearing  in the last equation, we use Burkholder-Davis-Gundy's inequality, and the fact that the coefficients of the {\modr SDE} \eqref{eq: model epsilon} are bounded in the following way :
$\norm{b}_\infty \le 1$, $\norm{1+\frac{\epsilon\psi(\cdot)}{M_n}}_\infty \le 1+\frac{\norm{\psi}_\infty}{M_n}  \le 3/2$. This implies that
$ \E_{x_0}\left[ \sup_{s \in [0,\Delta_n] } |X_s^\epsilon-x_0|^{r'}\right]  \le c_{r'} \Delta_n^{r'/2}$. In turn, we get $\mathbb{P}_{x_0}(\Omega_n^c) \le c_{r'}\Delta_n^{\gamma r'}$. As the constant $0<\gamma<1/2$ is fixed, and {\revarn the sampling step converges to zero at least with some polynomial rate in $n$,} it is possible to choose appropriately $r'$ in order to satisfy the condition 
\begin{equation*}
	\mathbb{P}_{x_0}(\Omega_n^c) \le c_{r}n^{-r},
\end{equation*}
with any $r > 0$.} 
We remark that the bound we get in the right hand side here above is independent on $x_0$. \\
We let then ${\revar A'_n} = 0$ and ${\revar B'_n = B_n} = X_{\Delta_n}^\epsilon$. We remark that, on $\Omega_n$, by the definition of $X^\epsilon$ we know that its derivative with respect to $\epsilon$ is going to be $0$ and so $ W_{x_0, \Delta_n, \epsilon} = 0$ as well. It follows that, on $\Omega_n$,  ${\revar A'_n = A_n}$ (and {\revar by definition} ${\revar B'_n = B_n}$).
To apply Lemma \ref{l: AB} we are left to check {\revar that the $L^p$ norm of the Malliavin weight $W_{x_0,\Delta_n,\epsilon}$ satisfies 
the condition \eqref{eq: maj grossiere moment}.
We use the following lemma whose proof is given in Appendix \ref{s: technical}.
\begin{lemma}\label{L: maj grossiere Malliavin}
	We have $\sup_{x_0\in\mathbb{R}} \E[|W_{x_0,\Delta_n,\epsilon}|^4] = O(\Delta_n^{-4})$.
\end{lemma}	
Using this lemma with
the condition $\Delta_n >  (\frac{1}{n})^\frac{1}{2 \beta + 1}$ we deduce that \eqref{eq: maj grossiere moment} holds true for some $r_0>0$ and $p=4$. }
{\revar Thus,} we can choose $p'=2$ in Lemma \ref{l: AB} and apply \eqref{eq: AB}. It follows that, for $|x_0| \ge \Delta_n^{\frac{1}{2} - \gamma}$, 
\begin{align*}
\E_{x_0}[\E_{x_0}[W_{x_0, \Delta_n, \epsilon} | X_{\Delta_n}^\epsilon]^2] & \le 2 \E_{x_0}[\E_{x_0}[0 | X_{\Delta_n}^\epsilon]^2] + o(n^{- r}) \\
= o(n^{- r})
\end{align*}
for any $r > 0$ and it is independent 
{\modarn of $x_0$. 
} We can therefore focus on the case where $|x_0| \le \Delta_n^{\frac{1}{2} - \gamma}$ {\revar and so the control 
\eqref{eq : but control Malliavin} becomes a consequence of
\begin{equation*}
	|\int_{|x_0|\le \Delta_n^{1/2-\gamma}}\mathbb{E}_{x_0} [\big(\E_{x_0}[W_{x_0, \Delta_n, \epsilon} | X_{\Delta_n}^\epsilon]\big)^2]
\pi^{\epsilon}(x_0)
dx_0 | \le \hat{c} \frac{h_n}{M_n^2}.
\end{equation*}}

We now want to show that the drift function does not provide any contribution and so we can replace the models here above with the same ones, but with a drift coefficient which is now $0$. \\
{\revarn Consider 
	${\revnotat \widehat{X}^\epsilon}$  the same model as 
	$X^\epsilon$, but with $b \equiv 0$.} Set now 
$$\tilde{\Omega}_n := \left \{ \omega, \, X_s^{\epsilon}(\omega) \in [-1, 1], \, \forall s \in [0, \Delta_n] \right \}.$$
Acting as we did in order to compute the probability of $\Omega_n^c$ it is {\revar clearly, for $n$ large enough,}
{\revar 
\begin{align*}
\sup_{|x_0|\le \Delta_n^{1/2-\gamma}} \mathbb{P}_{x_0}(\tilde{\Omega}_n^c) 
&\le
\sup_{|x_0|\le \Delta_n^{1/2-\gamma}}  \mathbb{P}_{x_0} (\sup_{s\in[0,\Delta_n]} |X^\epsilon_s-x_0| \ge 1-{\Delta_n^{1/2-\gamma}})
\\&\le
\sup_{|x_0|\le \Delta_n^{1/2-\gamma}}  \mathbb{P}_{x_0} (\sup_{s\in[0,\Delta_n]} |X^\epsilon_s-x_0| \ge 1/2)
= o(n^{- r})
\end{align*}
for any $r > 0$.} 
From the definition of the drift coefficient in \eqref{eq: def drift} we can see that, on $\tilde{\Omega}_n$, ${\revnotat \widehat{X}^\epsilon} = X^\epsilon$ $\forall \epsilon$ and 
$W_{x_0, \Delta_n, \epsilon} = {\revnotat \widehat{W}}_{x_0, \Delta_n, \epsilon} $, where $ {\revnotat \widehat{W}}_{x_0, \Delta_n, \epsilon} $ is the Malliavin weight associated to the model with $b \equiv 0$. Hence, applying Lemma \ref{l: AB} with ${\revar A_n} = W_{x_0, \Delta_n, \epsilon} $, ${\revar A'_n} =  {\revnotat \widehat{W}}_{x_0, \Delta_n, \epsilon} $ and ${\revar B'_n} = {\revnotat \widehat{X}}_{\Delta_n}^\epsilon$ we get
$$\E_{x_0}[\E_{x_0}[W_{x_0, \Delta_n, \epsilon} | X_{\Delta_n}^\epsilon]^2]  \le c \E_{x_0}[\E_{x_0}[{\revnotat \widehat{W}}_{x_0, \Delta_n, \epsilon}  | {\revnotat \widehat{X}}_{\Delta_n}^\epsilon]^2] + o(n^{- r}).$$
Therefore, for $|x_0| \le \Delta_n^{\frac{1}{2} - \gamma}$, we can do as if $b \equiv 0$, which yields to some simplification in the computation of the Malliavin weight. \\
{\modc It follows our goal becomes to show that 
\begin{equation}
\int_{|x_0| \le \Delta_n^{\frac{1}{2} - \gamma}} \E_{x_0}[(\E_{x_0} [{\revnotat \widehat{W}}_{x_0, \Delta_n, \epsilon}  | {\revnotat \widehat{X}}_{\Delta_n}^\epsilon ])^2] 
{\modarn \pi^\epsilon} (x_0) {\modarn d{x_0}} \le {\revar \hat{c}} \frac{h_n}{M_n^2}.
\label{eq: 16.5}
\end{equation}
{\modchi This will be a consequence of \eqref{eq: 22.5}, \eqref{eq: 23.5} and \eqref{eq: 31.5} that we show below.} \\
 We underline that, even if in the sequel in all our computations we will refer to the model where the drift function $b$ is identically $0$, the invariant density clearly refers to the stationary process, for which the drift is different from $0$. Moving to the case where the drift is equal to $0$, we need to prove some upper bounds  which would have been straightforward with the original drift, thanks to the stationarity of the process. They are gathered in the following lemma, whose proof can be found in the {\revar Appendix \ref{s: technical}.} 
\begin{lemma}
Let ${X}_{\Delta_n}^\epsilon$ and ${\revnotat\widehat{X}}_{\Delta_n}^\epsilon$ be defined as above. Then, the following estimations hold true for $r > 1$ as in Lemma \ref{l: AB}. 
\begin{enumerate}
    \item For any $p \ge 1$ there exists a constant $c> 0$ such that\\
    $
    {\modar \sup_{0<s\le\Delta_n}}
    \int_{\R} \E_{x_0}[|\psi_{h_n}({\revnotat\widehat{X}}_{{\modar s}}^\epsilon)|^p] {\modarn \pi^\epsilon} (x_0) d{x_0} \le c h_n + o(n^{- r})$. 
    \item For any $p \ge 1$ and $k \ge 1$ there exists a constant $c> 0$ such that\\
    $ {\modar \sup_{0<s\le\Delta_n}}\int_{\R} \E_{x_0}[|\psi_{h_n}^{(k)}({\revnotat\widehat{X}}_{{\modar s}}^\epsilon)|^p] {\modarn \pi^\epsilon}(x_0) d{x_0} \le c h_n^{1 - k p} + o(n^{- r})$.
    \item There exists a constant $c > 0$ such that \\
    $\int_0^{\Delta_n} \int_{\R} \E_{x_0}[({\revnotat\widehat{X}}_{u}^\epsilon - {\revnotat\widehat{X}}_{0}^\epsilon)^2 \psi_{h_n}^2({\revnotat\widehat{X}}_{u}^\epsilon)] du \, {\modarn \pi^\epsilon} (x_0) d{x_0} \le c \Delta_n^2 h_n + o(n^{- r})$.
\end{enumerate}
\label{l: bound stationarity}
\end{lemma}
}
\noindent
\textbf{Step 2: Formal derivation of the processes.} \\
  We recall that the Malliavin weight {\revnotat $\widehat{W}_{x_0,\Delta_n,\epsilon}$} is explicit and it is as in \eqref{eq: Malliavin weight}, {\revnotat where
  	the process $X^\epsilon$ is replaced by $\widehat{X}^\epsilon$}. 
 We now want to formally derive all the processes that play a role in ${\revnotat \widehat{W}}_{x_0, \Delta_n, \epsilon}$. As 
 {\modar
 	\begin{equation}
 		\label{eq: SDE sans drift}
 		{\revnotat \widehat{X}}_t^\epsilon = x_0 + \int_0^t (1 + \epsilon \frac{1}{M_n} \psi_{h_n}({\revnotat \widehat{X}}_s^\epsilon))d{\modarn B_s,}
 \end{equation} }
we have that ${\revnotat \smash{\dot{\widehat{X}}}}^\epsilon_t := \frac{\partial {\revnotat \widehat{X}}_t^\epsilon}{\partial \epsilon}$, {\revar by formal derivation (see Theorem 5.24 in \cite{Bichteler_et_al87})}, is solution to
\begin{equation}
{\revnotat \smash{\dot{\widehat{X}}}}^\epsilon_t = \int_0^t \frac{\epsilon}{M_n} \psi_{h_n}'({\revnotat \widehat{X}}_s^\epsilon) {\revnotat \smash{\dot{\widehat{X}}}}^\epsilon_s dB_s + \int_0^t \frac{1}{M_n} \psi_{h_n}({\revnotat \widehat{X}}_s^\epsilon) dB_s. 
\label{eq: def x dot}
\end{equation}
Let {\revar $Y_t^\epsilon := \frac{\partial {\revnotat \widehat{X}}_t^\epsilon}{\partial x_0}$ be the flow of {\modr SDE} \eqref{eq: SDE sans drift}, which is the} solution of 
$$Y_t^\epsilon = 1 + \int_0^t \frac{\epsilon}{M_n} \psi_{h_n}'({\revnotat \widehat{X}}_s^\epsilon) Y^\epsilon_s dB_s.$$
Using variation of the constant method, we have 
\begin{equation}
{\revnotat \smash{\dot{\widehat{X}}}}^\epsilon_t = Y_t^\epsilon \int_0^t \psi_{h_n}({\revnotat \widehat{X}}_s^\epsilon)(Y_s^\epsilon)^{-1}\frac{1}{M_n} dB_s - \frac{{\revar \varepsilon}Y_t^\epsilon}{M_n^2} \int_0^t \psi_{h_n}'({\revnotat \widehat{X}}_s^\epsilon) \psi_{h_n}({\revnotat \widehat{X}}_s^\epsilon)(Y_s^\epsilon)^{-1} ds
\label{eq: model X dot}
\end{equation}
and $Y^\epsilon$ is explicit as Dol\'eans-Dade exponential:
\begin{equation}
Y_t^\epsilon = \exp(\frac{\epsilon}{M_n}\int_0^t \psi_{h_n}'(X_s^\epsilon)dB_s - \frac{\epsilon^2}{2 M_n^2} \int_0^t (\psi_{h_n}'(X_s^\epsilon))^2 ds  ).
\label{eq: DD exp}
\end{equation}
{\modrev By standard computation (see Equation (2.59) in \cite{Nualart})} it is 
\begin{equation}
D_s {\revnotat \widehat{X}}_t^\epsilon = (Y_s^\epsilon)^{-1} (Y_t^\epsilon)(1 + \psi_{h_n}({\revnotat \widehat{X}}_s^\epsilon)\frac{\epsilon}{M_n}).
\label{eq: DsXt}
\end{equation}
{\modrev Recall the Malliavin weight is ${\revnotat \widehat{W}_{x_0, \Delta_n, \epsilon}} = {\revnotat \delta \big(\frac{D_\cdot \widehat{X}_{\Delta_n}^\epsilon \smash{\dot{\widehat{X}}}^\epsilon_{\Delta_n}}{\bracket{D_\cdot \widehat{X}_{\Delta_n}^\epsilon}{D_\cdot \widehat{X}_{\Delta_n}^\epsilon}	}\big)}$, as proven in Appendix \ref{S: Appendix Malliavin}. Then, we can compute it explicitly. It holds} {\revar 
\begin{align}{\label{eq: split W}}
	{\revnotat \widehat{W}_{x_0, \Delta_n, \epsilon}}
 & = {\revnotat\smash{\dot{\widehat{X}}}}^\epsilon_{\Delta_n} \delta \big(\frac{D_\cdot {\revnotat \widehat{X}}_{\Delta_n}^\epsilon}{\bracket{D_\cdot {\revnotat \widehat{X}}_{\Delta_n}^\epsilon}{D_\cdot {\revnotat \widehat{X}}_{\Delta_n}^\epsilon }}\big) - \frac{\bracket{ D_\cdot {\revnotat \widehat{X}}_{\Delta_n}^\epsilon}{D_\cdot {\revnotat \smash{\dot{\widehat{X}}}}_{\Delta_n}^\epsilon}}{\bracket{D_\cdot {\revnotat \widehat{X}}_{\Delta_n}^\epsilon}{D_\cdot {\revnotat \widehat{X}}_{\Delta_n}^\epsilon }} 
\\&= {\revnotat \widehat{W}}^1_{x_0, \Delta_n, \epsilon} + {\revnotat \widehat{W}}^2_{x_0, \Delta_n, \epsilon}
\label{eq: 22.5}
\end{align}
where we used Proposition 1.3.3 in Nualart \cite{Nualart}. This is possible as the process $u\mapsto\frac{D_u {\revnotat \widehat{X}}^\epsilon_{\Delta_n}}{
\bracket{D_\cdot {\revnotat \widehat{X}}_{\Delta_n}^\epsilon}{D_\cdot {\revnotat \widehat{X}}_{\Delta_n}^\epsilon}}$ can be shown, by the same computations as those in the proof of Lemma \ref{L: maj grossiere Malliavin}, to be an element of the space $\mathbb{D}^{1,2}(H)$, which is in the domain of $\delta$ (see Appendix \ref{S: Appendix Malliavin}).} 

\noindent
\textbf{Step 3: Handling ${\revnotat \widehat{W}}^1_{x_0, \Delta_n, \epsilon}$.} \\
{\modc We now study the two terms {\revar separately}. Regarding the first one,
	{\modarn the following proposition holds true. The proof of this proposition, given in \revar Appendix \ref{s: technical},} 
	is based on the fact that we can approximate $Y^{\epsilon}_t$, $(Y^{\epsilon}_s)^{-1}$ and $D_u {\revnotat \widehat{X}}^\epsilon_{\Delta_n}$ by $1$, comitting an 
error which is negligible. 
}
\begin{proposition}
Let $({\revnotat \widehat{X}}_{i \Delta_n}^\epsilon)_{i = 0, \dots , n-1}$ be the discrete observations of the process solution to {\revarn \eqref{eq: SDE sans drift}} and ${\revnotat \widehat{W}}^1_{x_0, \Delta_n, \epsilon}$ the Malliavin weight defined as above. Then, 
$${\revnotat \widehat{W}}^1_{x_0, \Delta_n, \epsilon} = \frac{1}{\Delta_n}\int_0^{\Delta_n} \frac{1}{M_n} \psi_{h_n}({\revnotat \widehat{X}}_s^\epsilon) dB_s B_{\Delta_n} + {\modar R_n^{(1)}},$$
where ${\modar R_n^{(1)}}$ is a remainder term and it is such that 
{\modar 
	\begin{equation} \label{eq: control Rn integrated}
\int_{|x_0|\le \Delta_n^{1/2-\gamma}}
\E_{x_0}\left[|R_n^{(1)}|^2\right] {\modarn \pi^\epsilon}(x_0) dx_0\le \frac{c}{M_n^4 h_n^4 }=o(\frac{h_n}{M_n^2}).
\end{equation}	
}	
\label{p: bound W1}
\end{proposition}
{\modchi From Jensen inequality \eqref{eq: control Rn integrated} implies that the contribution of the remainder term is negligible in \eqref{eq: 16.5}.}
We now prove the result \eqref{eq: 16.5} for the principal term of ${\revnotat\widehat{W}}^1_{x_0, \Delta_n, \epsilon}$.
Let us introduce the following notation:
$${\revar \mathcal{M}_{{u}}} := \int_0^{{\modarn u}} \frac{1}{M_n} \psi_{h_n}({\revnotat \widehat{X}}_s^\epsilon) dB_s, 
\qquad {\revar \mathcal{N}_{{u}}} := \frac{1}{\Delta_n} B_{{\modarn u}} {\modarn , \quad \text{ for $u \in [0,\Delta_n]$.} }$$
Then, we actually need to study in detail 
$$\int_{|x_0| \le \Delta_n^{\frac{1}{2} - \gamma}} \E_{x_0}[{\revar \mathcal{M}_{\Delta_n}^2} {\revar \mathcal{N}_{\Delta_n}^2}] {\modarn \pi^\epsilon}(x_0) dx_0.$$
Using {\revar It\^o's formula} it is 
\begin{align*}
\E_{x_0}\left[{\revar \mathcal{M}_{\Delta_n}^2} {\revar \mathcal{N}_{\Delta_n}^2}\right]& = \E_{x_0}\left[\int_0^{\Delta_n} {\revar \mathcal{M}_s^2} {\modarn d\bracketIto{{\revar \mathcal{N}}}{{\revar \mathcal{N}}}_s }+ \int_0^{\Delta_n} {\revar \mathcal{N}_s^2} {\modarn d\bracketIto{{\revar \mathcal{M}}}{{\revar \mathcal{M}}}_s} + \int_0^{\Delta_n} {\revar \mathcal{M}_s} {\revar \mathcal{N}_s} {\modarn d\bracketIto{{\revar \mathcal{M}}}{{\revar \mathcal{N}}}_s}\right] \\
& \le 2 \E_{x_0}\left[\int_0^{\Delta_n} {\revar \mathcal{M}_s^2} {\modarn d\bracketIto{{\revar \mathcal{N}}}{{\revar \mathcal{N}}}_s }+ \int_0^{\Delta_n} {\revar \mathcal{N}_s^2} {\modarn d\bracketIto{{\revar \mathcal{M}}}{{\revar \mathcal{M}}}_s }\right] {\revar ,}
\end{align*}
{\revar where we used Kunita-Watanabe inequality (see Corollary 1.16 of Chapter IV in \cite{RevYor_3rd_edition}).}
According to the decomposition here above we have  {\revar 
\begin{multline*}
\int_{|x_0| \le \Delta_n^{\frac{1}{2} - \gamma}} \E_{x_0}[{\revar \mathcal{M}_{\Delta_n}^2} {\revar \mathcal{N}_{\Delta_n}^2}] {\modarn \pi^\epsilon}(x_0) dx_0 \le 2\int_{|x_0| \le \Delta_n^{\frac{1}{2} - \gamma}}
\E_{x_0}[\int_0^{\Delta_n} {\revar \mathcal{M}_s^2} {\modarn d\bracketIto{{\revar \mathcal{N}}}{{\revar \mathcal{N}}}_s }]dx_0\\
+ 2\int_{|x_0| \le \Delta_n^{\frac{1}{2} - \gamma}} \E_{x_0}[  \int_0^{\Delta_n} {\revar \mathcal{N}_s^2} {\modarn d\bracketIto{{\revar \mathcal{M}}}{{\revar \mathcal{M}}}_s }] dx_0 
=:2 A_1 +2 A_2.
\end{multline*}
}

Regarding $A_1$, the first point of Lemma \ref{l: bound stationarity} ensures that
\begin{align*}
A_1 & \le \frac{c}{M_n^2{\modarn\Delta_n^2}} \int_{|x_0| \le \Delta_n^{\frac{1}{2} - \gamma}} \int_0^{\Delta_n} \int_0^s \E_{x_0}[\psi_{h_n}({\revnotat \widehat{X}}_u^{\epsilon})^2] du ds {\modarn \pi^\epsilon}(x_0) dx_0  \\
& \le {\modarn \frac{c}{M_n^2} h_n.}
\end{align*}
We now turn studying 
$$A_2= \frac{2}{M_n^2{\modarn\Delta_n^2}}\int_{|x_0| \le \Delta_n^{\frac{1}{2} - \gamma}} \int_0^{\Delta_n} \E_{x_0}[B_u^2 \psi_{h_n}({\revnotat \widehat{X}}_u^{\epsilon})^2] du {\modarn \pi^\epsilon}(x_0) dx_0.$$
To bound this term we want to replace $B_u^2$ and, in order to do that, we use the dynamics of our process ${\revnotat \widehat{X}}$. Indeed, it is 
\begin{align*}
dB_u & = a_\epsilon({\revnotat \widehat{X}}_u^\epsilon)^{-1} d{\revnotat \widehat{X}}_u^{\epsilon} \\
& = d{\revnotat \widehat{X}}_u^{\epsilon} - \frac{\frac{\epsilon}{M_n}\psi_{h_n}({\revnotat \widehat{X}}_u^\epsilon)}{1 +\frac{\epsilon}{M_n}\psi_{h_n}({\revnotat \widehat{X}}_u^\epsilon)} d{\revnotat \widehat{X}}_u^{\epsilon}.
\end{align*}
It follows 
$$B_u = ({\revnotat \widehat{X}}_u^{\epsilon} - {\revnotat \widehat{X}}_0^{\epsilon}) - \frac{\epsilon}{M_n} \int_0^u \frac{\psi_{h_n}({\revnotat \widehat{X}}_v^\epsilon)}{1 +\frac{\epsilon}{M_n}\psi_{h_n}({\revnotat \widehat{X}}_v^\epsilon)} d{\revnotat \widehat{X}}_v^{\epsilon}.$$
Hence, if we replace it in $A_2$, we obtain 
\begin{align*}
A_2 & \le \frac{c}{M_n^2{\modarn\Delta_n^2}} \int_{|x_0| \le \Delta_n^{\frac{1}{2} - \gamma}} \int_0^{\Delta_n} \E_{x_0}[({\revnotat \widehat{X}}_u^{\epsilon} - {\revnotat \widehat{X}}_0^{\epsilon})^2 \psi_{h_n}({\revnotat \widehat{X}}_u^{\epsilon})^2] du {\modarn \pi^\epsilon}(x_0) dx_0 \\
& + \frac{c}{M_n^4{\modarn\Delta_n^2}} \int_{|x_0| \le \Delta_n^{\frac{1}{2} - \gamma}} \int_0^{\Delta_n} \E_{x_0}[(\int_0^u \frac{\psi_{h_n}({\revnotat \widehat{X}}_v^\epsilon)}{1 +\frac{\epsilon}{M_n}\psi_{h_n}({\revnotat \widehat{X}}_v^\epsilon)} d{\revnotat \widehat{X}}_v^{\epsilon})^2  \psi_{h_n}({\revnotat \widehat{X}}_u^{\epsilon})^2] du {\modarn \pi^\epsilon}(x_0) dx_0
&= : A_{2, 1} + A_{2, 2}.
\end{align*}
The bound on $A_{2, 1}$ is a direct consequence of the third point of Lemma \ref{l: bound stationarity}
{\modarn which yields to $A_{2,1} \le c \frac{h_n}{M_n^2}$.}
Regarding $A_{2,2}$, from Cauchy-Schwarz, Burkholder-Davis-Gundy inequalities and the first point of Lemma \ref{l: bound stationarity} we have 
\begin{align*}
A_{2,2} & \le {\modarn c\frac{ h_n^{\frac{1}{2}}}{M_n^4}.}
\end{align*}
It is negligible compared to {\modarn  $c\frac{h_n}{M_n^2}$, since $\frac{1}{h_n^{1/2}M_n^2}=O(h_n^{2\beta-1/2}) \xrightarrow{n\to\infty}0$.} 
{\modchi Together with \eqref{eq: control Rn integrated}, it concludes the bound on the contribution of ${\revnotat\widehat{W}}^1_{x_0, \Delta_n, \epsilon}$ in \eqref{eq: 16.5}:
\begin{equation}
\int_{|x_0| \le \Delta_n^{\frac{1}{2} - \gamma}} \E_{x_0}[(\E_{x_0} [ {\revnotat \widehat{W}}^1_{x_0, \Delta_n, \epsilon}  | {\revnotat \widehat{X}}_{\Delta_n}^\epsilon ])^2] 
{\modarn \pi^\epsilon (x_0) d{x_0}} \le c \frac{h_n}{M_n^2}.
\label{eq: 23.5}
\end{equation}}}\\
\\
\textbf{Step 4: Handling ${\revnotat \widehat{W}}^2_{x_0, \Delta_n, \epsilon}$, derivation leading term.}\\
It is more complicated to find a bound on ${\revnotat \widehat{W}}^2_{x_0, \Delta_n, \epsilon}$. We need first of all to derive the explicit expression for $D_\cdot {\revnotat \smash{\dot{\widehat{X}}}}_{\Delta_n}^\epsilon$. 
{\revarn We omit the details of the computations, which follow the following route. First, we apply Theorem 2.2.1 of Nualart \cite{Nualart}, to
	obtain that $t\mapsto D_u {\revnotat \smash{\dot{\widehat{X}}}}_{t}^\epsilon$ is solution of a linear {\modr SDE} obtained by formal differentiation of the dynamics \eqref{eq: model X dot}. Second, we solve this linear {\modr SDE} by the usual method of variation of parameters. It yields to the following explicit representation for $D_u {\revnotat \smash{\dot{\widehat{X}}}}_{\Delta_n}^\epsilon$.}   
\begin{equation}
D_u {\revnotat \smash{\dot{\widehat{X}}}}_{\Delta_n}^\epsilon = (D_u {\revnotat \smash{\dot{\widehat{X}}}}_{\Delta_n}^\epsilon)_1 + (D_u {\revnotat \smash{\dot{\widehat{X}}}}_{\Delta_n}^\epsilon)_2 + (D_u {\revnotat \smash{\dot{\widehat{X}}}}_{\Delta_n}^\epsilon)_3, 
\label{eq: split DuX dot}
\end{equation}
with 
$$(D_u {\revnotat \smash{\dot{\widehat{X}}}}_{t}^\epsilon)_1 = [\frac{\epsilon}{M_n} \psi_{h_n}'({\revnotat \widehat{X}}_u^\epsilon){\revnotat \smash{\dot{\widehat{X}}}}_u^\epsilon + \psi_{h_n}({\revnotat \widehat{X}}_u^\epsilon) \frac{1}{M_n}  ] (Y_u^\epsilon)^{-1} Y_t^\epsilon,$$
\begin{align*}
(D_u {\revnotat \smash{\dot{\widehat{X}}}}_{t}^\epsilon)_2 &= Y_t^\epsilon \int_u^t (Y_s^\epsilon)^{-1} [\frac{\epsilon}{M_n} \psi_{h_n}''({\revnotat \widehat{X}}_s^\epsilon)(D_u {\revnotat \widehat{X}}_s^\epsilon) {\revnotat \smash{\dot{\widehat{X}}}}_s^\epsilon ] dB_s \\
& - Y_t^\epsilon \int_u^t (Y_s^\epsilon)^{-1}(\frac{\epsilon}{M_n})^2 \psi_{h_n}'({\revnotat \widehat{X}}_s^\epsilon)\psi_{h_n}''({\revnotat \widehat{X}}_s^\epsilon)(D_u {\revnotat \widehat{X}}_s^\epsilon) {\revnotat \smash{\dot{\widehat{X}}}}_s^\epsilon ds \\
& - Y_t^\epsilon \int_u^t (Y_s^\epsilon)^{-1}\frac{\epsilon}{M_n^2} (\psi_{h_n}'({\revnotat \widehat{X}}_s^\epsilon))^2(D_u {\revnotat \widehat{X}}_s^\epsilon)  ds,
\end{align*}
$$(D_u {\revnotat \smash{\dot{\widehat{X}}}}_{t}^\epsilon)_3 = Y_t^\epsilon \int_u^t (Y_s^\epsilon)^{-1}\frac{1}{M_n} \psi_{h_n}'({\revnotat \widehat{X}}_s^\epsilon) (D_u {\revnotat \widehat{X}}_s^\epsilon)  dB_s. $$

It follows that
$${\revnotat \widehat{W}}^2_{x_0, \Delta_n, \epsilon}= \frac{\bracket{D_\cdot {\revnotat \widehat{X}}_{\Delta_n}^\epsilon}{D_\cdot {\revnotat \smash{\dot{\widehat{X}}}}_{\Delta_n}^\epsilon}}{\bracket{D_\cdot {\revnotat \widehat{X}}_{\Delta_n}^\epsilon}{D_\cdot {\revnotat \widehat{X}}_{\Delta_n}^\epsilon }},$$ 
consists in three terms.
We therefore need to bound
\begin{align}
&\int_{|x_0| \le \Delta_n^{\frac{1}{2} - \gamma}} \E_{x_0}[\E_{x_0}[{\revnotat \widehat{W}}^2_{x_0, \Delta_n, \epsilon} | {\revnotat \widehat{X}}_\Delta^\epsilon]^2] {\modarn \pi^\epsilon}(x_0) dx_0 
\nonumber
\\
&= \int_{|x_0| \le \Delta_n^{\frac{1}{2} - \gamma}} \E_{x_0}[\E_{x_0}[\frac{\bracket{D_\cdot {\revnotat \widehat{X}}_{\Delta_n}^\epsilon}{ (D_\cdot {\revnotat \smash{\dot{\widehat{X}}}}_{\Delta_n}^\epsilon)_1 + (D_\cdot {\revnotat \smash{\dot{\widehat{X}}}}_{\Delta_n}^\epsilon)_2 + (D_\cdot {\revnotat \smash{\dot{\widehat{X}}}}_{\Delta_n}^\epsilon)_3 }}{\bracket{D_\cdot {\revnotat \widehat{X}}_{\Delta_n}^\epsilon}{D_\cdot {\revnotat \widehat{X}}_{\Delta_n}^\epsilon }} | {\revnotat \widehat{X}}_\Delta^\epsilon]^2] {\modarn \pi^\epsilon}(x_0) dx_0 
\nonumber
\\ \label{eq: def I1-I3}
& =: I_1 + I_2 + I_3.
\end{align}
From Equation \eqref{eq: split DuX dot} and below it is easy to see that there are some terms negligible in $I_1$ and $I_2$, due to presence of $\dot{X}^\epsilon_s$. Indeed, the following lemma on the process $\dot{X}^\epsilon$ holds true. Its proof can be found in the {\revar Appendix \ref{s: technical}.}
{\modar 
\begin{lemma}
	Let ${\revnotat \smash{\dot{\widehat{X}}}}^\epsilon$ be the process solution to \eqref{eq: model X dot}. Then, for any $p \ge 2$, there exists a constant $c > 0$ such that
	\begin{align} \label{eq: control Xdot}
	{\modarn \sup_{x_0 \in \mathbb{R}}}	\sup_{0<s\le \Delta_n}\E_{\modarn x_0}[|{\revnotat \smash{\dot{\widehat{X}}}}^\epsilon_{s}|^p] &\le \frac{ \Delta_n^{p/2}}{M_n^p}
		 \\ \label{eq: control Xdot integrated}
		\int_{|x_0|\le \Delta_n^{1/2-\gamma}}\E_{x_0}[(\frac{{\revnotat \smash{\dot{\widehat{X}}}}^\epsilon_{\Delta_n}}{\sqrt{\Delta_n}})^p] {\modarn \pi^\epsilon}(x_0)dx_0&\le c \frac{h_n}{M_n^p}
	\end{align}
	\label{l: bound moments Xdot}
\end{lemma}
}
{\modc
Moreover, some rough estimation are enough to bound $I_1$ and $I_2$, while we need sharper bounds for the analysis of $I_3$. This is {\revarn shown} in Lemma \ref{l: I1 I2} below, whose proof is in the {\revar Appendix \ref{s: technical}.} 
\begin{lemma}
Let $I_1$ and $I_2$ be as above. Then, there exists a constant $c > 0$ such that 
$$I_1 + I_2 \le c \frac{h_n}{M_n^2}.$$
\label{l: I1 I2}
\end{lemma}
{\modarn Concerning $I_3$, approximating $Y_t^\epsilon$, $(Y_s^\epsilon)^{-1}$ and $D_s{\revnotat \widehat{X}}^\epsilon_u$ by the constant $1$, as in Proposition
	\ref{p: bound W1}, leads us to the following decomposition.}
\begin{lemma}
Let all the processes be as previously defined. Then, the following decomposition holds true:
$$\frac{\bracket{ D_\cdot {\revnotat \widehat{X}}_{\Delta_n}^\epsilon}{(D_\cdot {\revnotat \smash{\dot{\widehat{X}}}}_{\Delta_n}^\epsilon)_3 }}{\bracket{D_\cdot {\revnotat \widehat{X}}_{\Delta_n}^\epsilon}{D_\cdot {\revnotat \widehat{X}}_{\Delta_n}^\epsilon }}
= \frac{1}{\Delta_n M_n} \int_{0}^{\Delta_n} \int_u^{\Delta_n} \psi_{h_n}'({\revnotat \widehat{X}}_s^\epsilon) dB_s du + {\modar R_n^{(2)}}, $$
{\modar where 
	$R^{(2)}_n$ is such that $\int_{|x_0|\le \Delta_n^{1/2-\gamma}}\E_{x_0}[|R^{(2)}_n|^2]{\modarn \pi^\epsilon(x_0)}dx_0 \le \frac{c h_n}{M_n^2}$ for some $c > 0$.}
\label{l: I3}
\end{lemma}
{\modarn The proof of Lemma \ref{l: I3} is postponed to the {\revar Appendix \ref{s: technical}.}}

\textbf{Step 5: Handling the principal term of ${\revnotat\widehat{W}}^2_{x_0, \Delta_n, \epsilon}$.} \\
{\modar We are left to study the main term of $I_3$ coming from the expansion given in the Lemma \ref{l: I3}. The first step is to get rid of the stochastic integral by application of {\revar the  It\^o's formula,} as given by the next lemma whose proof is given in the {\revar Appendix \ref{s: technical}.} 
\begin{lemma}\label{L: suppressing dB in principal term}
	We have 
		\begin{equation}
			\frac{1}{\Delta_n M_n} \int_{0}^{\Delta_n} \int_u^{\Delta_n} \psi_{h_n}'({\revnotat \widehat{X}}_s^\epsilon) dB_s du
				=-\frac{1}{2\Delta_n M_n}  \int_0^{\Delta_n}  \psi''_{h_n}({\revnotat \widehat{X}}_s^\epsilon)sds+R^{(3)}_n
			\label{eq: suppressing dB in principal term}
			\end{equation}
		where 
		$R^{(3)}_n$ is such that $\int_{|x_0|\le \Delta_n^{1/2-\gamma}}\E_{x_0}[|R^{(3)}_n|^2]dx_0 \le \frac{c h_n}{M_n^2}$ for some $c > 0$.
\end{lemma}
As a consequence, the main term in $I_3$ is given by 
\begin{equation}{\label{eq: 28.5}}
\hat{I}_3 := \int_{|x_0| \le \Delta_n^{\frac{1}{2} - \gamma}} \E_{x_0}[\big(\E_{x_0}[- \frac{1}{2 \Delta_n M_n} \int_0^{\Delta_n} \psi_{h_n}'' ({\revnotat \widehat{X}}_s^\epsilon)sds \mid {\revnotat \widehat{X}}^\epsilon_{\Delta_n}]\big)^2] {\modarn \pi^\epsilon}(x_0) dx_0.
\end{equation}
Here, using Jensen's inequality to get rid of the conditional expectation does not give the correct rate, and so we have to analyze in details the conditional expectation.
Hence, we need a sharp bound for this quantity, as the one gathered in Lemma \ref{l: main term malliavin} stated below and {\revarn shown} in the {\revar Appendix \ref{s: technical}.}
\begin{lemma}{\label{l: main term malliavin}}
	For any $0<\eta<1/2$, there exists $C_\eta$ such that
	{\modarn 
	\begin{multline} \label{E: terme pal crochet psi pp}
		|\E_x\left[
		\frac{1}{\Delta_n}\int_0^{\Delta_n} 
		(\psi_{h_n})^{\prime\prime}({\revnotat \widehat{X}}^\epsilon_s) sds
		\mid {\revnotat \widehat{X}}^\epsilon_{\Delta_n}=y\right]|
		\\
		\le C_\eta \left[(1+\frac{h_n}{\sqrt{\Delta_n}})\Indi{\Abs{y}\le 2h_n}+
		\frac{h_n}{\sqrt{\Delta_n}} \Phi( \frac{1}{\eta}\frac{\Abs{y}-h_n}{\sqrt{\Delta_n}})
		\Indi{\Abs{y}> 2h_n}
		\right]
		e^{\frac{3\eta(x-y)^2}{\Delta_n}}
		, 		
	\end{multline}
}
	where $\Phi(u)=1+\frac{e^{-u^2}}{u}$ for $u>0$.
\end{lemma}
It follows {\modarn using the Gaussian control \eqref{E:Azencott_derivative_x}, given in the {\revar Appendix \ref{s: technical},} on the transition density of the diffusion process, that} 
\begin{align}{\label{eq: principalv2}}
	& \int_{|x_0| \le \Delta_n^{\frac{1}{2} - \gamma}} \E_{x_0}[\big(\E_{x_0}[- \frac{1}{2 \Delta_n M_n} \int_0^{\Delta_n} \psi_{h_n}'' ({\revnotat \widehat{X}}_s^\epsilon)sds \mid {\revnotat \widehat{X}}^\epsilon_{\Delta_n}]\big)^2] {\modarn\pi^\epsilon}(x_0) dx_0 \\
	& \le \frac{c}{M_n^2} \int_{\R^2} \left[ (1+\frac{h_n}{\sqrt{\Delta_n}})^2\Indi{\Abs{y}\le 2h_n}+
	\frac{h_n^2}{{\Delta_n}} \Phi^2( \frac{1}{\eta}\frac{\Abs{y}-h_n}{\sqrt{\Delta_n}})
	\Indi{\Abs{y}> 2h_n}\right]
	e^{\frac{6\eta(x_0-y)^2}{\Delta_n}}\frac{e^{{\modarn - \frac{(x_0-y)^2}{c\Delta_n}}}}{\sqrt{\Delta_n}} dy {\modarn\pi^\epsilon}(x_0) dx_0{\modarn ,} \nonumber
\end{align}
{\modarn for some constant $c>0$.}
Remarking that $\eta \in (0, \frac{1}{2})$ can be chosen arbitrarily small, so that we have in particular ${\modarn c^{-1}} - 6 \eta > 0$, it is easy to see that 
{\revar
	\begin{multline*}
\int_{\R^2} (1+\frac{h_n}{\sqrt{\Delta_n}})^2\Indi{\Abs{y}\le 2h_n} \frac{e^{- \frac{({ c^{-1}}- 6 \eta)(x_0-y)^2}{\Delta_n}}}{\sqrt{\Delta_n}} dy {\pi^\epsilon}(x_0) dx_0 \\
\le (\sup_{x_0\in \mathbb{R},\epsilon\in[0,1]} \pi^\epsilon(x_0)) \times \int_\R (1+\frac{h_n}{\sqrt{\Delta_n}})^2 \Indi{|y|\le 2h_n}dy \le  c h_n,
\end{multline*}
where we used $h_n \le \Delta_n \le \sqrt{\Delta_n}$. It} provides the wanted bound on the first term of \eqref{eq: principalv2}. Regarding the second, we replace the function $\Phi$ obtaining 
\begin{align}{\label{eq: final}}
	& \frac{c}{M_n^2}\int_{\R^2} \frac{h_n^2}{{\Delta_n}}(1 + e^{- \frac{1}{\eta^2}\frac{(|y| - h_n)^2}{\Delta_n}}\frac{\eta \sqrt{\Delta_n}}{|y| - h_n})^2 \Indi{\Abs{y} > 2h_n} \frac{e^{- \frac{({\modarn c^{-1}}- 6 \eta)(x_0-y)^2}{\Delta_n}}}{\sqrt{\Delta_n}} dy {\modarn\pi^\epsilon}(x_0) dx_0 \nonumber \\
	& \le  \frac{c}{M_n^2}\frac{h_n^2}{{\Delta_n}} \int_{\R^2}  \frac{e^{- \frac{({\modarn c^{-1}}- 6 \eta)(x_0-y)^2}{\Delta_n}}}{\sqrt{\Delta_n}} dy {\modarn\pi^\epsilon}(x_0) dx_0 \nonumber \\
	& + \frac{c}{M_n^2} \frac{h_n^2}{{\Delta_n}} \int_{\R^2} e^{- \frac{2}{\eta^2}\frac{(|y| - h_n)^2}{\Delta_n}}\frac{\eta^2{\Delta_n}}{(|y| - h_n)^2} \Indi{\Abs{y} > 2h_n} \frac{e^{- \frac{({\modarn c^{-1}}- 6 \eta)(x_0-y)^2}{\Delta_n}}}{\sqrt{\Delta_n}} dy {\modarn\pi^\epsilon}(x_0) dx_0 \nonumber \\
	& \le \frac{c}{M_n^2}\frac{h_n^2}{{\Delta_n}}(1 + \int_{|y| > 2 h_n} e^{- \frac{2}{\eta^2}\frac{(|y| - h_n)^2}{\Delta_n}}\frac{\eta^2{\Delta_n}}{(|y| - h_n)^2} dy ).
\end{align}
On the last integral we apply the change of variable $\tilde{y}:= \frac{y - h_n}{\eta \sqrt{\Delta_n}}$ {\modarn on the part $y>2h_n$, and use symmetry of the integrand on the part $y<-2h_n$}. We obtain it is 
{\modarn smaller than 
} 
\begin{align*}
	&\frac{c}{M_n^2}\frac{h_n^2}{{\Delta_n}} \int_{\frac{h_n}{\eta \sqrt{\Delta_n}}}^{\infty} \sqrt{\Delta_n} \frac{e^{- 2 \tilde{y}^2}}{\tilde{y}^2} d\tilde{y} \\
	& \le \frac{c}{M_n^2}\frac{h_n^2}{{\Delta_n}} \sqrt{\Delta_n} \frac{\sqrt{\Delta_n} \eta}{h_n} = \frac{c h_n}{M_n^2},
\end{align*}
where we have used that ${\modarn \int_t^{\infty} \frac{e^{- u^2}}{u^2} du    = O\left( \frac{1}{t}\right)}$ and in the last {\revar inequality} we have included $\eta$ in the constant $c$. \\
{\modchi The control on the principal term of $\widehat{W}^2_{x_0, \Delta_n, \epsilon}$ is concluded by replacing the estimation here above in \eqref{eq: final} and remarking that, as $h_n \le \Delta_n$, it {\revar holds} $\frac{c}{M_n^2} \frac{h_n^2}{\Delta_n} \le \frac{c h_n}{M_n^2}$. We deduce $\hat{I}_3 \le \frac{c h_n}{M_n^2}$. Collecting \eqref{eq: def I1-I3} with Lemmas \ref{l: I1 I2}, \ref{l: I3}, \ref{L: suppressing dB in principal term} and with \eqref{eq: 28.5} it follows 
\begin{equation}
\int_{|x_0| \le \Delta_n^{\frac{1}{2} - \gamma}} \E_{x_0}[(\E_{x_0} [{\revnotat\widehat{W}}^2_{x_0, \Delta_n, \epsilon}  | {\revnotat \widehat{X}}_{\Delta_n}^\epsilon ])^2] 
{\modarn \pi^\epsilon} (x_0) {\modarn d{x_0}} \le c \frac{h_n}{M_n^2}.
\label{eq: 31.5}
\end{equation}
From \eqref{eq: 22.5}, \eqref{eq: 23.5} and \eqref{eq: 31.5} we obtain \eqref{eq: 16.5}, which concludes the proof. }

}
}
\begin{appendix}\label{Appendix : tech}
{\revar 
\section{proof of technical results}{\label{s: technical}}
}
This section is devoted to the proof of the results which are more technical and for which some preliminaries are needed.
{\revarn 
\subsection{Proof of Lemma \ref{l: hyper contract}}
\begin{proof}
For $\psi \in L^1(\pi)$, we write $\pi(x) P_s(\psi)(x)=\pi(x) \int_\R p_s(x,y) \psi(y) dy=\int_\R \frac{\pi(x) p_s(x,y)}{\pi(y)} \pi(y) \psi(y) dy$. Using that the one dimensional diffusion $X$ is reversible we have $\pi(x)p_s(x,y)=\pi(y)p_s(y,x)$, and as a result, $\pi(x) P_s(\psi)(x)=	\int_\R p_s(y,x) \pi(y) \psi(y) dy$. From Theorem \ref{thm : transition density}, we have for $s \in (0,1]$, the upper bound $p_s(y,x) \le c/\sqrt{s}$, where the constant $c$ is uniform on the class $\Sigma$. We deduce  $\pi(x) |P_s(\psi)(x)| \le \frac{c}{\sqrt{s}} \int_\R \pi(y) |\psi(y)| dy = \frac{c}{\sqrt{s}} \norm{\psi}_{L^1(\pi)}$.	
\end{proof}
}
\subsection{Proof of Lemma \ref{l: AB}}
\begin{proof}
	Let us denote $g_n({\revar B_n})=\E [{\revar A_n} \mid {\revar B_n}]$ and 
	$g'_n({\revar B'_n})=\E[{\revar A'_n}\mid {\revar B'_n}]$ where $g_n$ and $g'_n$ are some measurable functions. Let $p'<p$, by duality, we have
	\begin{equation} \label{eq : lemma 1 duality}
		\E\left[\abs{g_n({\revar B_n})-g'_n({\revar B'_n})}^{p'}\right]^{1/p'}
=\sup_{\underset{Z=h({\revar B_n},{\revar B'_n})}{\norm{Z}_{q'}\le1}} 
\E\left[(g_n({\revar B_n})-g'_n({\revar B'_n}) ) Z\right]
	\end{equation}
	where $q'=\frac{p'}{p'-1}$.
For $M>0$, we set $Z^{(M)}=Z \Indi{\abs{Z}\le M}$ and write
{\revar 
\begin{equation} \label{eq : splitting lemma 1}
\abs{\E\left[(g_n({\revar B_n})-g'_n({\revar B'_n}) ) Z\right]}
\le 
\abs{\E\left[(g_n({\revar B_n})-g'_n({\revar B'_n}) ) Z^{(M)}\right]}
+
\abs{\E\left[(g_n({\revar B_n})-g'_n({\revar B'_n}) ) Z  \Indi{\abs{Z}>M}\right]}.
\end{equation}
}
Using consecutively {\revarn H\"older's inequality where $q=p/(p-1)$ and Minkowski's inequality in the first line below, Jensen's inequality in the second line,  and \eqref{eq: maj grossiere moment} in the third line, we can write}
\begin{align*}
	\abs{\E\left[(g_n({\revar B_n})-g'_n({\revar B'_n}) ) Z  \Indi{\abs{Z}>M}\right]}
&\le 
	\left\{
	\E[|g_n({\revar B_n})|^p]^{1/p}+ \E[|g_n({\revar B'_n})|^p]^{1/p}
\right\} 
\times \E[|Z|^q \Indi{\abs{Z>M}}]^{1/q}
\\
&\le 
	\left\{
\E[|{\revar A_n}|^p]^{1/p}+ \E[|{\revar A'_n}|^p]^{1/p}
\right\} 
\times \E[|Z|^q \Indi{\abs{Z>M}}]^{1/q} 
\\
&\le {\revarn \kappa_p} n^{r_0} \E[|Z|^q \Indi{\abs{Z>M}}]^{1/q}.
 \end{align*}
As $p>p'$ we have $q<q'$ and using  again H\"older inequality we deduce
$E[|Z|^q \Indi{\abs{Z>M}}]^{1/q} \le E[|Z|^{q'}]^{1/{q'}} \mathbb{P}(\abs{Z}>M)^{(q-q')/{\modarn (q'q)}} \le (1/M)^{{\modarn (q-q')/q}}$, since $\norm{Z}_{q'}\le 1$. As a consequence, choosing $M=n^{\frac{(r_0+r)q}{{\modarn (q-q')}}}$  for $r>0$, we deduce
\begin{equation}\label{eq : diff est cond Z indi M}
\sup_{\underset{Z=h({\revar B_n},{\revar B'_n})}{\norm{Z}_{q'}\le1}}
\abs{\E\left[(g_n({\revar B_n})-g'_n({\revar B'_n}) ) Z  \Indi{\abs{Z}>M}\right]}
\le {\revarn \kappa_p}  \frac{n^{r_0}}{M^{{\modarn (q-q')/q}}} \le {\revarn \kappa_p} n^{-r}.
\end{equation}
We now focus on the the first term in the right hand side of \eqref{eq : splitting lemma 1}. Using that the $L^p$ norm of $g_n({\revar B_n})$ and
	$g'_n({\revar B_n'})$ is upper bounded by {\revarn $n^{r_0}$,} that $\abs{Z^{(M)}}\le M$ and \eqref{eq: proba omega}, we have
\begin{equation} \label{eq : esp cond reduit Omega}
\abs{\E\left[ {\revar (} g_n({\revar B_n})-g'_n({\revar B'_n}) {\revar )} Z^{(M)}	\right]}
\le 
\abs{\E\left[(g_n({\revar B_n})-g_n'({\revar B'_n})) Z^{(M)} \Indi{\Omega_n}	
	\right]}
+ O(n^{{\revarn r_0}}Mn^{-r'(1-1/p)}),
\end{equation}
where $r'>0$ can be chosen arbitrarily large.
On $\Omega_n$, we have $Z^{(M)}=h({\revar B_n},{\revar B'_n})\Indi{|h({\revar B_n},{\revar B'_n})|\le M}
=h({\revar B_n},{\revar B_n})\Indi{|h({\revar B_n},{\revar B_n})|\le M}$, and it follows
\begin{align*}
	\E\left[g_n({\revar B_n})Z^{(M)}\Indi{\Omega_n}\right]
	&=
	\E\left[g_n({\revar B_n})h({\revar B_n},{\revar B_n})\Indi{|h({\revar B_n},{\revar B_n})|\le M}\Indi{\Omega_n}\right]
	\\
	&=\E\left[g_n({\revar B_n})h({\revar B_n},{\revar B_n})\Indi{|h({\revar B_n},{\revar B_n})|\le M}\right]
	+O(n^{{\revarn r_0}}Mn^{-r'(1-1/p)})
	\\
	&=\E\left[A_nh({\revar B_n},{\revar B_n})\Indi{|h({\revar B_n},{\revar B_n})|\le M}\right]+O(n^{{\revarn r_0}}Mn^{-r'(1-1/p)})
\end{align*}
	where in the last line we used $g_n({\revar B_n})=\E[{\revar A_n} \mid {\revar B_n}]$. In an {\modarn analogous} way, we have 	$\E\left[g'_n({\revar B'_n})Z^{(M)}\Indi{\Omega_n}\right]=
	\E\left[A'_nh({\revar B'_n},{\revar B'_n})\Indi{|h({\revar B'_n},{\revar B'_n})|\le M}\right]+O(n^{{\revarn r_0}}Mn^{-r'(1-1/p)})$.
	Now, we deduce from \eqref{eq : esp cond reduit Omega},
\begin{multline*}
\abs{\E\left[   {\revar (} g_n({\revar B_n})-g'_n({\revar B'_n}) {\revar )} Z^{(M)}	\right]}
\le \abs{\E\left[A_nh({\revar B_n},{\revar B_n})\Indi{|h({\revar B_n},{\revar B_n})|\le M}\right]-\E\left[A'_nh({\revar B'_n},{\revar B'_n})\Indi{|h({\revar B'_n},{\revar B'_n})|\le M}\right]}
\\
+O(n^{{\revarn r_0}}Mn^{-r'(1-1/p)}).
\end{multline*}
As $({\revar A_n},{\revar B_n})=({\revar A'_n},{\revar B'_n})$ on $\Omega_n$, it implies
\begin{multline} \label{eq : diff est cond ZM}
	\abs{\E\left[
		{\revar (}
		g_n({\revar B_n})-g'_n({\revar B'_n}) {\revar )}Z^{(M)}	\right]}
	\le \abs{\E\left[A_nh({\revar B_n},{\revar B_n})\Indi{|h({\revar B_n},{\revar B_n})|\le M}\Indi{\Omega_n^c}\right]-\E\left[A'_nh({\revar B'_n},{\revar B'_n})\Indi{|h({\revar B'_n},{\revar B'_n})|\le M}\Indi{\Omega_n^c}\right]}
	\\
	+O(n^{{\revarn r_0}}Mn^{-r'(1-1/p)}) = O(n^{{\revarn r_0}}Mn^{-r'(1-1/p)}). 
\end{multline}
	Collecting \eqref{eq : lemma 1 duality}, \eqref{eq : splitting lemma 1}, \eqref{eq : diff est cond Z indi M} and \eqref{eq : diff est cond ZM} we deduce that 
	\begin{equation*}
		\E\left[\abs{g_n({\revar B_n})-g'_n({\revar B'_n})}^{p'}\right]^{1/p'}
		\le C n^{{\revarn r_0}}Mn^{-r'(1-1/p)}=Cn^{{\revarn r_0}}n^{{\modarn \frac{(r_0+r)q}{(q-q')}}}n^{-r'(1-1/p)} \le C n^{-r},
	\end{equation*}
	as we can choose $r'$ arbitrarily large. The lemma is shown.
\end{proof}
{\revar 
\subsection{Proof of Lemma \ref{L: maj grossiere Malliavin}}
We need to prove that $\sup_{x_0\in\mathbb{R}} \E[|W_{x_0,\Delta_n,\epsilon}|^4] = O(\Delta_n^{-4})$. 	Using the expression \eqref{eq: Malliavin weight} and the fact that the operator $\delta : \mathbb{D}^{1,p}(H) \to L^p$ is bounded (see Proposition 1.5.8 in Nualart 
	\cite{Nualart}), it is sufficient to prove that $u \in \mathbb{D}^{1,4}(H)$ and $\norm{u}_{\mathbb{D}^{1,4}(H)} \le c \Delta_n^{-1}$ with 
	$u_t=\frac{D_t X^\epsilon_{\Delta_n} \dot{X}^\epsilon_{\Delta_n} }{\bracket{DX^\epsilon_\cdot}{DX^\epsilon_\cdot}}$. 
	Now, using the {\modr Leibniz} rule for the Malliavin derivative and H\"older inequality, it is possible to extend the Proposition 1.5.6. in Nualart \cite{Nualart} and get
	$\norm{u}_{\mathbb{D}^{1,4}(H)}  \le c \norm{D_\cdot X^\epsilon_{\Delta_n}}_{\mathbb{D}^{1,16}(H)} 
	\norm{\dot{X}^\epsilon_{\Delta_n} }_{\mathbb{D}^{1,16}(\R)} \norm{\frac{1}{\bracket{DX^\epsilon_\cdot}{DX^\epsilon_\cdot}}}_{\mathbb{D}^{1,8}(\R)}$. It remains to bound the three norms in the right hand side of the last equation. 
	
	By recalling \eqref{eq: model epsilon} and \eqref{eq: X dot avec drift}, we remark that the process $t \mapsto (X_t^\epsilon,\dot{X}_t^\epsilon)$ is solution of the {\modr SDE}
\begin{equation*}
	\left\{
	\begin{aligned}
			X_t^\epsilon
			&=x_0+\int_0^t	b(X_s^\epsilon)	ds+ \int_0^t a_\epsilon(X_s^\epsilon) dB_s
			\\
			\dot{X}_t^\epsilon &=\int_0^t b'(X_s^\epsilon)\dot{X}_s^\epsilon ds +
			\int_0^t [\dot{a}_\epsilon(X_s^\epsilon)+ a'_\epsilon(X_s^\epsilon) \dot{X}_s^\epsilon]dB_s	
	\end{aligned}	
	\right.	
\end{equation*}
where $a_\epsilon(x)=1+\frac{\epsilon}{M_n}\psi_{h_n}(x)$ and $\dot{a}_\epsilon(x)=\frac{1}{M_n}\psi_{h_n}(x)$.
Since $\norm{a_\epsilon'}_\infty \le \frac{c}{M_n h_n}$, $\norm{a_\epsilon''}_\infty \le \frac{c}{M_n h_n^2}$, $\frac{1}{M_n h_n^2}=O(h_n^{\beta-2})=O(1)$ using $\beta \ge 3$, and the definition of $b$ given in \eqref{eq: def drift}, we see that the coefficients in the {\modr SDE} satisfied by $(X_t^\epsilon)_t$ are bounded, together with their first and second order derivatives. By Theorem 2.2.2 in \cite{Nualart}, this implies that the Malliavin derivatives up to order 2 of $X^\epsilon_t$ are bounded (see \eqref{eq: control Malliavin SDE} in Appendix \ref{s: Malliavin sub section recall}).
It yields 
$ \sup_{r\in[0,\Delta_n]} \E[\sup_{t\in[0,\Delta_n]}
\abs{D_rX_t^\epsilon}^p ] \le c(p)$,   and $ \sup_{r,r'\in[0,\Delta_n]} \E[\sup_{t\in[0,\Delta_n]}
\abs{D^2_{r,r'} X_t^\epsilon}^p ] \le c(p)$ for all $p \ge 2$, where the constant $c(p)$ does not depend on $\epsilon,n$.
To get a control on the Malliavin derivative of $\dot{X}^\epsilon_t$, we use  Theorem 2.2.1 in \cite{Nualart}, 
to obtain that the Malliavin derivatives of $(X_t^\epsilon,\dot{X}_t^\epsilon)$ are solution of the following {\modr SDE}, where $0\le  r \le t \le \Delta_n$,
\begin{multline}\label{eq: SDE DX DX dot with b}
	\begin{bmatrix}
		D_rX_t^\epsilon	\\
		D_r \dot{X}_t^\epsilon		
	\end{bmatrix}=
\begin{bmatrix}
		a_\epsilon(X^\epsilon_r)\\
		A_{\epsilon}(X_s^\epsilon, \dot{X}_s^\epsilon) 
\end{bmatrix}
+
	\int_r^t 
	\begin{bmatrix}
	 b'(X_s^\epsilon) D_r X_s^\epsilon  	
		\\
		\partial_x B(X_s^\epsilon, \dot{X}_s^\epsilon) D_r X_s^\epsilon + \partial_v B(X_s^\epsilon, \dot{X}_s^\epsilon) D_r \dot{X}_s^\epsilon  	ds
	\end{bmatrix}
	ds
	\\
	+
	\int_r^t
	\begin{bmatrix}
	 a'_{\epsilon}(X_s^\epsilon) D_r X_s^\epsilon  	
		\\
		\partial_x A_{\epsilon}(X_s^\epsilon, \dot{X}_s^\epsilon) D_r X_s^\epsilon + \partial_v A_{\epsilon}(X_s^\epsilon, \dot{X}_s^\epsilon) D_r \dot{X}_s^\epsilon  	ds
	\end{bmatrix}
	dB_s,
\end{multline}
with  $B(x,v)=b'(x)v$ and  $A_{\epsilon}(x,v)=\dot{a}_\epsilon(x)+ a'_\epsilon(x) v=\frac{1}{M_n}\psi_{h_n}(x)+\frac{\epsilon}{M_n}\psi'_{h_n}(x)v$.
Using $\frac{1}{M_n h_n^2}=O( h_n^{\beta - 2})=O(1)$, we have 
	that $\norm{\partial_v A_{\epsilon}}_\infty +  \norm{\partial_v B}_\infty \le c$ and $|\partial_x A_{\epsilon}(x,v)| + |\partial_x B(x,v)| \le c(1+|v|)$ for some constant $c$ independent of $\epsilon,n$. 
We apply Lemma 2.2.1 in \cite{Nualart} on the second component of the {\modr SDE} \eqref{eq: SDE DX DX dot with b} and deduce
	$\sup_{r\in[0,\Delta_n]} \E[\sup_{t\in[0,\Delta_n]}
	\abs{D_r\dot{X}_t^\epsilon}^p ] \le c(p)$. 
It is sufficient to infer that
$ \norm{D_\cdot X^\epsilon_{\Delta_n}}_{\mathbb{D}^{1,16}(H)} \le c$; $ \norm{X^\epsilon_{\Delta_n}}_{\mathbb{D}^{2,16}(\R)}\le c$ and $\norm{\dot{X}^\epsilon_{\Delta_n} }_{\mathbb{D}^{1,16}(\R)}\le c$. 

It remains to prove that $\norm{\frac{1}{\bracket{DX^\epsilon_\cdot}{DX^\epsilon_\cdot}}}_{\mathbb{D}^{1,8}(\R)}=O(\Delta_n^{-1})$.
Using Proposition 2.1 in \cite{Nualart04}, we write the explicit representation for $D_rX_{\Delta_n}$ available in the univariate case,
\begin{equation}\label{eq : representation DX explicit}
D_r{X}_{\Delta_n}^\epsilon=a_\epsilon(X_{\Delta_n}^\epsilon)\exp\left( \int_r^{\Delta_n} [b'(X^\epsilon_u)-\frac{a'_\epsilon}{a_\epsilon}b(X^\epsilon_u)-\frac{1}{2}a''_\epsilon a_\epsilon(X^\epsilon_u)]du\right), \text{ where $a_\epsilon=1+\frac{\epsilon}{M_n} \psi_{h_n}$.}
\end{equation}
 Using the boundedness of $1/(M_nh_n^2)$ and $1/a_\epsilon$ we deduce that  $1/c \le |D_r{X}^\epsilon_{\Delta_n}| \le c$ for some constant $c$. In turn, 
 $\bracket{D_\cdot X^\epsilon_{\Delta_n}}{D_\cdot  X^\epsilon_{\Delta_n}} =\int_0^{\Delta_n} |D_r{X}_{\Delta_n}^\epsilon|^2 dr \ge c \Delta_n$ for some $c>0$.
 By the chain rule property for the Malliavin derivative, see Proposition 1.2.3 in \cite{Nualart}, we have $D(\frac{1}{\bracket{D_\cdot X^\epsilon_{\Delta_n}}{D_\cdot  X^\epsilon_{\Delta_n}}})=-\frac{D (\bracket{D_\cdot X^\epsilon_{\Delta_n}}{D_\cdot  X^\epsilon_{\Delta_n}})}{\bracket{D_\cdot X^\epsilon_{\Delta_n}}{D_\cdot  X^\epsilon_{\Delta_n}}}$. Therefore, $\norm{\frac{1}{\bracket{D_\cdot X^\epsilon_{\Delta_n}}{D_\cdot  X^\epsilon_{\Delta_n}}}}_{\mathbb{D}^{1,8}(\mathbb{R})} \le 
 \frac{c}{\Delta_n^2} \norm{\bracket{D_\cdot X^\epsilon_{\Delta_n}}{D_\cdot  X^\epsilon_{\Delta_n}}}_{\mathbb{D}^{1,8}(\mathbb{R})}$. We write
 \begin{equation*}
  D_u(\bracket{D_\cdot X^\epsilon_{\Delta_n}}{D_\cdot  X^\epsilon_{\Delta_n}})=2\int_0^{\Delta_n} D^2_{u,r}X^\epsilon_{\Delta_n} D_{r}X^\epsilon_{\Delta_n} dr	
 \end{equation*}
 and use that  $|D_{r}X^\epsilon_{\Delta_n}| \le  c$, by the representation \eqref{eq : representation DX explicit}, to deduce 
 \begin{equation*}
 	\int_0^{\Delta_n}
  D_u(\bracket{D_\cdot X^\epsilon_{\Delta_n}}{D_\cdot  X^\epsilon_{\Delta_n}})^2du \le c \Delta_n \int_0^{\Delta_n}\int_0^{\Delta_n}
  |D^2_{u,s}X^\epsilon_{\Delta_n} |^2duds.
\end{equation*}
It entails $\norm{\bracket{D_\cdot X^\epsilon_{\Delta_n}}{D_\cdot  X^\epsilon_{\Delta_n}}}_{\mathbb{D}^{1,8}(\mathbb{R})} \le c \Delta_n
\norm{X^\epsilon_{\Delta_n}}_{\mathbb{D}^{2,8}(\mathbb{R})}=O(\Delta_n)$ and in turn 
 $\norm{\frac{1}{\bracket{D_\cdot X^\epsilon_{\Delta_n}}{D_\cdot  X^\epsilon_{\Delta_n}}}}_{\mathbb{D}^{1,8}(\mathbb{R})}=O(1/\Delta_n)$. The lemma is proved.
\qed
}
{\modc
\subsection{Proof of Lemma \ref{l: bound stationarity}}
\begin{proof}
The proof of the three points relies on Lemma \ref{l: AB} and on the stationarity of the process $(X_s^\epsilon)_{s \ge 0}$. We recall that on $\Tilde{\Omega}_n$ it is ${\revnotat \widehat{X}}^\epsilon = X^\epsilon$ $\forall \epsilon > 0$ while on the {\revar complement} the following bound holds: $\mathbb{P}(\Omega_n^c)= o(n^{- r})$. Hence, using also the boundedness of $\psi_{h_n}$ we have for any $p \ge 1$
\begin{align*}
&{\modar \sup_{0<s\le\Delta_n}}\int_{\R} \E_{x_0}[|\psi_{h_n}({\revnotat \widehat{X}}_{{\modar s}}^\epsilon)|^p] {\modarn \pi^\epsilon}(x_0) dx_0 \\
&  \le
{\modar \sup_{0<s\le\Delta_n}}
\int_{\R} \E_{x_0}[|\psi_{h_n}({X_{{\modar s}}^\epsilon})|^p] {\modarn \pi^\epsilon}(x_0) dx_0 + o(n^{-r}) \\
& =
{\modar \sup_{0<s\le\Delta_n}}{\modarn \E_{\pi^\epsilon}}[|\psi_{h_n}({X_{{\modar s}}^\epsilon})|^p] + o(n^{-r}).
\end{align*}
Then from the stationarity of the process ${\modarn X^\epsilon}$ it is, for any $p \ge 1$, 
\begin{align*}
{\modar \forall s>0,}\quad
{\modarn \E_{\pi^\epsilon}}[|\psi_{h_n}({X_{{\modar s}}^\epsilon})|^p] = \int_{-h_n}^{h_n} (\psi_{h_n}(y))^p {\modarn \pi^\epsilon}(y) dy \le c h_n.
\end{align*}
We act similarly in order to show the second point of the lemma. We remark we can use Lemma \ref{l: AB} as both the $L^{p'}$ norms of $|\psi_{h_n}^{(k)}({\revnotat \widehat{X}}_{{\modar s}}^\epsilon)|^p$ and $|\psi_{h_n}^{(k)}({X_{{\modar s}}^\epsilon})|^p$ are upper bounded by $h_n^{-kp} = n^{r_0}$ for some $r_0$ and some $p' > 1$. Hence, for $p\ge 1$ and $k \ge 1$ we have 
\begin{equation*}
\int_{\R} \E_{x_0}[|\psi_{h_n}^{(k)}({\revnotat \widehat{X}}_{{\modar s}}^\epsilon)|^p] {\modarn \pi^\epsilon}(x_0) dx_0 \\
 \le {\modarn \E_{\pi^\epsilon}}[|\psi_{h_n}^{(k)}({X_{{\modar s}}^\epsilon})|^p] + o( n^{-r}).
\end{equation*}
Regarding the first term we observe it is
\begin{align*}
{\modarn \E_{\pi^\epsilon}}[|\psi_{h_n}^{(k)}({X_{{\modar s}}^\epsilon})|^p] = \int_{-h_n}^{h_n} |\psi_{h_n}^{(k)}(y)|^p {\modarn \pi^\epsilon}(y) dy \le c  h_n^{1 - kp},
\end{align*}
while the second is negligible, up to choose an $r$ which is large enough. \\
We are left to show the third point of the lemma. The idea is once again to move back to the stationary process. To do that, Lemma \ref{l: AB} comes in handy one more time. Its applicability is ensured by the boundedness of $\psi_{h_n}$ and the fact that both ${\revnotat \widehat{X}}_{u}^\epsilon - {\revnotat \widehat{X}}_{0}^\epsilon$ and ${X_{u}^\epsilon} - {X_{0}^\epsilon}$ have bounded moments of any order. Then, 
\begin{align*}
&\int_0^{\Delta_n} \int_{\R} \E_{x_0}[({\revnotat \widehat{X}}_{u}^\epsilon - {\revnotat \widehat{X}}_{0}^\epsilon)^2 \psi_{h_n}^2({\revnotat \widehat{X}}_{u}^\epsilon)] du \, \pi (x_0) d{x_0} \\
& \le \int_0^{\Delta_n} \int_{\R} \E_{x_0}[({X}_{u}^\epsilon - {X}_{0}^\epsilon)^2 \psi_{h_n}^2({X}_{u}^\epsilon)] du \, \pi (x_0) d{x_0} + o(n^{- r}) \\
& = \int_0^{\Delta_n}  {\modarn \E_{\pi^\epsilon}}[({X}_{u}^\epsilon - {X}_{0}^\epsilon)^2 \psi_{h_n}^2({X}_{u}^\epsilon)] du + o(n^{- r}) \\
&  = \int_0^{\Delta_n}  {\modarn \E_{\pi^\epsilon}}[({X}_{0}^\epsilon - {X}_{u}^\epsilon)^2 \psi_{h_n}^2({X}_{0}^\epsilon)] du + o(n^{- r}), 
\end{align*}
where we have used that the diffusion is reversible. Introducing the conditional expectation with respect to $X_0^\epsilon$ we obtain that the integral here above is equal to
\begin{align*}
&\int_0^{\Delta_n}  {\modarn \E_{\pi^\epsilon}}[\E[({X}_{0}^\epsilon - {X}_{u}^\epsilon)^2| {X}_{0}^\epsilon] \psi_{h_n}^2({X}_{0}^\epsilon)] du + o(n^{- r}) \\
& \le c \int_0^{\Delta_n}  {\modarn \E_{\pi^\epsilon}}[u \, \psi_{h_n}^2({X}_{0}^\epsilon)] du + o(n^{- r}) \\
& \le c \Delta_n^2 h_n + o(n^{- r}), 
\end{align*}
as we wanted.

\end{proof}
}

\subsection{Proof of Proposition \ref{p: bound W1}}
In order to get an expansion for ${\revnotat \widehat{W}}^1_{x_0,\Delta_n,\epsilon}$ we need asymptotic controls on the Malliavin derivatives of the process ${\revnotat \widehat{X}}$. It is the purpose of the next Section to collect some properties on $D{\revnotat \widehat{X}}$, that will be useful for the proof of Proposition \ref{p: bound W1} and Lemmas \ref{l: bound moments Xdot}--\ref{L: suppressing dB in principal term}. 

\subsubsection{Controls on $D{\revnotat \widehat{X}}$}
First, we focus on $Y^\epsilon$ given explicitly by \eqref{eq: DD exp}. Let us define
\begin{equation} \label{eq: def I ronde}
	\mathcal{I}_t=\int_0^t \psi_{h_n}'({\revnotat \widehat{X}}_s^\epsilon)dB_s=\int_0^t \frac{\psi_{h_n}'({\revnotat \widehat{X}}_s^\epsilon)}{a_\epsilon ({\revnotat \widehat{X}}_s^\epsilon) } a_\epsilon ({\revnotat \widehat{X}}_s^\epsilon) dB_s,
\end{equation}	
where we recall that $a_\epsilon ({\revnotat \widehat{X}}_s^\epsilon)$ is the volatility of the process ${\revnotat \widehat{X}}^\epsilon$ appearing in \eqref{eq: model epsilon}, i.e. $a_\epsilon (X_s^\epsilon) = 1 + \frac{\epsilon}{M_n} \psi_{h_n}({\revnotat \widehat{X}}_s^\epsilon)$. We denote as $\Xi$ a primitive of the function $\frac{\psi_{h_n}'(y)}{a_\epsilon (y) }$ which is null at $0$. From Ito formula it follows
\begin{align}{\label{eq: ito}}
{\modarn \mathcal{I}_t= }	\int_0^t \psi_{h_n}'({\revnotat \widehat{X}}_s^\epsilon)dB_s = \Xi({\modar {\revnotat \widehat{X}}^\epsilon_t}) - \Xi({\modar {\revnotat \widehat{X}}^\epsilon_0}) - \frac{1}{2} \int_0^t (\frac{\psi_{h_n}'}{a_\epsilon })' ({\revnotat \widehat{X}}_s^\epsilon) a_\epsilon^2({\revnotat \widehat{X}}_s^\epsilon) ds.
\end{align}
We now observe that
\begin{align*}
	|\Xi(u)| & = |\int_0^u \frac{\psi_{h_n}'(y)}{(1 + \frac{\epsilon}{M_n} \psi_{h_n}  (y)) } dy| \\
	& \le \int_0^u |\psi_{h_n}'(y)| dy \\
	& = \int_0^u \frac{1}{h_n} |\psi'(\frac{y}{h_n})| dy \\
	& \le \int_{\mathbb{R}}|\psi'(\tilde{y})| d\tilde{y} < \infty.
\end{align*}
In order to bound the last term in the right hand side of \eqref{eq: ito} we remark the following estimations hold true:
\begin{equation}
	|\psi_{h_n}'(y)| \le \frac{c}{h_n}, \qquad |\psi_{h_n}''(y)| \le \frac{c}{h_n^2}, \qquad |a_\epsilon'(y)| \le \frac{1}{h_n M_n}.
	\label{eq: bound psi and a}
\end{equation}
It follows, using also the fact that $Y^\epsilon$ is explicit as in \eqref{eq: DD exp}
{\modarn 
	\begin{equation}\label{E : maj Y 3 terms}	
	Y_t^\epsilon \le \exp(\frac{c}{M_n} + \frac{c \, t}{h_n^2 M_n} + \frac{c \, t}{h_n^2 M_n^2}).
	\end{equation}
}
In the same way we have an analogous upper bound for $(Y_t^\epsilon)^{- 1}$, which provides 
\begin{equation}
	\sup_{t \in [0, \Delta]} |Y_t^\epsilon| + |Y_t^\epsilon|^{-1} \le 2 \exp(\frac{c \, \Delta_n}{h_n^2 M_n}) < \infty,
	\label{eq: Y bounded}
\end{equation}
{\modarn where we used $h_n \to 0$, $M_n \to \infty$ and that {\revar $h_n \le \Delta_n$,} 
	{\revar so that} 
	in consequence
$\frac{\Delta_n}{h_n^2 M_n}$ dominates the three terms in the exponential of \eqref{E : maj Y 3 terms}.}

Following the same reasoning we have used in order to get \eqref{eq: Y bounded} and having in mind the explicit expression for $Y^\epsilon$ and $(Y^\epsilon)^{- 1}$ {\revar given by \eqref{eq: DD exp} } it is easy to see that 
{\modar 
	\begin{align} \label{eq: control Y}
		|Y_{\Delta_n}^\epsilon - 1| &\le \frac{c \Delta_n}{M_n h_n^2},
		\\ \label{eq: control Y inv}
		\sup_{u \in [0, \Delta_n]}|(Y_u^\epsilon)^{-1}- 1| &\le \frac{c \Delta_n}{M_n h_n^2}.
	\end{align}
}
From \eqref{eq: DsXt}, we deduce that  {\modar for all $0<u<s<\Delta_n$,}
{\modar\begin{equation}\label{eq: majo DX}
{\modarn 0< ~}	c \le  c(1 - \frac{1}{h_n M_n}) \le	|D_u {\revnotat \widehat{X}}_s^\epsilon| \le c'(1 + \frac{1}{h_n M_n}) \le c'.
\end{equation}}
In turn, we have {\modarn simple bounds on} the Malliavin {\revar bracket, 
 $\forall r\in (0,\Delta_n]$,}
	\begin{equation}\label{eq: mino bracket}
		{\revar
		c r  \le \bracket{D_\cdot {\revnotat \widehat{X}}_{r}^\epsilon}{D_\cdot {\revnotat \widehat{X}}_{r}^\epsilon}
				\le c' r .}
	\end{equation}

{\modar \subsubsection{Proof of \eqref{eq: control Rn integrated}}
\begin{proof}
{\revar 	Let us denote  
		\begin{equation} \label{eq : def grand Phi}
		\Phi_{\Delta_n}=\sqrt{\Delta_n}\delta\left(
	\frac{D_.{\revnotat \widehat{X}}^\epsilon_{\Delta_n}}{\bracket{D_.{\revnotat \widehat{X}}^\epsilon_{\Delta_n}}{D_.{\revnotat \widehat{X}}^\epsilon_{\Delta_n}}}\right)	.
\end{equation}}	
	$\bullet$ We start by proving for all $p \ge 2$
	\begin{equation} \label{eq: control Phi B}
		\sup_{x_0\in\mathbb{R}}	\E_{x_0}\left[\abs{\Phi_\Delta-\frac{B_{\Delta_n}}{\sqrt{\Delta_n}}}^p \right]
		\le \frac{c(p)}{(M_n h_n^2)^p}.
	\end{equation}
{\revar Using Proposition 1.3.3. in Nualart \cite{Nualart} with the notation 
 $L=-\delta \circ D$ for the so-called Ornstein-Uhlenbeck operator,}
we have
	\begin{equation}
		\label{eq: Phi second express}
		\Phi_{\Delta_n}=\frac{-\Delta_n^{1/2}L({\revnotat \widehat{X}}_{\Delta_n}^\epsilon)}{\bracket{D_.{\revnotat \widehat{X}}^{\epsilon}_{\Delta_n}}{D_.{\revnotat \widehat{X}}^{\epsilon}_{\Delta_n}}}+
		\frac{\Delta_n^{1/2} \bracket{ D_.\left(\bracket{D_.{\revnotat \widehat{X}}^{\epsilon}_{\Delta_n}}{D_.{\revnotat \widehat{X}}^{\epsilon}_{\Delta_n}}\right)}{D_.{\revnotat \widehat{X}}^{\epsilon}_{\Delta_n}}}{\norm{D_.{\revnotat \widehat{X}}^{\epsilon}_{\Delta_n}}_H^4}{\revar .}
	\end{equation}
Application of the {\revar linear} operator $L$ to the dynamic {\modarn \eqref{eq: SDE sans drift},
	 recalling $a_\epsilon=1+\epsilon\frac{\psi_{h_n}}{M_n}$,}	
	 yields to
	\begin{equation}\label{eq: linear SDE L}
		L({\revnotat \widehat{X}}_s^\epsilon)=\int_0^s \{a'_\epsilon({\revnotat \widehat{X}}^\epsilon_r) L({\revnotat \widehat{X}}^\epsilon_r) + a''_\epsilon({\revnotat \widehat{X}}_r) \bracket{D_.{\revnotat \widehat{X}}^\epsilon_r}{D_.{\revnotat \widehat{X}}^\epsilon_r} \}d B_r
		- \int_0^s a_\epsilon({\revnotat \widehat{X}}_r^{\epsilon}) d B_r,
	\end{equation}
	where we used the property $L(\int_0^s u_r dB_r)=\int_0^s L(u_r) dB_r - \int_0^s u_r dB_r$ for $(u_r)_r$ an adapted process taking values in the space of smooth random variables in the Malliavin sense {\revar together with Proposition 1.4.5 in \cite{Nualart}. A rigorous justification of \eqref{eq: linear SDE L} is given by Theorem 10.3 of \cite{Bichteler_et_al87}.}
	{\modarn We deduce that for $p\ge 2$,
	\begin{align*}
		\E_{x_0}[|L({\revnotat \widehat{X}}^\epsilon_s)|^p]
		\le c_p \norm{a'_\epsilon}_\infty^p s^{p/2-1}&\int_0^s \E_{x_0}[|L({\revnotat \widehat{X}}^\epsilon_r)|^p dr 
		\\ &+
		c_p \norm{a''_\epsilon}_\infty^p s^{p/2-1}\int_0^s \E_{x_0}[\norm{D_.{\revnotat \widehat{X}}_r^\epsilon}_H^{2p}
		]dr
		+
		c_p \norm{a''_\epsilon}_\infty^p s^{p/2}.
	\end{align*}
Remark that $\norm{a_\epsilon}_\infty \le c$, 
$\norm{a'_\epsilon}_\infty \le C/(M_nh_n) \le C$ and 
$\norm{a''_\epsilon}_\infty \le C/(M_nh_n^2) \le C$, as {\modrev $1/M_n = \alpha_0 h_n^\beta$} with $\beta\ge3$.}
	Using {\modarn \eqref{eq: mino bracket},} 
	{\modarn$	 \E_{x_0}[|L({\revnotat \widehat{X}}^\epsilon_s)|^p]
	\le c_p  s^{p/2-1}\int_0^s \E_{x_0}[|L({\revnotat \widehat{X}}^\epsilon_r)|^p]dr  + c_p s^{p/2}$} and thus, by Gronwall inequality,
	$$
	\E_{x_0}[|L({\revnotat \widehat{X}}^\epsilon_s)|^p] \le c_p s^{p/2}.
	$$
	From \eqref{eq: linear SDE L} and the expression of $a_\epsilon$ we deduce, using Burkholder--Davis--Gundy inequality
	\begin{align*}	
		{\modarn \E_{x_0}} \left[ \abs{L({\revnotat \widehat{X}}^\epsilon_{\Delta_n})+B_{\Delta_n}}^p \right]
		&\le c
		{\modarn \E_{x_0}} \left[ \abs{\int_0^{\Delta_n} a'_\varepsilon({\revnotat \widehat{X}}_r^{\epsilon}) L({\revnotat \widehat{X}}^\epsilon_r) dB_r}^p
		\right] + \\
		&\qquad\qquad\qquad
		c {\modarn \E_{x_0}} \left[ \abs{\int_0^{\Delta_n} a''_\varepsilon({\revnotat \widehat{X}}_r^{\epsilon}) \bracket{D_.{\revnotat \widehat{X}}_r^\epsilon}{D_.{\revnotat \widehat{X}}_r^\epsilon} dB_r}^p
		\right] 
		+
		c {\modarn \E_{x_0}} \left[ \abs{ \int_0^{\Delta_n} 
			\frac{\epsilon}{M_n} \psi({\revnotat \widehat{X}}^\epsilon_r) dB_r}^p
		\right]
		\\
		&\le c \norm{{\modarn a'_\epsilon}}_\infty^p \Delta_n^{p/2-1} \int_0^{\Delta_n}
		{\modarn \E_{x_0}}[|L({\revnotat \widehat{X}}^\epsilon_r)|^p] dr +
		\\
		&\qquad\qquad\qquad
		c \norm{{\modarn a''_\epsilon}}_\infty^p \Delta_n^{p/2-1} \int_0^{\Delta_n}
		{\modarn \E_{x_0}}[
		\norm{D_.{\revnotat \widehat{X}}_r^\epsilon}_H^{2p}
		] 
		dr
		+ c\frac{\norm{\psi}^p_\infty}{M_n^p} \Delta_n^{p/2}
		\\
		&\le c \left( \frac{\Delta_n^p}{M_n^p h_n^p} +
		\frac{\Delta_n^{3p/2}}{M_n^p h_n^{2p}}+ \frac{\Delta_n^{p/2}}{M_n^p} \right)
		{\modarn \le
		c \left( \frac{\Delta_n^p}{M_n^p h_n^p} +
		\frac{\Delta_n^{3p/2}}{M_n^p h_n^{2p}}\right),}
	\end{align*}
{\modarn where in the last line we used $h_n\le\Delta_n$.}
	From \eqref{eq: mino bracket}, we deduce
	\begin{equation*} 
		{\modarn\E_{x_0}} \left[ \abs{\frac{L({\revnotat \widehat{X}}^\epsilon_{\Delta_n})+B_{\Delta_n}}{
				\bracket{D_. {\revnotat \widehat{X}}^\epsilon_{\Delta_n}}{D_. {\revnotat \widehat{X}}^\epsilon_{\Delta_n}}  }}^p \right] \le c
		\left( \frac{1}{M_n^p h_n^p} +
		\frac{\Delta_n^{p/2}}{M_n^p h_n^{2p}} \right).
	\end{equation*}
	Using that from \eqref{eq: control Y}--\eqref{eq: control Y inv} we have,
	{\modarn 
		\begin{equation} \label{eq: bracket moins 1}
	\abs{ \frac{\Delta_n}{\bracket{D_. {\revnotat \widehat{X}}_{\Delta_n}^{\epsilon}}{D_. {\revnotat \widehat{X}}_{\Delta_n}^{\epsilon}}} -1}\le c\frac{\Delta_n}{M_n h_n^2}, 
	\end{equation}}
	it is deduced
	\begin{equation}\label{eq: control Phi B part 1}
		{\modarn \E_{x_0}} \left[ \abs{
			\frac{\sqrt{\Delta_n}L({\revnotat \widehat{X}}^\epsilon_{\Delta_n})}{
				\bracket{D_. {\revnotat \widehat{X}}^\epsilon_{\Delta_n}}{D_. {\revnotat \widehat{X}}^\epsilon_{\Delta_n}}}  
			+ \frac{B_{\Delta_n}}{\sqrt{\Delta_n}}
		}^p \right] \le c
		\left( \frac{\Delta_n^{p/2}}{M_n^p h_n^p} +
		\frac{\Delta_n^{p}}{M_n^p h_n^{2p}} \right).
	\end{equation}
	
	We consider now the second term of \eqref{eq: Phi second express}.  From \eqref{eq: SDE sans drift} {\revar and Theorem 2.2.2. in Nualart \cite{Nualart},} we derive the dynamics of the second Malliavin derivative of $X^\epsilon$, 
	\begin{equation*}
		D_{s_1,s_2}{\revnotat \widehat{X}}^\epsilon_t=\int_{s_1 \vee s_2}^t \{a''_\epsilon D_{s_1}{\revnotat \widehat{X}}^\epsilon_r
		D_{s_2}{\revnotat \widehat{X}}^\epsilon_r + a'_\epsilon({\revnotat \widehat{X}}_u)D_{s_1,s_2}{\revnotat \widehat{X}}^\epsilon_r
		\} dB_r \\+ a'_\epsilon({\revnotat \widehat{X}}_{s_2}^\epsilon) D_{s_1}{\revnotat \widehat{X}}_{s_2}^\epsilon
		+ a'_\epsilon({\revnotat \widehat{X}}_{s_1}^\epsilon) D_{s_2}{\revnotat \widehat{X}}_{s_1}^\epsilon,
	\end{equation*}
	for $s_1 \vee s_2<t$. As a consequence, we have $\sup_{s_1\vee s_2\le  \Delta_n}{\modarn \E_{x_0}}\left[|D_{s_1,s_2}{\revnotat \widehat{X}}^\epsilon_{\Delta_n}|^p\right] \le c \norm{a_\epsilon'}_{\infty}^p \le c/(M_n h_n)^p$. Next, using Cauchy-Schwarz inequality for the Malliavin bracket, we have
	\begin{align*}
{\revar 	\frac{\Delta_n^{1/2}
			\abs{\bracket{D_.\left(\bracket{D_.{\revnotat \widehat{X}}^{\epsilon}_{\Delta_n}}{D_.{\revnotat \widehat{X}}^{\epsilon}_{\Delta_n}}\right)}{D_.{\revnotat \widehat{X}}^{\epsilon}_{\Delta_n}}}}
		{\norm{D_.{\revnotat \widehat{X}}^{\epsilon}_{\Delta_n}}^4_H}
	}
		&\le		
{\revar 		\frac{\Delta_n^{1/2}
			\norm{ D_.\left( \bracket{D_.{\revnotat \widehat{X}}^{\epsilon}_{\Delta_n}}{D_.{\revnotat \widehat{X}}^{\epsilon}_{\Delta_n}}\right)}_H} 
	{\norm{D_.{\revnotat \widehat{X}}^{\epsilon}_{\Delta_n}}_H^{3}}
}		
		\\
				&\le \frac{{\modarn c}}{\Delta^{1/2}_n}\left(\int_0^{\Delta_n} \int_0^{\Delta_n} 
		\abs{D_{s_1,s_2} {\revnotat \widehat{X}}^\epsilon_{\Delta_n}}^2\abs{D_{s_1} {\revnotat \widehat{X}}^\epsilon_{\Delta_n}}^2
		ds_1 ds_2\right)^{1/2},
	\end{align*} 
	where we used \eqref{eq: mino bracket} and Jensen inequality in the second line. Using \eqref{eq: majo DX} and the upper bound on the second order Malliavin derivative of $X^\epsilon_\Delta$ we deduce
	\begin{equation} \label{eq: control Phi B part 2}
		{\modarn \E_{x_0}}\left[\abs{\frac{\Delta_n^{1/2}
				\bracket{D_.\left(\bracket{D_.{\revnotat \widehat{X}}^{\epsilon}_{\Delta_n}}{D_.{\revnotat \widehat{X}}^{\epsilon}_{\Delta_n}}\right)}{D_.{\revnotat \widehat{X}}^{\epsilon}_{\Delta_n}}}
		{\norm{D_.{\revnotat \widehat{X}}^{\epsilon}_{\Delta_n}}_H^4}}^{p}
		\right]
		\le c\frac{\Delta_n^{p/2}}{(M_n h_n)^p} {\modarn .}
	\end{equation}
	{\revar Collecting  \eqref{eq: control Phi B part 1}--\eqref{eq: control Phi B part 2} with \eqref{eq: Phi second express}, we deduce \eqref{eq: control Phi B}.}  
	
	$\bullet$ We now prove \eqref{eq: control Rn integrated}. From
	\eqref{eq: split W} {\revar and \eqref{eq : def grand Phi}} we have
	\begin{equation*}
	R_n^{(1)}=\left( 
	\frac{{\revnotat \smash{\dot{\widehat{X}}}}^\epsilon_{\Delta_n}}{\Delta_n^{1/2}}
	- {\modarn \frac{1}{\Delta_n^{1/2}M_n}} \int_0^{\Delta_n}
	\psi_{h_n}({\revnotat \widehat{X}}^\epsilon_s) dB_s
	\right)\frac{B_{\Delta_n}}{\Delta_n^{1/2}}
	+
	\frac{{\revnotat \smash{\dot{\widehat{X}}}}^\epsilon_{\Delta_n}}{\Delta_n^{1/2}}\left(
	\Phi_{\Delta}-\frac{B_{\Delta_n}}{\Delta_n^{1/2}}\right)
{\modarn 	=:R_n^{(1,2)}+R_n^{{\revar (1,2)}}.}
	\end{equation*}
By Cauchy--Schwarz and \eqref{eq: def x dot} we have
\begin{align*}
	{\modarn \E_{x_0}}\left[|R_n^{{\modarn (1,1)}}|^2\right]
	&\le\Delta_n^{-1}{\modarn \E_{x_0}}\left[
	\left(\int_0^{\Delta_n} \frac{\epsilon}{M_n}{\modarn \psi'_{h_n}}({\revnotat \widehat{X}}_s^\epsilon) {\revnotat \smash{\dot{\widehat{X}}}}^\epsilon_s dB_s \right)^4 \right]^{1/2}	
\\	&\le
	\frac{1}{\Delta^{1/2}M_n^2}\left(
	\int_0^{\Delta_n} 
	{\modarn \E_{x_0}}\left[ |{\modarn \psi'_{h_n}}({\revnotat \widehat{X}}_s^\epsilon)|^4 |{\revnotat \smash{\dot{\widehat{X}}}}^\epsilon_s|^4 \right] ds  \right)^{1/2}	
\end{align*}
	Using that by Lemma \ref{l: bound moments Xdot} we have ${\modarn \E_{x_0}}[|{\revnotat \smash{\dot{\widehat{X}}}}_s^\epsilon|^4 ]\le c \Delta_n^2/M_n^4$ and
	$\norm{{\modarn \psi'_{h_n}}}_\infty \le c/(h_n M_n)$ we deduce,
	${\modarn \E_{x_0}}\left[|R_n^{{\modarn (1,1)}}|^2\right] \le c {\modarn \Delta_n}/(M_n^6h_n^4)  $.

Using Cauchy-Schwarz and \eqref{eq: control Phi B}, we have 	${\modarn \E_{x_0}}\left[|R_n^{{\modarn(1,2)}}|^2\right] \le c (1/(\Delta_n M_n^2 h_n^4)) {\modarn \E_{x_0}}{\revar
	 [|{\revnotat \smash{\dot{\widehat{X}}}}_s^\epsilon|^4 ]^{1/2}} \le  c /(M_n^4 h_n^4)$.

 Collecting the upper bounds on $R_n^{{\revar (1,1)}}$ and $R_n^{{\revar (1,2)}}$, we deduce ${\modarn \E_{x_0}}\left[|R_n^{{\revar (1)}}|^2\right] \le c/(M_n^4 h_n^4)$ and recalling \eqref{eq: control Rn integrated}, the proposition is proved {\revar  as $1/(M_n^2h_n^5)= O(h_n^{2\beta-5})=o(1)$, using $\beta\ge3$.}
\end{proof}
}
\subsection{Proof of Lemma \ref{l: bound moments Xdot}}
\begin{proof}
{\revar Using 
\eqref{eq: def x dot},} it is, for any $p \ge 2$, {\revar and $u\in[0,\Delta_n]$,}
	\begin{align*}
		\E_{\modarn x_0}[|{\revnotat \smash{\dot{\widehat{X}}}}_{\revar u}^\epsilon|^p] & \le c  \E_{\modarn x_0}[( \int_0^{\revar u} (\frac{\epsilon}{M_n} \psi_{h_n}'({\revnotat \widehat{X}}_s^\epsilon) {\revnotat \smash{\dot{\widehat{X}}}}^\epsilon_s)^2 ds)^{\frac{p}{2}}] + c \E_{\modarn x_0}[ (\int_0^{\revar u} (\frac{1}{M_n} \psi_{h_n}({\revnotat \widehat{X}}_s^\epsilon))^2 ds)^{\frac{p}{2}}] \\
		& \le c \Delta_n^{\frac{p}{2}- 1} \int_0^{\revar u}\E_{\modarn x_0}[(\frac{\epsilon}{M_n} {\revar | }\psi_{h_n}'({\revnotat \widehat{X}}_s^\epsilon) {\revnotat \smash{\dot{\widehat{X}}}}^\epsilon_s{\revar | })^p] ds + c \frac{\Delta_n^{\frac{p}{2}- 1}}{M_n^p} \int_0^{\revar u} \E_{\modarn x_0}[{\revar | }\psi_{h_n}({\revnotat \widehat{X}}_s^\epsilon){\revar | }^p] ds \\
		& \le c \frac{\Delta_n^{\frac{p}{2}- 1}}{(M_n h_n)^p} \int_0^{\revar u} \E_{\modarn x_0}[|{\revnotat \smash{\dot{\widehat{X}}}}^\epsilon_s|^p] ds + 
		{\modar 
			c \frac{\Delta_n^{\frac{p}{2}- 1}}{M_n^p} \int_0^{\revar u} {\modarn \E_{x_0}}[{\revar | }\psi_{h_n}({\revnotat \widehat{X}}_s^\epsilon){\revar | }^p] ds }	,
	\end{align*}
	where we have used Burkholder-Davis-Gundy and Jensen inequalities {\modarn with \eqref{eq: bound psi and a}}. 
	{\revar Let $M_{\revar u} =\sup_{\revar s \le u}{\modarn\E_{x_0}}[|{\revnotat \smash{\dot{\widehat{X}}}}^\epsilon_{\revar s}|^p]$. Then, from above it follows }
	$$M_{\revar u} \le c {\modar 
		\frac{\Delta_n^{\frac{p}{2}- 1}}{M_n^p} \int_0^{\revar u} {\modarn\E_{x_0}}[|\psi_{h_n}({\revnotat \widehat{X}}_s^\epsilon)|^p] ds }	 + c \frac{\Delta_n^{\frac{p}{2}- 1}}{(M_n h_n)^p} \int_0^{\revar u} M_s ds.$$
	Using Gronwall lemma, it yields 
	\begin{equation} \label{eq: control Xdot intermediaire}
	M_{\Delta_n} \le c\exp( c \frac{\Delta_n^{\frac{p}{2}}}{(M_n h_n)^p}) {\modar 
		\frac{\Delta_n^{\frac{p}{2}- 1}}{M_n^p} \int_0^{\Delta_n} {\modarn\E_{x_0}}[|\psi_{h_n}({\revnotat \widehat{X}}_s^\epsilon)|^p] ds }\le {\modar c
		\frac{\Delta_n^{\frac{p}{2}- 1}}{M_n^p} \int_0^{\Delta_n} {\modarn\E_{x_0}}[|\psi_{h_n}({\revnotat \widehat{X}}_s^\epsilon)|^p] ds }, 
	\end{equation}
	recalling that the constant $c$ may change value from line to line and that the quantity $\frac{\Delta_n^{\frac{p}{2}}}{(M_n h_n)^p}$ is bounded as $\Delta_n \rightarrow 0$ and $M_n h_n \rightarrow \infty$ for $n \rightarrow \infty$. {\modar As $\psi$ is a bounded function, this yields \eqref{eq: control Xdot}. Integrating \eqref{eq: control Xdot intermediaire} with respect to ${\modarn \pi^\epsilon}(x_0)dx_0$ and using the first point of Lemma \ref{l: bound stationarity} we deduce  \eqref{eq: control Xdot integrated}.}
\end{proof}

{\modc
\subsection{Proof of Lemma \ref{l: I1 I2}}

\begin{proof}

{\modar To control $I_1$ and $I_2$, defined by \eqref{eq: def I1-I3} we use Cauchy--Schwarz inequality. {\modarn It} 
provides, for $j \in \{1, 2 \}$,}
\begin{align*}
{\revarn \Bigg\lvert}\frac{\bracket{D_\cdot {\revnotat \widehat{X}}_{\Delta_n}^\epsilon}{(D_\cdot {\revnotat \smash{\dot{\widehat{X}}}}_{\Delta_n}^\epsilon)_j}}{\bracket{D_\cdot {\revnotat \widehat{X}}_{\Delta_n}^\epsilon}{D_\cdot {\revnotat \widehat{X}}_{\Delta_n}^\epsilon }} 
{\revarn \Bigg\rvert}
& \le 
\frac{\norm{D_\cdot {\revnotat \widehat{X}}_{\Delta_n}^\epsilon}_H\norm{(D_\cdot {\revnotat \smash{\dot{\widehat{X}}}}_{\Delta_n}^\epsilon)_j}_H}{\bracket{D_\cdot {\revnotat \widehat{X}}_{\Delta_n}^\epsilon}{ D_\cdot {\revnotat \widehat{X}}_{\Delta_n}^\epsilon}}
\\ 
& = \frac{\norm{(D_\cdot {\revnotat \smash{\dot{\widehat{X}}}}_{\Delta_n}^\epsilon)_j}_H}{\norm{D_\cdot {\revnotat \widehat{X}}_{\Delta_n}^\epsilon}_H}.
\end{align*}
Hence, using Jensen inequality, for $j=1, 2$ we have
\begin{align*}
I_j \le \int_{|x_0| \le \Delta_n^{\frac{1}{2} - \gamma}} \E_{x_0} \Bigg[\frac{\bracket{(D_\cdot {\revnotat \smash{\dot{\widehat{X}}}}_{\Delta_n}^\epsilon)_j}{(D_\cdot {\revnotat \smash{\dot{\widehat{X}}}}_{\Delta_n}^\epsilon)_j}}{\bracket{D_\cdot {\revnotat \widehat{X}}_{\Delta_n}^\epsilon}{D_\cdot {\revnotat \widehat{X}}_{\Delta_n}^\epsilon}} \Bigg] {\modarn \pi^\epsilon}(x_0) dx_0
\end{align*}
{\modar Using \eqref{eq: mino bracket}, we need to upper bound }
$\frac{1}{\Delta_n} \int_{|x_0| \le \Delta_n^{\frac{1}{2} - \gamma}} \E_{x_0}\big[\bracket{(D_\cdot {\revnotat \smash{\dot{\widehat{X}}}}_{\Delta_n}^\epsilon)_j}{(D_\cdot {\revnotat \smash{\dot{\widehat{X}}}}_{\Delta_n}^\epsilon)_j}\big] {\modarn \pi^\epsilon}(x_0) dx_0$.
We start considering what happens for $j= 1$. Using Cauchy-Schwarz inequality, the first two points of Lemma \ref{l: bound stationarity} and Lemma \ref{l: bound moments Xdot} we have 
\begin{align*}
I_1 & \le \frac{c}{\Delta_n} 
{\modar
\int_{|x_0|\le \Delta_n^{1/2-\gamma}}
\E_{x_0} [\int_0^{\Delta_n} ((D_s{\revnotat \smash{\dot{\widehat{X}}}}_{\Delta_n}^\epsilon)_1)^2 ds]  {\modarn \pi^\epsilon}(x_0) dx_0} \\
& \le  \frac{c}{\Delta_n} 
{\modar
		\int_{|x_0|\le \Delta_n^{1/2-\gamma}} \E_{x_0} [\int_0^{\Delta_n} 
		(\frac{\epsilon}{M_n} {\modarn |\psi_{h_n}'({\revnotat \widehat{X}}_s^\epsilon) {\revnotat \smash{\dot{\widehat{X}}}}_{s}^\epsilon|} + \frac{1}{M_n} \psi_{h_n}({\revnotat \widehat{X}}_s^\epsilon))^2 ds] {\modarn \pi^\epsilon}(x_0) dx_0}\\
& \le \frac{c}{\Delta_n M_n^2} 
{\modar 
	\int_{|x_0|\le \Delta_n^{1/2-\gamma}} }
\{
\int_0^{\Delta_n} \E_{\modar {x_0}}[{\modarn |\psi_{h_n}'({\revnotat \widehat{X}}_s^\epsilon)|^{{\revar 4}}}]^{{\revar\frac{1}{2}}} \E_{\revar x_0}[({\revnotat \smash{\dot{\widehat{X}}}}_{s}^\epsilon)^{{\revar 4}}]^{{\revar \frac{1}{2}}} ds + \int_0^{\Delta_n} \E_{\modar {x_0}}[(\psi_{h_n}({\revnotat \widehat{X}}_s^\epsilon))^2] ds
 \} {\modar {\modarn \pi^\epsilon}(x_0)dx_0} \\
 & \le \frac{c}{\Delta_n M_n^2} 
 {\modar 
 	\int_{|x_0|\le \Delta_n^{1/2-\gamma}} }
 \{
 \int_0^{\Delta_n} \E_{\modar {x_0}}[{\modarn |\psi_{h_n}'({\revnotat \widehat{X}}_s^\epsilon)|^{{\revar 4}}}]^{{\revar \frac{1}{2}}} {\modar \frac{\Delta_n}{M_n^2}}ds
 \\
 & \qquad\qquad\qquad \qquad\qquad\qquad\qquad\qquad 
  + \int_0^{\Delta_n} \E_{\modar {x_0}}[(\psi_{h_n}({\revnotat \widehat{X}}_s^\epsilon))^2] ds
 \} {\modar {\modarn \pi^\epsilon}(x_0)dx_0}, 
 { \revar \text{ \hfill \quad [by \eqref{eq: control Xdot},]}  }
 \\
 & \le \frac{c}{\Delta_n M_n^2} \{\Delta_n ({\revar h^{ 1 - 4}_n})^{{\revar \frac{1}{2}}}\frac{\Delta_n}{M_n^{2}} + \Delta_n {\revar  h_n} \} 
  \text{ \hfill\quad\quad\revar [by the first two points of Lemma \ref{l: bound stationarity} with Jensen's inequality,]}
 \\
& \le \frac{c{\revar \Delta_n }}{M_n^4h_n^{3/2}} + \frac{c {\revar  h_n}}{M_n^2} \\
& \le \frac{c {\modar h_n}}{M_n^2},
\end{align*}
{\revar where the last line is a consequence of $1/M_n=O(h_n^\beta)$ with $\beta\ge3$.}
\\
We now deal with $I_2$. In order to bound it we will use several times \eqref{eq: Y bounded}, {\modar \eqref{eq: majo DX}} and Lemma \ref{l: bound moments Xdot}.
It follows
\begin{align*}
\E_{\modar x_0}[((D_s{\revnotat \smash{\dot{\widehat{X}}}}_{\Delta_n}^\epsilon)_2)^2] & \le \frac{c}{M_n^2} 
\int_s^{\Delta_n} \E_{\modar x_0}[(\psi_{h_n}''({\revnotat \widehat{X}}_u^\epsilon) {\revnotat \smash{\dot{\widehat{X}}}}_u^\epsilon)^2] du + \frac{c \Delta_n}{M_n^4} \int_s^{\Delta_n} \E_{\modar x_0}[(\psi_{h_n}'({\revnotat \widehat{X}}_u^\epsilon) \psi_{h_n}''({\revnotat \widehat{X}}_u^\epsilon) {\revnotat \smash{\dot{\widehat{X}}}}_u^\epsilon)^2] du \\
& + \frac{c \Delta_n}{M_n^4} \int_s^{\Delta_n} \E_{\modar x_0}[(\psi_{h_n}'({\revnotat \widehat{X}}_u^\epsilon))^2] du \\
&\le \frac{c}{M_n^2} \int_s^{\Delta_n} \E_{\modar x_0}[|\psi_{h_n}''({\revnotat \widehat{X}}_u^\epsilon)|^{{\revar 4}}]^{\revar 1/2}{\modar\frac{\Delta_n}{M_n^2}} du + \frac{c \Delta_n}{M_n^4} \int_s^{\Delta_n} \norm{\psi_{h_n}'}_\infty^2
\E_{\modar x_0}[|\psi_{h_n}''({\revnotat \widehat{X}}_u^\epsilon)|^{\revar 4}]^{\revar 1/2} {\modar\frac{\Delta_n}{M_n^2}} du \\
& + \frac{c \Delta_n}{M_n^4} \int_s^{\Delta_n} \E_{\modar x_0}[(\psi_{h_n}'({\revnotat \widehat{X}}_u^\epsilon))^2] du 
\end{align*}
Integrating with respect to $x_0$ the last equation and applying the first two points of Lemma \ref{l: bound stationarity}, we find
\begin{align*}
I_2 \le &\frac{c}{M_n^2} \int_s^{\Delta_n} (h_n^{\revar 1 - 8})^{\revar \frac{1}{2}} 
{\modar \frac{\Delta_n}{M_n^2}} du + {\modar \frac{c \Delta_n^2}{M_n^6} {\modar \frac{1}{h_n^2}}}
\int_s^{\Delta_n} {\revar (h_n^{1 - 8})^{\frac{1}{2}}} du 
 + \frac{c \Delta_n}{M_n^4} \int_s^{\Delta_n} \frac{1}{h_n} du \\
 \le &\frac{c \Delta_n^2}{M_n^4 {\modar h_n^{\revar 7/2}}} + \frac{c \Delta_n^3}{M_n^6  {\modar h_n^{{\revar 11/2}}}} + \frac{c \Delta_n^2}{M_n^4 h_n}.
\end{align*}
{\revar 
We remark that, as the choice of the calibration parameter $M_n$ is such that  $1/M_n=O(h_n^\beta)=O(h_n^3)$, all 
the three terms here above are smaller than $\frac{h_n}{M_n^2}$. 
It follows $I_2  
\le c \frac{h_n}{M_n^2}$, 
as we wanted.} 
\end{proof}

\subsection{Proof of Lemma \ref{l: I3}}

\begin{proof}
We observe that, according to the definition of $(D_\cdot {\revnotat \smash{\dot{\widehat{X}}}}_{\Delta_n}^\epsilon)_3$, the rest term is 
{ \modc \begin{align*}
{\modar R_n^{(2)}} & = \frac{1}{\Delta_n} \big[\frac{\Delta_n}{\bracket{D_\cdot {\revnotat \widehat{X}}_{\Delta_n}^\epsilon}{D_\cdot {\revnotat \widehat{X}}_{\Delta_n}^\epsilon }} - 1\big] Y_{\Delta_n}^\epsilon \int_{0}^{\Delta_n} (\int_u^{\Delta_n}  (Y_s^\epsilon)^{-1} \frac{1}{M_n} \psi_{h_n}'({\revnotat \widehat{X}}_s^\epsilon)(D_u {\revnotat \widehat{X}}_s^\epsilon) dB_s)  (Y_u^\epsilon)^{-1} Y_{\Delta_n}^\epsilon \\
& \times (1 + \frac{\epsilon}{M_n} \psi_{h_n}({\revnotat \widehat{X}}_u^\epsilon) ) du \\
&+ \frac{{\modarn \epsilon}}{\Delta_n M_n^2} \int_{0}^{\Delta_n} (Y_{\Delta_n}^\epsilon)^2 (Y_u^\epsilon)^{-1} \int_u^{\Delta_n} (Y_s^\epsilon)^{-1} \psi_{h_n}'({\revnotat \widehat{X}}_s^\epsilon)(D_u {\revnotat \widehat{X}}_s^\epsilon) dB_s  \psi_{h_n}({\revnotat \widehat{X}}_u^\epsilon) \epsilon du \\
& + \frac{1}{\Delta_n M_n} ((Y_{\Delta_n}^\epsilon)^2 - 1) \int_{0}^{\Delta_n} \int_u^{\Delta_n} (Y_s^\epsilon)^{-1} \psi_{h_n}'({\revnotat \widehat{X}}_s^\epsilon)(D_u {\revnotat \widehat{X}}_s^\epsilon) dB_s (Y_u^\epsilon)^{-1} du \\
& 
{\modarn + \frac{1}{\Delta_n M_n} 
	\int_{0}^{\Delta_n} \int_u^{\Delta_n} \psi_{h_n}'({\revnotat \widehat{X}}_s^\epsilon)((Y_u^\epsilon)^{-1}(Y_s^\epsilon)^{-1}D_u {\revnotat \widehat{X}}_s^\epsilon - 1) dB_s du }\\
& = \sum_{j = 1}^4 {\modar R_{n}^{(2,j)}}.
\end{align*}}
{\modar Using the {\modarn expression of $D_u {\revnotat \widehat{X}}_s^\epsilon$ as} in \eqref{eq: DsXt} with \eqref{eq: control Y}--\eqref{eq: control Y inv}, we have
$$\sup_{u, s \in [0, \Delta_n]}|D_u {\revnotat \widehat{X}}_s^\epsilon- 1| \le \frac{c \Delta_n}{M_n h_n^2} + \frac{c}{M_n} \psi_{h_n}({\revnotat \widehat{X}}_s^\epsilon).$$}
{ \modc It provides the following bounds, using also {\modarn \eqref{eq: Y bounded}, \eqref{eq: bracket moins 1}, and} the fact that $\psi_{h_n}$ is bounded }
\begin{align*}
	\E_{\modar x_0}[{\modar |R_{n}^{(2,1)}|^2}] \le & (\frac{ {\modarn \Delta_n}}{ M_n h_n^2})^2 \frac{1}{\Delta_n^2} \Delta_n \int_{0}^{\Delta_n} \E_{\modar x_0}[ (\int_u^{\Delta_n}  (Y_s^\epsilon)^{-1} \frac{1}{M_n} \psi_{h_n}'({\revnotat \widehat{X}}_s^\epsilon)(D_u {\revnotat \widehat{X}}_s^\epsilon) dB_s)^2] du \\
	& \le (\frac{ c}{ M_n h_n^2})^2 {\modarn \Delta_n} \int_{0}^{\Delta_n} \int_u^{\Delta_n} \frac{1}{M_n^2}  
	\E_{\modar x_0}[ (\psi_{h_n}'({\revnotat \widehat{X}}_s^\epsilon))^2] ds du, \\
	\\
	\E_{\modar x_0}[{\modar |R_{n}^{(2,2)}|^2}] \le & (\frac{ c}{ \Delta_n M_n^2})^2 \Delta_n \int_{0}^{\Delta_n} \E_{\modar x_0}[ (\int_u^{\Delta_n}  (Y_s^\epsilon)^{-1} \frac{1}{M_n} \psi_{h_n}'({\revnotat \widehat{X}}_s^\epsilon)(D_u {\revnotat \widehat{X}}_s^\epsilon) dB_s)^2 \psi_{h_n}({\revnotat \widehat{X}}_u^\epsilon)] du \\
	& \le (\frac{ c}{ \Delta_n M_n^2})^2 \Delta_n \int_{0}^{\Delta_n} \int_u^{\Delta_n} \E_{\modar x_0}[ (\psi_{h_n}'({\revnotat \widehat{X}}_s^\epsilon))^2] ds du, \\ 
	\\
	\E_{\modar x_0}[{\modar |R_{n}^{(2,3)}|^2}] \le & (\frac{ c}{ h_n^2 M_n^2})^2 \Delta_n \int_{0}^{\Delta_n} \int_u^{\Delta_n} \E_{\modar x_0}[ (\psi_{h_n}'({\revnotat \widehat{X}}_s^\epsilon))^2] ds du, \\
	\\
	\E_{\modar x_0}[{\modar |R_{n}^{(2,4)}|^2}] \le & (\frac{ c}{{\modarn  \Delta_n M_n}})^2 \Delta_n \int_{0}^{\Delta_n} \E_{\modar x_0}[(\int_u^{\Delta_n} \psi_{h_n}'({\revnotat \widehat{X}}_s^\epsilon)((Y_u^\epsilon)^{-1}(Y_s^\epsilon)^{-1}D_u {\revnotat \widehat{X}}_s^\epsilon - 1) dB_s)^2] du \\
	& \le  (\frac{ c}{{\modarn  \Delta_n M_n}})^2 \Delta_n \int_{0}^{\Delta_n} \int_u^{\Delta_n} (\frac{\Delta_n}{M_n h_n^2} + \frac{1}{M_n})^2 \E_{\modar x_0}[ (\psi_{h_n}'({\revnotat \widehat{X}}_s^\epsilon))^2] ds du. \\
\end{align*} 
{\modar We now use the second point of Lemma \ref{l: bound stationarity}, and deduce
{\revar 
	\begin{equation*}
\int_{|x_0|\le \Delta_n^{1/2-\gamma}} 	
\sum_{j = 1}^4  \E_{x_0}[|R_{n}^{(2,j)}|^2] { \pi^\epsilon}(x_0) dx_0\le
c\{
\frac{{ \Delta_n^3}}{M_n^4 h_n^5}+\frac{\Delta_n}{M_n^4 h_n}+\frac{\Delta_n^3}{M_n^4 h_n^5} +\frac{\Delta_n^3}{M_n^4 h_n^5}+\frac{\Delta_n}{M_n^4 h_n}  \}.
\end{equation*}
}
}
 It is easy to see that all these terms are smaller than $\frac{h_n}{M_n^2}$, up to {\revar using $\frac{1}{M_n}= O( h_n^\beta)$} and remarking we have requested $\beta \ge 3$. It follows that 
$$
{\modar \int_{|x_0|\le \Delta_n^{1/2-\gamma}} \E_{x_0}[|R_n^{(2)}|^2] {\modarn \pi^\epsilon}(x_0) dx_0} \le \frac{c \, h_n}{M_n^2}.$$
The proof is therefore concluded.
\end{proof}
}
\subsection{Proof of Lemma \ref{L: suppressing dB in principal term}}
{\modar 
	We want to use Ito's formula, in order to get rid of the stochastic integral in the left hand side of \eqref{eq: suppressing dB in principal term}.}
{\modar Hence, we write}
{\modar 
	\begin{equation} \label{eq : term dB diff I ronde}
		\int_u^{\Delta_n} \psi_{h_n}'({\revnotat \widehat{X}}_s^\epsilon) dB_s=\mathcal{I}_{\Delta_n}-\mathcal{I}_{u},
	\end{equation}
where $\mathcal{I}_t$ is defined in \eqref{eq: def I ronde}. From \eqref{eq: ito}, we deduce
\begin{equation} \label{eq : suppressing dB Ito}
	\int_u^{\Delta_n} \psi_{h_n}'({\revnotat \widehat{X}}_s^\epsilon) dB_s=
	\Xi({\revnotat \widehat{X}}_{\Delta_n}^\epsilon)-\Xi({\revnotat \widehat{X}}_u^\epsilon)-\frac{1}{2} \int_u^{\Delta_n} (\frac{{\psi_{h_n}'}}{a_\epsilon})'({\revnotat \widehat{X}}_s^\epsilon) a^2_\epsilon({\revnotat \widehat{X}}^\epsilon_s) ds
\end{equation}
}
{\modar Now, recall }
\begin{align*}
	\Xi (v) & : = \int_0^v \frac{\psi_{h_n}'(u)}{a_\epsilon (u) } du 
	{ \revar~ = \int_0^v  \frac{\psi_{h_n}'(u)}{(1\frac{\epsilon}{M_n}\psi_{h_n}(u))} du   }\\
	& = \psi_{h_n} (u) - \frac{\epsilon}{M_n} \int_0^v \frac{\psi_{h_n}'(u) \, \psi_{h_n} (u)}{1 + \frac{\epsilon}{M_n} \psi_{h_n}(u)}  du
	{\modarn ~ - ~ \psi_{h_n}(0)}
	 \\
	& = : \psi_{h_n} (u) + n_{h_n} (u) {\modarn ~ - ~\psi_{h_n}(0)},
\end{align*}
{\revar where we used $a_\epsilon(u) = 1 + \frac{\epsilon}{M_n} \psi_{h_n} (u)$ in the first line. We have $\left \| n_{h_n} \right \|_{\infty} \le \frac{c}{M_n}$.	
Then,  by \eqref{eq : suppressing dB Ito}, we obtain} 
\begin{align*}
{\modar 	\mathcal{I}_{\Delta_n}-\mathcal{I}_{u}} =  \psi_{h_n} ({\revnotat \widehat{X}}_{\Delta_n}^\epsilon) - \psi_{h_n} ({\revnotat \widehat{X}}_u^\epsilon) + n_{h_n} ({\revnotat \widehat{X}}_{\Delta_n}^\epsilon) - n_{h_n} ({\revnotat \widehat{X}}_u^\epsilon) - \frac{1}{2} \int_u^{\Delta_n} (\frac{\psi_{h_n}'}{a_\epsilon })' ({\revnotat \widehat{X}}_v^\epsilon) a_\epsilon^2({\revnotat \widehat{X}}_v^\epsilon) dv.
\end{align*}
We observe that $(\frac{\psi_{h_n}'}{a_\epsilon })' = \psi_{h_n}'' + m_{h_n}$, where $\left \| m_{h_n} \right \|_{\infty} \le \frac{c}{M_n h_n^2}$. It yields
\begin{align*}
{\modar 	\mathcal{I}_{\Delta_n}-\mathcal{I}_{u}} =  \psi_{h_n} ({\revnotat \widehat{X}}_{\Delta_n}^\epsilon) - \psi_{h_n} ({\revnotat \widehat{X}}_u^\epsilon) + n_{h_n} ({\revnotat \widehat{X}}_{\Delta_n}^\epsilon) - n_{h_n} ({\revnotat \widehat{X}}_u^\epsilon) - \frac{1}{2} \int_u^{\Delta_n} \psi_{h_n}'' ({\revnotat \widehat{X}}_v^\epsilon)dv + o_{\left \| \cdot \right \|_\infty}(\Delta_n \frac{1}{h_n^2 M_n}),
\end{align*}
{\modrev where we have introduced the notation $o_{\left \| \cdot \right \|_\infty}(\cdot)$ for $o(\cdot)$ such that the control is uniform over $\R$.}
{\revar Using  \eqref{eq : term dB diff I ronde}, we deduce}
\begin{align*}
	&{\revar \frac{1}{\Delta_n M_n} \int_0^{\Delta_n} 
	\int_{u}^{\Delta_n} \psi_{h_n}'({\revnotat \widehat{X}}_s^\epsilon)dB_s		
		 du
	= }
	\frac{1}{\Delta_n M_n} \int_0^{\Delta_n}{\modar 	(\mathcal{I}_{\Delta_n}-\mathcal{I}_{u}) }  du \\
	&= - \frac{1}{2 \Delta_n M_n} \int_0^{\Delta_n} (\int_u^{\Delta_n} \psi_{h_n}'' ({\revnotat \widehat{X}}_v^\epsilon)dv) du + \frac{1}{\Delta_n M_n} \int_0^{\Delta_n} (\psi_{h_n} ({\revnotat \widehat{X}}_{\Delta_n}^\epsilon) - \psi_{h_n} ({\revnotat \widehat{X}}_u^\epsilon)) du \\
	& + \frac{1}{\Delta_n M_n} \int_0^{\Delta_n} (n_{h_n} ({\revnotat \widehat{X}}_{\Delta_n}^\epsilon) - n_{h_n} ({\revnotat \widehat{X}}_u^\epsilon)) du + \frac{c}{\Delta_n M_n} \int_0^{\Delta_n} 
	{\modar o_{\left \| \cdot \right \|_\infty}(\Delta_n \frac{1}{h_n^2 M_n})}	
	du.
\end{align*}
{\revar
The} previous equation can be written
\begin{equation*}
	\frac{1}{\Delta_n M_n} \int_{0}^{\Delta_n} \int_u^{\Delta_n} \psi_{h_n}'({\revnotat \widehat{X}}_s^\epsilon) dB_s du
	  {\revar ~=:} -\frac{1}{2\Delta_n M_n}  \int_0^{\Delta_n}  \psi''_{h_n}({\revnotat \widehat{X}}_s^\epsilon)sds+\sum_{j=1}^3R^{(3,j)}_n
\end{equation*}
Comparing with \eqref{eq: suppressing dB in principal term}, we need to 
prove $\int_{|x_0|\le \Delta_n^{1/2-\gamma}} \E_{x_0}[|R^{(3,j)}_n|^2] (x_0){\modarn \pi^\epsilon}dx_0
\le \frac{c h_n}{M_n^2}$ for $j=1,2,3$.
In particular, using also the first point of Lemma \ref{l: bound stationarity}, we have for $j=1$,
\begin{align*}
\int_{|x_0|\le \Delta_n^{1/2-\gamma}} \E_{x_0}[|R^{(3,1)}_n|^2] {\modarn \pi^\epsilon}(x_0)dx_0
&\le \frac{c}{ \Delta_n^2 M_n^2} \Delta_n \int_0^{\Delta_n} \int_{|x_0|\le \Delta_n^{1/2-\gamma}}(\E[\psi^2_{h_n} ({\revnotat \widehat{X}}_{\Delta_n}^\epsilon)] + \E[\psi^2_{h_n} ({\revnotat \widehat{X}}_{u}^\epsilon)]) {\modarn \pi^\epsilon}(x_0)dx_0du 
\\ & \le \frac{c h_n}{M_n^2}.
\end{align*}
{\revar We have,  using $\norm{\eta_{h_n}} \le C/M_n$,}
\begin{align*}
	\int_{|x_0|\le \Delta_n^{1/2-\gamma}} \E_{x_0}[|R^{(3,2)}_n|^2] {\modarn \pi^\epsilon}(x_0)dx_0
	\le \frac{c}{M_n^2}(\frac{1}{{\revar  M_n}})^2 = c \frac{1}{{\revar M_n^4}}.
\end{align*}
This is negligible with respect to $\frac{h_n}{M_n^2}$ as 
{\revar $1/M_n=O(h_n^3)$.
} Regarding {\modar $R^{(3,3)}$ we have, $|R^{(3,3)}|\le c \Delta_n / (h_n^2M_n^2)$ and thus, 
\begin{align*}
	\int_{|x_0|\le \Delta_n^{1/2-\gamma}} \E_{x_0}[|R^{(3,3)}_n|^2] {\modarn \pi^\epsilon}(x_0)dx_0\le c \frac{\Delta_n^2}{M_n^4 h_n^4},
\end{align*}
which is clearly negligible compared to $\frac{h_n}{M_n^2}$, acting as {\revar for $R^{(3,2)}$.}}

\subsection{Proof of Lemma \ref{l: main term malliavin}}
\begin{proof}
	Let us denote by $p^\epsilon_{s,\Delta_n}(\cdot \mid x,y)$ 
	the density of the law of ${\revnotat \widehat{X}}^\epsilon_s$ conditional to ${\revnotat \widehat{X}}^\epsilon_0=x$ and ${\revnotat \widehat{X}}^\epsilon_{\Delta_n}=y$. Then, we have
	\begin{align} \label{E: def ell TL}
	\ell_{x,y,\Delta_n,h_n}:=\E_x \left[
	\int_0^{\Delta_n} 
	(\psi_{h_n})^{\prime\prime}({\revnotat \widehat{X}}^\epsilon_s) sds
	\mid {\revnotat \widehat{X}}^\epsilon_{\Delta_n}=y\right]	
		&=\int_0^{\Delta_n}s\int_{\mathbb{R}}  (\psi_{h_n})^{\prime\prime}(z ) p^\epsilon_{s,\Delta_n}(z\mid x,y) dzds
		\\ \label{E: ell majo debut}
		&=\int_0^{\Delta_n}s\int_{\mathbb{R}}  \psi_{h_n}(z ) 
		\frac{\partial^2 p^\epsilon_{s,\Delta_n}(z\mid x,y)}{\partial z^2} dzds
	\end{align}
Recalling the expression for the density of the pinned diffusion {\revar 
	\begin{equation*} p^\epsilon_{s,\Delta_n}(z\mid x,y)=\frac{p^\epsilon_s(x,z)p^\epsilon_{\Delta_n-s}(z,y)}{p^\epsilon_{\Delta_n}(x,y)},
	\end{equation*}}
 we have
		$$
	 \frac{\partial^2 p^\epsilon_{s,\Delta_n}(z\mid x,y)}{\partial z^2}= 
		\frac{\frac{\partial^2 p^\epsilon_{s}(x,z)}{\partial z^2} p_{\Delta_n-s}(z,y)+
			2\frac{\partial p^\epsilon_{s}(x,z)}{\partial z} \frac{\partial p^\epsilon_{\Delta_n-s}(z,y)}{\partial z}+ 
			 p^\epsilon_{s}(x,z)\frac{\partial^2p_{\Delta_n-s}(z,y)}{\partial z^2}
	}{p^\epsilon_{\Delta_n}(x,y)}.
	$$ 
{\modr Now, we use the Gaussian controls  given in Sections A.2.2 - A.2.3 of Azencott \cite{Azencott84} for the density of the transition of a diffusion and of its derivatives. 
	Notice that the results of  \cite{Azencott84} apply here to the diffusion $\widehat{X}^\epsilon$ as the conditions given in Section A.2.1 of \cite{Azencott84} are satisfied. Especially, the condition (3) p. 477 of \cite{Azencott84} requires that at least the first three derivatives of the coefficients of the SDE are bounded, which is valid here from the condition $1/(M_n h_n^3) \le 1$.
	An important feature of the controls given in \cite{Azencott84} is that they sharply relate the variance of the Gaussian controls with the diffusion coefficient of the diffusion process. }
{\modr From (6) in Section A.2.2 of  \cite{Azencott84} (see also (75) p. 492 of \cite{Azencott84}), we have}
	\begin{equation} \label{E:Azencott_derivative_x}
\abs{\frac{\partial^j p^\epsilon_s(u,v)}{\partial u^j} }\le  \frac{C}{s^{j/2}}\gaussian_{s\lambda}(u-v), \quad j=0,1,2,
	\end{equation}
	where $\gaussian_{s\lambda}$ is the density of the {\modarn centred real} Gaussian variable with variance {\revar $s\lambda $} and $\lambda>0$ is any real constant such that {\modr $\lambda> (1+\epsilon \frac{1}{M_n} \psi_{h_n}(z))^2$,} for all $z \in \mathbb{R}$. 
As the diffusion process is symmetric with respect to the stationary probability, we know
$p^\epsilon_s(u,v)=p^\epsilon_s(v,u) \frac{\pi^\epsilon(v)}{\pi^\epsilon(u)}$ where $\pi^\epsilon$ is  {\revar the function  $ u \mapsto \frac{c_{\pi^\epsilon}}{(1+\frac{\epsilon}{M_n}\psi_{h_n}(u))^2}
e^{-2\eta \int_0^u \widehat{\text{sgn}}(w)dw}$, and the constant $c_{\pi^\epsilon}$ is bounded independently of $n,\epsilon$ recalling \eqref{eq : pi epsilon explicit} and the discussion below.} Using the assumption 
$\limsup_{n} \frac{1}{h_n^2 M_n}<\infty$, we can check that the first two derivatives of $\pi^\varepsilon$ are bounded and it follows from \eqref{E:Azencott_derivative_x} that
	\begin{equation} \label{E:Azencott_derivative_y}
	\abs{\frac{\partial^j p^\epsilon_s(u,v)}{\partial v^j} }\le  \frac{C}{s^{j/2}}\gaussian_{s\lambda}(u-v), \quad j=0,1,2.
\end{equation}
As $\norm{\psi_{h_n}/M_n}_\infty \to0$, it is possible for $n$ large enough to use \eqref{E:Azencott_derivative_x}--\eqref{E:Azencott_derivative_y} with {\modr $\lambda=1+\eta ~ \in (1,5/4]$,} we deduce
$$
\abs{\frac{\partial^2 p^\epsilon_{s,\Delta_n}(z\mid x,y)}{\partial z^2}}
\le C \left[\frac{1}{s}+ \frac{1}{\sqrt{s} \sqrt{\Delta_n - s}} + \frac{1}{\Delta_n-s}\right]
\frac{\gaussian_{s(1{\modr+}\eta)}(z-x)\gaussian_{(\Delta_n -s)(1{\modr+}\eta)}(y-z)}{p^\epsilon_{\Delta_n}(x,y)}.
$$
{\revarn Using Section A.2.3. in} Azencott \cite{Azencott84}, with {\modr $(1-\epsilon \frac{1}{M_n} \psi_{h_n}(z))^2> (1-\eta)$} for $n$ large enough, we have $p^\epsilon_{\Delta_n}(x,y) \ge C \gaussian_{{\modr \Delta_n(1-\eta)}}(y-x)$. In turn, we deduce
\begin{align*}
\abs{\frac{\partial^2 p^\epsilon_{s,\Delta_n}(z\mid x,y)}{\partial z^2}} &
\le C \left[\frac{1}{s}+\frac{1}{\Delta_n-s}\right]
\frac{\gaussian_{s(1{\modr+}\eta)}(z-x)\gaussian_{(\Delta_n-s)(1{\modr+}\eta)}(y-z)}{\gaussian_{\Delta_n(1{\modr-}\eta)}(y-x)}
\\
&
\le C \left[\frac{1}{s}+\frac{1}{\Delta_n-s}\right]
\frac{\gaussian_{s(1{\modr+}\eta)}(z-x)\gaussian_{(\Delta_n-s)(1{\modr+}\eta)}(y-z)}{\gaussian_{\Delta_n(1{\modr+}\eta)}(y-x)} e^{2\eta\frac{(y-x)^2}{\Delta_n}}.
\end{align*}
Remarking that the ratio of Gaussian densities in the previous display is the law at time $s\in(0,\Delta_n)$ of a Brownian bridge from $x$ to $y$ with {\revarn diffusion coefficient $\sqrt{1{\modr+}\eta}$,} we deduce from \eqref{E: ell majo debut},
\begin{equation} \label{E: control ell par BM} 
\abs{\ell_{x,y,\Delta_n,h_n}} \le C \int_0^{\Delta_n} s \left[\frac{1}{s}+\frac{1}{\Delta_n -s}\right] \E_x\left[\abs{\psi_{h_n}({\modarn \widetilde{B}}_s)} \mid {\modarn \widetilde{B}}_{\Delta_n}=y\right]ds,\end{equation}
where ${\modarn \widetilde{B}}$ is a Brownian motion with {\revarn diffusion coefficient $\sqrt{1{\modr+}\eta}$.}  Denoting the local time at level $z$ of the Brownian motion ${\modarn \widetilde{B}}$ by $(L^{z}_t({\modarn \widetilde{B}}))_t$ we deduce using Fubini's Theorem and the occupation time formula {\revar (see Corollary (1.6) and Exercise 1.15 of Chapter VI in \cite{RevYor_3rd_edition})},
\begin{align}\nonumber
\abs{\ell_{x,y,\Delta_n,h_n}}& \le C \E_x\left[
\int_0^{\Delta_n}
\left[1+\frac{s}{\Delta_n-s}\right] \abs{\psi_{h_n}({\modarn \widetilde{B}}_s)} ds \mid {\modarn \widetilde{B}}_{\Delta_n}=y\right]e^{2\eta\frac{(y-x)^2}{\Delta_n}}\\
&
\label{E: control l}
= \frac{C}{1-\eta}\int_{\mathbb{R}} \abs{\psi_{h_n}(z)}\E_x\left[
\int_0^{\Delta_n}
\left[1+\frac{s}{\Delta_n-s}\right] dL^z_s({\modarn \widetilde{B}})  \mid {\modarn \widetilde{B}}_{\Delta_n}=y\right]dze^{2\eta\frac{(y-x)^2}{\Delta_n}}.
\end{align}
Using that, after time reversal, the law of the Brownian bridge remains a Brownian {\modarn bridge,}
we have
\begin{equation*}
\E_x\left[
\int_0^{\Delta_n}
\left[1+\frac{s}{\Delta_n-s}\right] dL^z_s({\modarn \widetilde{B}})  \mid {\modarn \widetilde{B}}_{\Delta_n}=y\right]
= \E_y\left[
\int_0^{\Delta_n}
\left[1+\frac{\Delta_n-s}{s}\right] dL^z_s({\modarn \widetilde{B}})  \mid {\modarn \widetilde{B}}_{\Delta_n}=x\right].
\end{equation*}
Using Lemma \ref{L: majo func Local time cond Delta} with $\beta=0$ and $\beta=1$ {\revar on the right hand side of the previous equation}, we deduce, for $y\neq z$
\begin{equation} \label{E: cont TL y ge h}
\E_x\left[
\int_0^{\Delta_n}
\left[1+\frac{s}{\Delta_n-s}\right] dL^z_s({\modarn \widetilde{B}})  \mid {\modarn \widetilde{B}}_{\Delta_n}=y\right]
\le C \sqrt{\Delta_n} \left[1+
{\modarn \left[ 
\abs{\frac{y-z}{\sqrt{\Delta_n}}}^{-1} +1 \right]}e^{-\frac{(y-z)^2}{\eta\Delta_n}}e^{\frac{\eta(y-x)^2}{\Delta_n}}
\right].
\end{equation}
We now focus on the case $\abs{y}>2h_n$.  Using \eqref{E: control l}, \eqref{E: cont TL y ge h} and that the support of $\psi_{h_n}$ is included in $[-h_n,h_n]$, we have,
\begin{equation*}
	\abs{\ell_{x,y,\Delta_n h_n}}
	\le 
	C \sqrt{\Delta_n} 
	\norm{{\revar\psi}}_\infty
	\int_{-h_n}^{h_n} \left[1+
	\abs{\frac{y-z}{\sqrt{\Delta_n}}}^{-1} e^{-\frac{(y-z)^2}{\eta\Delta_n}} \right]dz e^{\frac{3\eta(y-x)^2}{\Delta_n}}.
\end{equation*}
Using that  $u \mapsto u^{-1} e ^{-\frac{u}{\eta\Delta_n}}$ is non increasing on $(0,\infty)$ and that $|y| \ge 2h_n$, we deduce
\begin{equation*}
\abs{\ell_{x,y,\Delta_n,h_n}} \le 2 C \sqrt{\Delta_n} h_n [1+\abs{\frac{\abs{y}-h_n}{\sqrt{\Delta_n}}}^{-1} e^{-\frac{(\abs{y}-h_n)^2}{\eta\Delta_n}} ] e^{\frac{3\eta(y-x)^2}{\Delta_n}}. 
\end{equation*}
Recalling the definition of $\ell_{x,y,\Delta_n,h_n}$ {\modarn given in} \eqref{E: def ell TL},
 this entails \eqref{E: terme pal crochet psi pp} for $\abs{y} >2h_n$.

We now treat the case $\abs{y} \le 2h_n$. Recalling the definition of $\ell_{x,y,\Delta_n,h_n}$ in \eqref{E: def ell TL} we write $\int_0^{\Delta_n} (\psi_{h_n})^{\prime\prime}({\revnotat \widehat{X}}_s)s ds=
\int_0^{\Delta_n-h^2_n} (\psi_{h_n})^{\prime\prime}({\revnotat \widehat{X}}_s)s ds + O(\Delta_n)$, where we used
$\norm{\psi_{h_n}^{\prime\prime}}_\infty \le C/h_n^2$.
We deduce that 
$$\ell_{x,y,\Delta_n,h_n} \le 
\E_x \left[
\int_0^{\Delta_n-h^2_n} (\psi_{h_n})^{\prime\prime}({\revnotat \widehat{X}}^\epsilon_s)s ds
\mid {\revnotat \widehat{X}}^\epsilon_{\Delta_n}=y\right]
+O(\Delta_n).$$
  Then, we repeat the same computations that yields to \eqref{E: control l}, we deduce
\begin{equation}\label{E: maj ell pour y petit}
	\ell_{x,y,\Delta_n,h_n}\le
	\frac{C}{1-\eta}\int_{\mathbb{R}} \abs{\psi_{h_n}(z)}\E_x\left[
	\int_0^{\Delta_n-h_n^2}
	\left[1+\frac{s}{\Delta_n-s}\right] dL^z_s(\tilde{B})  \mid \tilde{B}_{\Delta_n}=y\right]dze^{2\eta\frac{(y-x)^2}{\Delta_n}} + O(\Delta_n)
\end{equation} 
By time reversal invariance of the law of the Brownian bridge, we have
\begin{align}{\label{E: new}}
	\E_x \left[\int_0^{\Delta_n-h_n^2}
	\left[1+\frac{s}{\Delta_n-s}\right] dL^z_s({\modarn \widetilde{B}})  \mid {\modarn \widetilde{B}}_{\Delta_n}=y\right]
	&=\E_y \left[
	\int_{h_n^2}^{\Delta_n}
	\left[1+\frac{\Delta_n-s}{s}\right] dL^z_s({\modarn \widetilde{B}}) 
	 \mid {\modarn \widetilde{B}}_{\Delta_n}=x\right]
	 \\ \nonumber
	 & \le \E_y \left[
	 \int_{0}^{\Delta_n}
	  dL^z_s({\modarn \widetilde{B}}) 
	 \mid {\modarn \widetilde{B}}_{\Delta_n}=x\right]
	 \\ \nonumber & \quad\quad\quad
	+ \Delta_n
	\E_y \left[
	\int_{h_n^2}^{\Delta_n}\frac{1}{s} dL^z_s({\modarn \widetilde{B}}) 
	\mid {\modarn \widetilde{B}}_{\Delta_n}=x\right]
\\ \nonumber& \le C \sqrt{\Delta_n} e^{\eta\frac{(y-x)^2}{\Delta_n}} + \frac{C \Delta_n}{h_n},
\end{align}
where we used Lemma \ref{L: majo func Local time cond Delta} with $\beta=0$, and \eqref{E: moment LT int cond xy tronque h2} in Lemma \ref{L: divers majo TL}.
Since the support of $\psi_{h_n}$ is $[-h_n,{\modarn h_n}]$, we deduce from \eqref{E: maj ell pour y petit} and \eqref{E: new}
\begin{equation*}
	\ell_{x,y,\Delta_n,h_n}\le  C \sqrt{\Delta_n} h_n e^{3\eta\frac{(y-x)^2}{\Delta_n}} + C \Delta_n e^{2\eta\frac{(y-x)^2}{\Delta_n}}.
\end{equation*}
This yields \eqref{E: terme pal crochet psi pp} for $\abs{y} \le2h_n$.
\end{proof}
\subsubsection{Conditional first moment for integrals of local time}
In this section we obtain upper bound for conditional first moment of quantities of the form
$\int_0^1 \varphi(s) dL_s^0({\revarn \widehat {B}})$ where ${\revarn \widehat {B}}$ is a B.M. with 
{\revarn diffusion coefficient $\sqrt{1{\modr+}\eta}\in {\modr (1, \sqrt{2}]}$} 
and deduce controls useful in the proof of
{\modarn Lemma \ref{l: main term malliavin}}. 
In \cite{Gradinaru_et_al_1999} the authors study the law of such quantities when $\varphi(s)=s^{-\beta}$ with $\beta \le 1/2$. We follow some ideas from \cite{Gradinaru_et_al_1999}.
\begin{lemma} \label{L: divers majo TL}
	Consider $({\modarn \widehat{B}}_s)_{s\in[0,1]}$ a Brownian motion with diffusion coefficient $\sqrt{1{\modr+}\eta}\in 
		{\modr (1,\sqrt{2}]}$.
	\begin{enumerate}
		\item Let $\psi: \mathbb{R} \mapsto [0,\infty)$ be a bounded measurable function. Then,
		\begin{align} \label{E: moment LT int}
			&\E_0 \left[ \int_0^1 \psi(s) {\modarn dL_s^0(\widehat{B})}  \right] \le {\modr \frac{1}{\sqrt{\bm{\pi}}}} \int_0^1 \frac{\psi(u)}{\sqrt{u}} du,
			\\ \label{E: moment LT int cond}
			&\E_0 \left[ \int_0^a \psi(s) {\modarn dL_s^0(\widehat{B})} \mid {\modarn \widehat{B}}_a=0 \right] \le \frac{{\modr \sqrt{2}}}{\sqrt{\bm{\pi}}} \int_0^a \psi(u)\left(\frac{1}{\sqrt{u}}+\frac{1}{\sqrt{a-u}}\right) du, \quad
			\text{$\forall a \in (0,1)$}.
		\end{align}
		\item For $\beta\ge 0$, $\beta\neq 1/2$ and $(x,y) \in \mathbb{R}^2$, $y\neq 0$, we have
		\begin{equation} \label{E: moment LT int cond xy}
			\E_y \left[\int_0^1 \frac{dL_s^0({\modarn \widehat{B}})}{s^\beta} \mid {\modarn \widehat{B}}_1=x\right]\le 
		C(1+\abs{y}^{(1-2\beta)\land 0} e^{-\frac{y^2}{\eta}+\eta x^2}).
		\end{equation}
		Moreover, the result holds true for $y=0$ if $\beta\in[0,1/2)$.
		\item We have, for all ${\modarn (x,y,z) \in \mathbb{R}^3}$, $0<h_n^2<\Delta_n<1$,
		\begin{equation}\label{E: moment LT int cond xy tronque h2}
			\E_x \left[\int_{h_n^2}^{\Delta_n} \frac{dL_s^{\modarn z}({\modarn \widehat{B}})}{s}\mid {\modarn \widehat{B}_{\Delta_n}}=y\right] \le \frac{C}{h_n}
			{\modarn ,}
		\end{equation}
	{\modarn where the constant $C$ is independent of $x,y,z,h_n$,$\Delta_n$.}	
	\end{enumerate}
\end{lemma}
\begin{proof} {\em Point 1.}
	We start with the proof of \eqref{E: moment LT int}. We follow the ideas in {\modarn Remark 1.25 iii)} of \cite{Gradinaru_et_al_1999}. Let us denote by $g_t=\sup\{ s \in [0,t], {\modarn \widehat{B}}_s=0\}$ {\modr which belongs to $[0,t]$ as $\widehat{B}_0=0$.} 
	Applying the balayage formula (Theorem 4.2 in \cite{RevYor_3rd_edition}),
	we have
	\begin{align*}
		\psi(g_1)|{\modarn \widehat{B}}_1|&=\psi(0)|{\modarn \widehat{B}}_0| + \int_0^1 \psi(g_s) d |{\modarn \widehat{B}}_s|
		\\
		&=\psi(0)|{\modarn \widehat{B}}_0| + \int_0^1 \psi(g_s) \text{sign}({\modarn \widehat{B}}_s) d {\modarn \widehat{B}}_s + \int_0^t \psi(g_s) d L^0_s({\modarn \widehat{B}}).
	\end{align*}
	Taking expectation, and remarking that on the support of the measure $dL^0({\modarn \widehat{B}})$, we have $g_s=s$, we deduce  
		$$
	\E_0\left[\int_0^1 \psi(s) dL^{0}_s({\modarn \widehat{B}})\right]=\E_0\left[\psi(g_1) |{\modarn \widehat{B}}_1| \right].
	$$
	{\revar From Section 3 of Chapter XII in \cite{RevYor_3rd_edition} (see also \cite{BertoinPitman94}), we know that conditionally} 
	to $g_1$ the law of $|{\modarn \widehat{B}}_1|$ is $\sqrt{1-g_1} \sqrt{1{\modr+}\eta} M_1$ where $(M_t)_t$ is a Brownian {\revar meander. By \cite{Durrett_et_al77}, the variable
		$M_1$ follows the Rayleigh distribution, for which $\E[M_1]=\sqrt{\frac{\bm{\pi}}{2}}$.} It entails,
	$$ \E_0\left[\int_0^1 \psi(s) dL^{0}_s({\modarn \widehat{B}})\right] \le \sqrt{\frac{(1{\modr+}\eta)\bm{\pi}}{2}}
	\E_0[\psi(g_1){\modarn \sqrt{1-g_1}}].$$
	As $g_1$ is distributed according to the arcsin law {\revar (e.g. see Paragraph 18 in Chapter 4 of \cite{BorodinSalminenBook02}),} with density $\frac{1}{\bm{\pi}\sqrt{s(1-s)}}$ we deduce \eqref{E: moment LT int}.
	
	We now prove \eqref{E: moment LT int cond}. We split the left hand side of \eqref{E: moment LT int cond} as
	\begin{equation*}
		\E_0 \left[ \int_0^{a/2} \psi(s) dL_s^0({\modarn \widehat{B}}) \mid {\modarn \widehat{B}}_a=0 \right]+
		\E_0 \left[ \int_{a/2}^{a} \psi(s) dL_s^0({\modarn \widehat{B}}) \mid {\modarn \widehat{B}}_a=0 \right].
	\end{equation*}
	{\revar The law of the Brownian bridge is absolutely continuous with respect to the the law of the Brownian motion on $\mathcal{F}_u=\sigma\{{ \widehat{B}}_s, { s \le u}\}$,  for  $u <a$, and the Radon--Nikodym density as given in Paragraph 23 of Chapter 4 of \cite{BorodinSalminenBook02} is
	\begin{equation*}
		\frac{\gaussian_{(1-\eta)(a-u)}({ \widehat{B}}_{u})}{\gaussian_{(1-\eta)a}(0)}.
\end{equation*}
This enables us to write,}
	\begin{align*}
		\E_0 \left[ \int_0^{a/2} \psi(s) dL_s^0({\modarn \widehat{B}}) \mid {\modarn \widehat{B}}_a=0 \right]
	&=\E_0 \left[ \int_0^{a/2} \psi(s) dL_s^0({\modarn \widehat{B}}) \frac{\gaussian_{(1-\eta)a/2}({\revar \widehat{B}}_{a/2})}{\gaussian_{(1-\eta)a}({\revar0})} \right]\\
	&\le \sqrt{2} \E_0 \left[ \int_0^{a/2} \psi(s) dL_s^0({\modarn \widehat{B}}) \right]
	\le \frac{{\modr \sqrt{2}}}{\sqrt{\bm{\pi}}} \int_0^{a/2} \frac{\psi(u)}{\sqrt{u}}du,
	\end{align*}
{\revar where the first inequality on the second line follows from the explicit expression for the Gaussian density, and the last inequality
comes from \eqref{E: moment LT int}.
	Using invariance by time reversion of the law of the Brownian bridge, we get,}
	\begin{align*}
	\E_0 \left[ \int_{a/2}^a \psi(s) dL_s^0(\widehat{B}) \mid \widehat{B}_a=0 \right]
	&=
	\E_0 \left[ \int_0^{a/2} \psi(a-s) dL_s^0(\widehat{B}) \mid \widehat{B}_a=0 \right]
	\le \frac{{\modr \sqrt{2}}}{\sqrt{\bm{\pi}}}\int_0^{a/2} \frac{\psi(a-u)}{\sqrt{u}}du
	\\
	&=\frac{{\modr \sqrt{2}}}{\sqrt{\bm{\pi}}}\int_{a/2}^a \frac{\psi(u)}{\sqrt{a-u}}du.
	\end{align*}
	We deduce \eqref{E: moment LT int cond}.
	
	{\em Point 2.} We now prove \eqref{E: moment LT int cond xy}. Let us denote $\tau_0=\inf\{t \ge 0, \widehat{B}_t=0\}$ and remark that $g_1=\sup\{t \le 1, \widehat{B}_t=0\}$ is finite on the event $\tau_0<1$. Using that $\int_0^1\frac{dL_s^0(\widehat{B})}{s^\beta}=0$ on $\tau_0 \ge 1$, we have
	\begin{align} \nonumber
		\E_y \left[\int_0^1\frac{dL_s^0({\modarn \widehat{B}})}{s^\beta} \mid {\modarn \widehat{B}}_1=x\right]
		&=\E_y \left[\int_0^1\frac{dL_s^0({\modarn \widehat{B}})}{s^\beta} \Indi{\{\tau_0<1\}}
		 \mid {\modarn \widehat{B}}_1=x\right]	
		\\
		&=\E_y \left[
		\E_y\left[ 
		\int_{\tau_0}^{g_1} 
		\frac{dL_s^0({\modarn \widehat{B}})}{s^\beta} \mid \tau_0,g_1,{\modarn \widehat{B}}_1=x \right] \Indi{\{\tau_0<1\}} \mid {\modarn \widehat{B}}_1=x\right].	
		\label{E: maj TL cond by G}
	\end{align}
	Conditionally to $\tau_0$, the law of the process {\modarn $\displaystyle (\frac{{\modarn \widehat{B}}_{\tau_0+s}}{\sqrt{g_1-\tau_0}})_{s \in [0,g_1-\tau_0]}$} is on the event $\tau_0<1$, the law of a Brownian bridge from $0$ to $0$. Moreover, this Brownian bridge is independent of $g_1$ and of $({\modarn \widehat{B}}_{g_1+s})_{s \in [0,1-g_1]}$ {\modarn (see {\revar Section 3 of} Chapter XII in \cite{RevYor_3rd_edition})}. We deduce that, on $\tau_0<1$,
	\begin{align} \label{E: def G tau g}
		G(\tau,g):=&
		\E_y\left[
		\int_{\tau_0}^{g_1} 
		\frac{dL_s^0({\modarn \widehat{B}})}{s^\beta} \mid \tau_0=\tau,g_1=g,{\modarn \widehat{B}}_1=x \right]
		\\ \nonumber
		=&\E_0\left[ \int_{0}^{g-\tau} 
		\frac{dL_s^0({\modarn \widehat{B}})}{(\tau+s)^\beta}  \mid {\modarn \widehat{B}}_{g-\tau}=0\right].
	\end{align}
	Using \eqref{E: moment LT int cond}, we obtain
	\begin{align*}
			G(\tau,g)&\le \frac{{\modr \sqrt{2}}}{\sqrt{\bm{\pi}}} \int_0^{g-\tau} \frac{1}{(\tau+s)^\beta} \left(\frac{1}{\sqrt{s}}+
			\frac{1}{\sqrt{g-\tau-s}} \right) ds \\
			&{\modarn \le
				\frac{{\modr \sqrt{2}}}{\sqrt{\bm{\pi}}}
			\int_0^{\frac{g-\tau}{2}} \frac{2}{(\tau+s)^\beta\sqrt{s}} ds +
				\frac{{\modr \sqrt{2}}}{\sqrt{\bm{\pi}}}
			\int_{\frac{g-\tau}{2}}^{g-\tau} \frac{2}{(\tau+s)^\beta\sqrt{g-\tau-s}} ds }
		\\	
		&{\modarn \le
			\frac{{\modr \sqrt{2}}}{\sqrt{\bm{\pi}}}
			\int_0^{\frac{g-\tau}{2}} \frac{2}{(\tau+s)^\beta\sqrt{s}} ds +
			\frac{{\modr \sqrt{2}}}{\sqrt{\bm{\pi}}}
			\int_0^{\frac{g-\tau}{2}} \frac{2}{(g-s')^\beta\sqrt{s'}} ds', \quad \text{ setting $s'=g-\tau-s$.}}
		\end{align*}
{\modarn For $s' \le \frac{g-\tau}{2}$, we have $g-s' \ge \tau+s'$, and thus
\begin{equation*}
	G(\tau,g) \le
	\frac{4{\modr \sqrt{2}}}{\sqrt{\bm{\pi}}}
	\int_0^{\frac{g-\tau}{2}} \frac{1}{(\tau+s)^\beta\sqrt{s}} ds.
\end{equation*}
}
From this equation we can see that for $\beta\in[0,1/2)$, $G(\tau,g) \le {\modarn (4{\modr \sqrt{2}}/\sqrt{\bm{\pi}})} \int_0^1 s^{-1/2-\beta} ds \le C<\infty$ for some constant $C$. And for $\beta>1/2$, $G(\tau,g) \le  	\frac{4{\modr \sqrt{2}}}{\sqrt{\bm{\pi}}}
	\int_0^{\infty} \frac{ds}{(\tau+s)^\beta\sqrt{s}}  {\revar ~= \frac{4{\modr \sqrt{2}}}{\bm{\pi}} \int_0^\infty \frac{du}{(u+1)^\beta\sqrt{u}} \tau^{1/2-\beta}}\le
C \tau^{1/2-\beta}$. Collecting this with {\modarn \eqref{E: maj TL cond by G} and \eqref{E: def G tau g},} we deduce,
\begin{equation*}
	\E_y \left[\int_0^1\frac{dL_s^0({\modarn \widehat{B}})}{s^\beta} \mid {\modarn \widehat{B}}_1=x\right]
\le C \E_y \left[ (\tau_0)^{(1/2-\beta)\wedge 0} 1_{\tau_0 < 1} \mid {\modarn \widehat{B}}_1=x\right]. 	
\end{equation*}
If $\beta\in [0,1/2)$, we deduce $	\E_y \left[\int_0^1\frac{dL_s^0({\modarn \widehat{B}})}{s^\beta} \mid {\modarn \widehat{B}}_1=x\right] \le C$ and \eqref{E: moment LT int cond xy} follows in the case $\beta \in [0,1/2)$.

Assume now that $\beta>1/2$ and $y \neq 0$. We use that the Radon-Nikodym ratio between the law of the Brownian bridge and the Brownian motion restricted to the sigma field $\mathcal{F}_{\tau_0}$, and the event {\modarn $\tau_0<1$} is 
{\modarn $ \displaystyle \frac{\gaussian_{(1-\tau_0)(1{\modr+}\eta)}(x-\widehat{B}_{\tau_0})}{\gaussian_{1{\modr+}\eta}(x-y)}=
\frac{\gaussian_{(1-\tau_0)(1{\modr+}\eta)}(x)}{\gaussian_{1{\modr+}\eta}(x-y)}$,} we have {\modarn for any $0<\varepsilon\le1/2$,}
\begin{align} \nonumber
	\E_y \left[ \tau_0^{1/2-\beta} \Indi{\tau_0<1}\mid {\modarn \widehat{B}}_1=x\right]
	&\le \varepsilon^{1/2-\beta} + 	\E_y\left[ \tau_0^{1/2-\beta} \Indi{\tau_0<\varepsilon} \mid {\modarn \widehat{B}}_1=x \right]
	\\ \nonumber
	&\le \varepsilon^{1/2-\beta} + 	\E_y\left[ \tau_0^{1/2-\beta} \Indi{\tau_0<\varepsilon} \frac{\gaussian_{(1-\tau_0)(1{\modr+}\eta)}(x)}{\gaussian_{1-\eta}(x-y)}  \right]
	\\ \label{E: maj tau half-beta_cond}
	&\le \varepsilon^{1/2-\beta} + 	{\modr C}\E_y\left[ \tau_0^{1/2-\beta} \Indi{\tau_0<\varepsilon}   \right] \frac{\gaussian_{(1{\modr+}\eta)}(x)}{\gaussian_{1{\modr+}\eta}(x-y)}
\end{align}
where in the last line we used $\gaussian_{(1-\tau_0)(1{\modr+}\eta)}(x)=\frac{1}{\sqrt{2\bm{\pi}(1-\tau_0)(1-\eta)}}e^{-\frac{x^2}{2(1-\tau_0)(1{\modr+}\eta)}}\le C \gaussian_{(1{\modr+}\eta)}(x) $ {\modarn  for $\tau_0 \le \varepsilon \le 1/2$}.
On the one hand, we compute using the explicit expression for the law of the first hitting time of a Brownian motion 
{\revar (e.g. see page 107 in \cite{RevYor_3rd_edition}),}
\begin{align*}
\E_y\left[ \tau_0^{1/2-\beta} \Indi{\tau_0<\varepsilon}   \right] =
\int_0^\varepsilon t^{1/2-\beta} \frac{|y|e^{-\frac{y^2}{2t(1-\eta)}}}{t^{3/2}}dt
=\int_{\frac{y^2}{2\varepsilon(1-\eta)}}^\infty |y|^{1-2\beta} u^{\beta-1} e^{-u} du{\modarn\times (1{\modr+}\eta)^\beta},
\end{align*}
where we have set {\modarn $u=\frac{y^2}{2t(1{\modr+}\eta)}$.} Using that $\beta>0$, we deduce
\begin{equation} \label{E: maj tau half-beta}
\E_y\left[\tau_0^{1/2-\beta} \Indi{\tau_0<\varepsilon}   \right]  \le C |y|^{1-2\beta}
e^{-\frac{y^2}{4\varepsilon(1{\modr+}\eta)}}.
\end{equation}

On the other hand, $\frac{g_{1-\eta}(x)}{g_{1-\eta}(x-y)}=e^{-\frac{[2xy-y^2]}{2(1-\eta)}}$.
Using $2|xy|\le 2\eta(1{\modr+}\eta)x^2+\frac{y^2}{2\eta(1{\modr+}\eta)}$ we have
$ \frac{g_{1{\modr+}\eta}(x)}{g_{1{\modr+}\eta}(x-y)} \le e^{\eta x^2 + y^2 [\frac{1}{4\eta(1{\modr+}\eta)^2}+{\modarn {\frac{1}{2(1{\modr+}\eta)}}}]}$. Choosing $\varepsilon$ such that $\frac{1}{(1{\modr+}\eta)4\varepsilon}>[\frac{1}{4\eta(1{\modr+}\eta)^2}+{\modarn {\frac{1}{2(1{\modr+}\eta)}}}]+\frac{1}{\eta}$, we deduce from \eqref{E: maj tau half-beta_cond} and \eqref{E: maj tau half-beta} that
 \begin{equation*}
 		\E_y \left[ \tau_0^{1/2-\beta} \Indi{\tau_0<1}\mid {\modarn \widehat{B}}_1=x\right]
 		\le \varepsilon^{1/2-\beta} + C |y|^{1-2\beta} e^{\eta x^2-\frac{1}{\eta}y^2}.
 \end{equation*}
Hence \eqref{E: moment LT int cond xy} is proved in the case $\beta>1/2$.

{\em Point 3.} {\modarn By translation it is sufficient to consider the case $z=0$.} The proof follows the scheme of the Point 2. Exactly as we obtain \eqref{E: maj TL cond by G}--\eqref{E: def G tau g}, we prove
\begin{equation}\label{E: maj TL cond by G tilde}
\E_y\left[
\int_{{\modarn h_n^2}}^{\Delta_n} \frac{dL_s({\modarn \widehat{B}})}{s} \mid {\modarn \widehat{B}}_{\Delta_n}=x
\right]	=
\E_y\left[
\tilde{G}(\tau_0,g_{\Delta_n}) \Indi{\tau_0<{\modarn \Delta_n}} \mid {\modarn \widehat{B}}_{\Delta_n}=x
\right],
\end{equation}
where $g_{\Delta_n}=\sup\{s\in[0,\Delta_n];~{\modarn \widehat{B}}_t=0\}$ and
\begin{equation*}
	\tilde{G}(\tau,g)=
	\E_0 \left[ \int_{(h^2-\tau)\vee0}^{g-\tau} \frac{dL^0({\modarn \widehat{B}})}{(s+\tau)}
	\mid {\modarn \widehat{B}}_{g-\tau}=0
	\right].
\end{equation*}
{\revar Using \eqref{E: moment LT int cond} with $\psi(s)=\Indi{\{s\ge h^2-\tau\}}\frac{1}{s+\tau}$,} we have
\begin{align*}
	\tilde{G}(\tau,g) &\le \frac{1}{\sqrt{\pi}}
	\int_{(h^2-\tau)\vee0}^{g-\tau} \frac{ds}{(s+\tau)\sqrt{s}}
	+ \frac{1}{\sqrt{\pi}}
	\int_{(h^2-\tau)\vee0}^{g-\tau}\frac{ds}{(s+\tau)\sqrt{g-\tau-s}}
	\\
	& \le \frac{1}{\sqrt{\pi} \sqrt{\tau}} \int_{(\frac{h^2}{\tau}-1) \vee 0}^\infty \frac{du}{(u+1)\sqrt{u}}
	+
	\int_0^{g-(\tau\vee h^2)} 
	\frac{ds'}{(s'+(\tau\vee h^2))\sqrt{g-(\tau \vee h^2)-s'}}
	\\
	&
	{\revar
\le \frac{1}{\sqrt{\pi} \sqrt{\tau}} \int_{(\frac{h^2}{\tau}-1) \vee 0}^\infty \frac{du}{(u+1)\sqrt{u}}
+
\int_0^{g-(\tau\vee h^2)} 
\frac{ds'}{2\sqrt{s'} \sqrt{\tau\vee h^2}\sqrt{g-(\tau \vee h^2)-s'}}
}
\end{align*}
where in the first integral we have set $u=\tau h$, and in the second integral we made the change of variable $s'=s+\tau-h^2$ in the case $h^2>\tau$, and
$s'=s$ if $h^2 \le \tau$ {\revar and then used $2\sqrt{s'} \sqrt{\tau\vee h^2} \le s'+\tau\vee h^2$.} After some {\revar computations, 
	using $\int_0^\gamma \frac{ds}{\sqrt{s}\sqrt{\gamma-s}}=\int_0^1 \frac{du}{\sqrt{u}\sqrt{1-u}}$ for any $\gamma>0$,}
	and discussing according to the relative positions of $h^2$ and $\tau$, {\revar 
we can} get that $\tilde{G}(\tau,g) \le C/h$.

Then, the point 4 of the lemma follows from \eqref{E: maj TL cond by G tilde}. 
\end{proof}
{ \begin{lemma}\label{L: majo func Local time cond Delta}
	Assume that $\beta \ge 0$, $\beta \neq 1/2$ {\revarn and $\widetilde{B}$ is a B.M.  with 
		 diffusion coefficient $\sqrt{1{\modr+}\eta}\in{\modr(1,\sqrt{2}]}$.}
	Then, there exists $C_\eta$ such that for $\forall (x,y,z)\in \mathbb{R}^3$, $y\neq z$
	\begin{equation}
		\E_y \left[ \int_0^{\Delta_n} \frac{d L^z_s({\modarn \widetilde{B}})}{s^\beta}  \mid {\modarn \widetilde{B}}_{\Delta_n}=x\right] 
		\le C_\eta \Delta_n^{1/2 - \beta} 
		\left(
		1+ \abs{\frac{y-z}{\sqrt{\Delta_n}}}^{(1-2\beta)\land 0}
		e^{-\frac{(y-z)^2}{\eta {\modarn \Delta_n}}}	e^{\eta\frac{(x-y)^2}{{\modarn \Delta_n}}}
		\right).
	\end{equation}
	Moreover, for $\beta \in [0,1/2)$ the inequality holds true also for $y=z$.
\end{lemma}
\begin{proof}
	If $({\modarn \widetilde{B}_{\revar u}})_{u \in [0,\Delta_n]}$ is a Brownian motion with 
	initial value $y$, we set
	${\modarn \widehat{B}_u= \Delta_n^{-1/2}{\modarn \widetilde{B}}_{u\Delta_n}}$ which is a Brownian motion 
	starting from $y\Delta_n^{-1/2}$. We have by change of variable,
	$$
	\int_0^{\Delta_n} \frac{dL^0_s({\modarn \widetilde{B}})}{s^\beta} = \Delta^{1/2-\beta}_n \int_0^1  \frac{dL^0_s({\modarn \widehat{B}})}{s^\beta}.
	$$
	We deduce
	$$
	\E_y\left[
	\int_0^{\Delta_n} 	\frac{d L^0_s({\modarn \widetilde{B}})}{s^\beta}  \mid {\modarn \widetilde{B}}_{\Delta_n}=x\right] 
	= \Delta_n^{1/2-\beta}	\E_{\frac{y}{\Delta_n^{1/2}}}\left[
	\int_0^1	\frac{d L^0_s({\modarn \widehat{B}})}{s^\beta}  \mid {\modarn \widehat{B}}_{1}=\frac{x}{\Delta_n^{1/2}}\right],
	$$
	and the {\modarn lemma} follows in the case $z=0$ by using \eqref{E: moment LT int cond xy} in Lemma \ref{L: divers majo TL}. The case $z \neq 0$ is obtained by translation.
\end{proof}
}

{\modrev
	\section{Malliavin calculus, some basic tools}
{	\revar
\label{S: Appendix Malliavin}	
}
	
	We have seen that the main idea to derive the lower bound in Theorem \ref{th: borne inf discrete} is to use Malliavin calculus technique to represent a score function in a treatable way. For this reason we introduce some basic facts on Malliavin calculus, needed for our computations. We refer to Nualart \cite{Nualart} for more details. \\
\subsection{Notations and basic properties}	 \label{s: Malliavin sub section recall}
	Consider a filtered probability space $(\Omega, \mathcal{F}, (\mathcal{F}_t), \mathbb{P})$, associated with a $1$- dimensional Brownian motion $\{ {\revarn B}_t, \, t\in [0, {\revar \Delta}] \}$ on a
	 finite interval $[0, {\revar \Delta}]$, {\revar where $0<\Delta\le 1$.} The underlying Hilbert space we consider is $H:= L^2([0, {\revar \Delta}], \R)$. For any $h \in H = L^2([0, {\revar \Delta}], \R)$ we denote as ${\revarn B}(h)$ the Ito integral $\int_0^{\revar \Delta} h(t) d{\revarn B}_t$. Moreover, we denote as {\revar $C^\infty_p(\R^m,\R)$} the set of all infinitely continuously differentiable functions $f: {\revar \R^m} \rightarrow \R$ such that $f$ and all its partial derivatives have polynomial growth. Let $\mathcal{S}$ denote the class of smooth random variables such that a random variable $F \in \mathcal{S}$ has the form 
	\begin{equation}{\label{eq: F in Malliavin}}
		\revar{ F = f({\revarn B}(h_1),\dots,{\revarn B}(h_m)),} 
	\end{equation}
	where $f$ belongs to {\revar $C^\infty_p(\R^m,\R)$} and $h_1,\dots,h_m \in H = L^2([0, {\revar \Delta}], \R)$ {\revar and $m\ge1$.} The derivative of a smooth random variable $F \in \mathcal{S}$ of the form \eqref{eq: F in Malliavin} is the $H$-valued random variable given by $DF = (D_t F)_{t \in [0, {\revar \Delta}]}$, 
	\begin{equation}
		\label{eq: Dt F in Malliavin} D_t F := {\revar \sum_{i=1}^m\frac{\partial f}{\partial x_{i}} ({\revarn B}(h_1),\dots,{\revarn B}(h_m)) h_i(t).}
	\end{equation}
	Roughly speaking, the scalar product $< DF, h>_H$ is the derivative at $\varepsilon$ of the random variable $F$ composed with the shifted process $\{ W(g) + \varepsilon \bracket{g}{h}  \, g \in H\}$. \\
	The following integration-by-parts formula holds true for any $F$ smooth random variable and $h \in H$: 
	$$\E[\bracket{DF}{h}] = \E[F {\revarn B}(h)].$$
	Moreover {\revarn by Chapter 1.2 in \cite{Nualart},} the operator $D$ is closable from $L^p(\Omega)$ to $L^p(\Omega, H)$ for any $p \ge 1$. We denote the domain of $D$ in $L^p(\Omega)$ by $\mathbb{D}^{1, p}(\R)$, which means that $\mathbb{D}^{1, p}(\R)$ is the closure of the class of smooth random variables $\mathcal{S}$ with respect to the norm 
	$$\left \| F \right \|_{\mathbb{D}^{1, p}(\R)} := (\E[|F|^p] + \E[ \left \| D F \right \|_H^p])^{\frac{1}{p}}.$$
	{\revar A crucial property is the chain rule formula for the Malliavin derivative (Proposition 1.2.3. in \cite{Nualart}). If $\varphi: \mathbb{R}\to \mathbb{R}$ is $\mathcal{C}^1$ with a bounded derivative and $F \in \mathbb{D}^{1,p}(\R)$, $p \ge 1$, then $\varphi(F) \in  \mathbb{D}^{1,p} (\R)$ and 
		\begin{equation}
			\label{eq: chain rule Malliavin}
			D_t \varphi(F)=\varphi'(F) D_t F.
		\end{equation}
	 It is possible to define higher order derivatives. The second order derivative $D^2_{s,t}F$ of the simple functional $F$ is obtained by differentiating the expression \eqref{eq: Dt F in Malliavin}, considering that $D_tF$ is a simple functional taking values in $H$, which yields to the expression
		\begin{equation*}
			D^2_{s,t} F := {\revar \sum_{1\le i,j\le m}\frac{\partial^2 f}{\partial x_{i}x_j} ({\revarn B}(h_1),\dots,{\revarn B}(h_m)) h_i(t)h_j(s).}
		\end{equation*}
		By iteration, the derivative $D^k_{s1,\dots,s_k} F$ is defined for a simple functional $F \in \mathcal{S}$ and the operator $D^k$ from $\mathcal{S}\subset L^p(\Omega) $ to $L^p(\Omega,H^{\otimes k})$ is closable and can be extended from $\mathcal{S}$ to $\mathbb{D}^{k,p}(\R)$ by closure under the norm defined by
		\begin{equation*}
			\norm{F}_{\mathbb{D}^{k,p}(\R)}^p:= \E[\abs{F}^p] + \sum_{j=1}^k \E[ \norm{D^j F}_{H^{\otimes j}}^p ]=\E[\abs{F}^p] + \sum_{j=1}^k \E \left[ \left( \int_{[0,{\revar \Delta}]^j} |D^j_{s_1,\dots,s_j}F|^2 ds_1\dots ds_j\right)^{p/2}\right].
		\end{equation*}
		The space of infinitely differentiable variables, in the Malliavin sense, is defined by $\mathbb{D}^\infty (\R)=\cap_{k\ge1}\cap_{p\ge1} \mathbb{D}^{k,p}(\R)$.}
	
	We can now introduce the divergence operator, defined as the adjoint of the derivative operator. As the underlying Hilbert space $H$ is an $L^2$ space we interpret the divergence operator as a stochastic integral and we call it Skorohod integral. Indeed, in the Brownian motion case, it coincides with the generalization of the Ito stochastic integral to anticipating integrands introduced by Skorohod \cite{Sko}.
	\begin{definition} \label{D: Skorohod} 
		The Skorohod integral $\delta$ is a linear operator on $L^2(\Omega, H)$ with values in $L^2(\Omega)$ such that
		\begin{enumerate}
			\item The domain of $\delta$, denoted by $Dom(\delta)$, is the set of $H$-valued square integrable random variables $u \in L^2(\Omega, H)$ such that
			$$|\E[\bracket{ DF}{ u }]| = |\E[\int_0^{\revar \Delta} D_tF u_t dt]| \le c \left \| F \right \|_2 $$
			for all $F \in \mathbb{D}^{1, 2}$; where $c$ is some constant depending on $u$. 
			\item If $u$ belongs to $Dom (\delta)$, then $\delta(u)$ is the element of $L^2(\Omega)$ characterized by
			\begin{equation}{\label{eq: m2}}
				\E[F \delta(u)] = \E[\bracket{ DF}{ u }] (= \E[\int_0^{\revar \Delta} D_tF  u_t dt])
			\end{equation}
			for any $F \in \mathbb{D}^{1, 2}(\R)$. 
		\end{enumerate}
	\end{definition}
	We observe that, taking $F= 1$ in \eqref{eq: m2}, we obtain $\E[\delta(u)]=0$ for any $u \in Dom(\delta)$. }	

	{\revar It is possible to extend the definition of the Malliavin derivative for $H$-valued variables $u$, in a way analogous to the case of real valued random variables $F$. This enables to define a set $\mathbb{D}^{k,p}(H)\subset L^{p}(\Omega,H)$ of $H$-valued variables admitting $k$ derivatives in $L^p$, endowed with the norm
		\begin{align*}
			\norm{u}_{\mathbb{D}^{k,p}(H)}^p&= \E[\norm{u}_{H}^p] + \sum_{j=1}^k \E[ \norm{\norm{D^j u}_{H}}_{H^{\otimes j}}^p ]
			=\E[\norm{u}_H] + \sum_{j=1}^k \E \left[ \left( \int_{[0,{\revar \Delta}]^{j+1}} |D^j_{s_1,\dots,s_j}u_r|^2 drds_1\dots ds_j\right)^{p/2}\right]
			\\
			&=\norm{u}_{L^p(\Omega,H)}^p + \sum_{j=1}^k \norm{D^j u}_{L^p(\Omega,H^{\otimes (j+1)})}^p.
		\end{align*}
		The following proposition gives a smoothness criteria on $u$ which ensures that $u \in Dom(\delta)$, and provides useful formulae.
	}  
{\modrev
	\begin{proposition}{\label{prop: properties Skorohod operator}}
		\begin{enumerate}
			\item For any $p > 1$ the space of the {\revar Malliavin differentiable variables $\mathbb{D}^{1, p}(H)$ is included in} $Dom(\delta)$ and it is 
			$$\left \| \delta(u) \right \|_p 
			{\revar	\le c_p \norm{u}_{\mathbb{D}^{1,p}(H)}}.$$
			{\revar More generally, the operator $\delta$ is bounded from $ \mathbb{D}^{k, p}(H) $ to  $\mathbb{D}^{k-1, p}(\R)$, for $k\ge 1$, $p>1$.}
			\item If $u$ is an adapted process belonging to $L^2([0, {\revar \Delta}] \times \Omega, \R)$, then the Skorohod integral coincides with the Ito integral, i.e.
			$$\delta(u) = \int_0^\Delta u_t d{\revarn B}_t.$$
			\item If $F$ belongs to $\mathbb{D}^{1, 2}$ then for any $u \in Dom(\delta)$ for which $\E[F^2 \int_0^{\revar \Delta} u_t^2 dt] < \infty$, it is 
			$$\delta(Fu) = F \delta(u) - \int_0^{\revar \Delta} D_tF u_t dt$$
			whenever the right hand side belongs to $L^2(\Omega)$. 
		\end{enumerate}
	\end{proposition}
{\revarn The first point in the proposition above can be found in Proposition 1.5.7 of \cite{Nualart}, the second point in Proposition 1.3.11 of \cite{Nualart}, and the third point in Proposition 1.3.3. \cite{Nualart}.}

	{\revar We recall that solutions of {\modr SDE} with smooth coefficients are smooth variables in the Malliavin sense. Assume that $X$ is solution to the {\modr SDE}
		\eqref{eq: model}, that the coefficients $a$, $b$, are $\mathcal{C}^\infty$ and that $a^{(k)}$, $b^{(k)}$ are bounded functions for any $k \ge 1$. Then, by Theorem 2.2.2 in \cite{Nualart}, $X_t\in\mathbb{D}^\infty(\R)=\cap_{p\ge 1, k \ge 1} \mathbb{D}^{k,p}(\R)$ and we have
		\begin{equation} \label{eq: control Malliavin SDE}
			\sup_{(r_1,\dots,r_k) \in [0,{\revar \Delta}]^k}\E\left[
			\sup_{0\le t \le {\revar \Delta}} \abs{ D^k_{r_1,\dots,r_k}(X_t) }^p \right] \le c_{p,\Delta,k}
		\end{equation}
		where the constant $c_{p,{\revar \Delta},k}$, for bounded $\Delta$, is upper bounded by a constant $c_{p,k}$ that only depends on $p$, $\abs{a(0)}$, $\abs{b(0)}$, $\sup_{l=1,\dots,k} \norm{a^{(l)}}_\infty$, and $\sup_{l=1,\dots,k} \norm{b^{(l)}}_\infty$. The control \eqref{eq: control Malliavin SDE} is the property (P1) in the proof of  Theorem 2.2.2 in \cite{Nualart}, where the dependence of $c_{p,T,k}$ on the first $k$ derivatives of the coefficient is a consequence of the expression (P2) in the proof of Theorem 2.2.2. in \cite{Nualart}.
	}
}
{\revar 
\subsection{Proof of formula \eqref{eq : Score by Malliavin}} \label{Ss: proof of score by Malliavin}
We follow closely the proof of Theorem 5 in \cite{Glo_Gob}, in our simpler one-dimensional situation, and relying on the properties recalled in Section \ref{s: Malliavin sub section recall}. As the coefficients of the {\modr SDE} \eqref{eq: model epsilon} are $\mathcal{C}^\infty$ we know that $(x,y,\epsilon)\mapsto p_{\Delta_n}^\epsilon(x,y)$ is smooth.
 For $\varphi : \mathbb{R} \to \mathbb{R}$ a $\mathcal{C}^\infty$ bounded function, we differentiate with respect to $\epsilon$ the relation 
$	\E_{x_0}[\varphi(X^\epsilon_{\Delta_n})]= \int_\mathbb{R} p^\epsilon_{\Delta_n}(x_0,y)\varphi(y) dy$. It yields to
$\E_{x_0}[\varphi'(X^\epsilon_{\Delta_n})\dot{X}_{\Delta_n}^\epsilon]= \int_\mathbb{R} \dot{p}^\epsilon_{\Delta_n}(x_0,y)\varphi(y) dy$. We now use the chain rule formula \eqref{eq: chain rule Malliavin} to write $\varphi'(X^\epsilon_{\Delta_n})=\frac{\bracket{D_\cdot(\varphi(X^\epsilon_{\Delta_n}))}{D_\cdot X^\epsilon_{\Delta_n}}}{\bracket{X^\epsilon_{\Delta_n}}{X^\epsilon_{\Delta_n}}}$ and thus
\begin{equation*}
\E_{x_0}[\varphi'(X^\epsilon_{\Delta_n})\dot{X}_{\Delta_n}^\epsilon]=
\E_{x_0}\left[\frac{\bracket{D_\cdot(\varphi(X^\epsilon_{\Delta_n}))}{D_\cdot X^\epsilon_{\Delta_n}}}{\bracket{X^\epsilon_{\Delta_n}}{X^\epsilon_{\Delta_n}}}
\dot{X}_{\Delta_n}^\epsilon\right]	
=\E_{x_0}\left[
\bracket{ D_\cdot(\varphi(X^\epsilon_{\Delta_n})) }
{ \frac{D_\cdot X^\epsilon_{\Delta_n}\dot{X}_{\Delta_n}^\epsilon}{\bracket{D_\cdot X^\epsilon_{\Delta_n}}{D_\cdot X^\epsilon_{\Delta_n}}}}\right].
\end{equation*}
We now apply \eqref{eq: m2}, with $u_\cdot= \frac{D_\cdot X^\epsilon_{\Delta_n} \dot{X}^\epsilon_{\Delta_n} }{\bracket{DX^\epsilon_\cdot}{DX^\epsilon_\cdot}}$, and deduce $\E_{x_0}[\varphi'(X^\epsilon_{\Delta_n})\dot{X}_{\Delta_n}^\epsilon]=\E_{x_0}[\varphi(X^\epsilon_{\Delta_n}) W_{x_0,\Delta_n,\epsilon}]$ with $W_{x_0,\Delta_n,\epsilon}$ given by \eqref{eq: Malliavin weight}. This is possible, as we show in the proof of Lemma \ref{L: maj grossiere Malliavin} that $u \in \mathbb{D}^{1,4}(H)$ which implies $u\in Dom(\delta)$. Now, by conditioning on $X^\varepsilon_{\Delta_n}$, we get
\begin{equation*}
\int_\mathbb{R} \dot{p}^\epsilon_{\Delta_n}(x_0,y)\varphi(y) dy=\E_{x_0}[\varphi(X^\epsilon_{\Delta_n}) W_{x_0,\Delta_n,\epsilon}]=\int_\mathbb{R} p^\epsilon_{\Delta_n}(x_0,y)\E_{x_0}[W_{x_0,\Delta_n,\epsilon} \mid X_{\Delta_n}^\epsilon=y]\varphi(y)dy .
\end{equation*}
The equation \eqref{eq : Score by Malliavin} follows.
\qed
}

\end{appendix}

\end{document}